\documentclass[reqno, huge]{amsart}
\usepackage{comment}
\usepackage[dvipsnames, table]{xcolor}
\usepackage{amsthm}
\theoremstyle{plain}
\usepackage{amssymb}
\usepackage{marvosym}
\usepackage{bm}
\usepackage{mathrsfs}
\usepackage{enumerate}  
\usepackage{mathtools}
\usepackage{tikz-cd}
\usepackage{graphicx}
\usepackage{float}
\usepackage[titletoc,title]{appendix}
\usepackage{marginnote}
\usepackage{todonotes}
\usepackage{etoolbox}
\usepackage[all]{xy}
\makeatletter
\patchcmd{\Ginclude@eps}{"#1"}{#1}{}{}
\makeatother
\usepackage[outdir=./]{epstopdf}
\usepackage[numbers]{natbib}
\usepackage[utf8]{inputenc}
\usepackage[english]{babel}
\usepackage{changepage}
\usepackage{makecell}
\usepackage{enumitem}
\usepackage{comment}

\usepackage[colorlinks,citecolor=blue]{hyperref}

\usepackage{tikz}
\usetikzlibrary{calc,trees,positioning,arrows,chains,shapes.geometric,%
    decorations.pathreplacing,decorations.pathmorphing,shapes,%
    matrix,shapes.symbols}

\tikzset{
>=stealth',
  punktchain/.style={
    rectangle,
    rounded corners,
    draw=black, thick,
    minimum height=3em,
    text centered,
    on chain},
  line/.style={draw, thick, <-},
  element/.style={
    tape,
    top color=white,
    bottom color=blue!50!black!60!,
    minimum width=8em,
    draw=blue!40!black!90, very thick,
    text width=10em,
    minimum height=3.5em,
    text centered,
    on chain},
  every join/.style={->, thick,shorten >=1pt},
  decoration={brace},
  tuborg/.style={decorate},
  tubnode/.style={midway, right=2pt},
}

\definecolor{lightblue}{HTML}{1F88CD}
\definecolor{lightgrey}{HTML}{727272}
\definecolor{lightblue2}{HTML}{009EC1}
\definecolor{mypink}{HTML}{FD00B0}
\definecolor{lightred}{HTML}{ff4d4d}

\usepackage{hyperref}
\hypersetup{
	colorlinks=true,
    linkcolor={OliveGreen},
    citecolor={lightred},
	urlcolor={black}
}

\newtheorem*{theorem*}{Theorem}
\newtheorem{theorem}{Theorem}[section]
\newtheorem{corollary}[theorem]{Corollary}
\newtheorem{lemma}[theorem]{Lemma}
\newtheorem{conjecture}[theorem]{Conjecture}

\newtheorem{proposition}[theorem]{Proposition}
\theoremstyle{definition}
\newtheorem{example}[theorem]{Example}
\theoremstyle{definition}
\newtheorem{definition}[theorem]{Definition}
\theoremstyle{definition}
\newtheorem{remark}[theorem]{Remark}
\theoremstyle{definition}

\theoremstyle{definition}

\theoremstyle{definition}

\theoremstyle{definition}

\theoremstyle{definition}

\theoremstyle{definition}
\newtheorem{question!}[theorem]{Question!}
\theoremstyle{definition}

\makeatletter
\newcommand*\sbt{\mathpalette\sbt@{.75}}
\newcommand*\sbt@[2]{\mathbin{\vcenter{\hbox{\scalebox{#2}{$\m@th#1\bullet$}}}}}
\makeatother

\newcommand{\xra}{\xrightarrow}

\newcommand{\sst}{\subset}

\let\emptyset\varnothing

\newcommand{\bR}{\bm{\mathrm{R}}}
\newcommand{\bL}{\bm{\mathrm{L}}}
\newcommand{\bO}{\bm{\mathrm{O}}}

\newcommand{\D}{\mathrm{D}}

\newcommand{\R}{\mathbb{R}}
\newcommand{\ZZ}{\mathbb{Z}}

\newcommand{\CC}{\mathbb{C}}
\newcommand{\PP}{\mathbb{P}}

\newcommand{\ch}{\mathrm{ch}}

\newcommand{\pr}{\mathrm{pr}}
\newcommand{\ev}{\mathrm{ev}}
\newcommand{\cok}{\mathrm{coker}}
\newcommand{\C}{\mathbb{C}}

\renewcommand{\Re}{\operatorname{Re}}
\renewcommand{\Im}{\operatorname{Im}}

\DeclareMathOperator{\Aut}{Aut}

\DeclareMathOperator{\identity}{id}
\DeclareMathOperator{\im}{im}

\DeclareMathOperator{\rk}{rk}

\DeclareMathOperator{\Coh}{\mathrm{Coh}}

\DeclareMathOperator{\Ext}{Ext}

\DeclareMathOperator{\Hom}{Hom}
\DeclareMathOperator{\RHom}{RHom}

\DeclareMathOperator{\ext}{ext}

\DeclareMathOperator{\Pic}{Pic}

\DeclareMathOperator{\cone}{cone}
\DeclareMathOperator{\Stab}{Stab}
\DeclareMathOperator{\Gr}{Gr}

\newcommand{\cC}{\mathcal{C}}
\newcommand{\cA}{\mathcal{A}}
\newcommand{\cE}{\mathcal{E}}

\newcommand{\cH}{\mathcal{H}}

\newcommand{\cK}{\mathcal{K}}
\newcommand{\cI}{\mathcal{I}}
\newcommand{\cJ}{\mathcal{J}}
\newcommand{\cT}{\mathcal{T}}
\newcommand{\cQ}{\mathcal{Q}}
\newcommand{\Ku}{\mathcal{K}u}
\newcommand{\cP}{\mathcal{P}}
\newcommand{\cD}{\mathcal{D}}

\newcommand{\cN}{\mathcal{N}}
\newcommand{\cM}{\mathcal{M}}

\DeclareMathOperator{\oh}{\mathcal{O}}
\DeclareMathOperator{\QY}{\mathcal{Q}_Y}
\newcommand{\bv}{\mathbf{v}}
\newcommand{\bw}{\mathbf{w}}
\newcommand{\bs}{\mathbf{s}}
\newcommand{\bt}{\mathbf{t}}

\newcommand{\sol}[1]{\textcolor{red}{#1}}
\newcommand{\zy}[1]{\textcolor{blue}{#1}}
\newcommand{\sz}[1]{\textcolor{green}{#1}}

\usetikzlibrary{decorations.pathmorphing}
\begin{document}

\title[Categorical Torelli theorems for del Pezzo threefolds]{
new perspectives on categorical Torelli theorems \\ for del Pezzo threefolds}


\author{Soheyla Feyzbakhsh, Zhiyu Liu and Shizhuo Zhang}

\address{}
\email{}

\address{Department of Mathematics, Imperial College, London SW7 2AZ, United Kingdom}
\email{s.feyzbakhsh@imperial.ac.uk}

\address{Institute for Advanced Study in Mathematics, Zhejiang University, Hangzhou, Zhejiang Province 310030, P. R. China}
\email{jasonlzy0617@gmail.com}

\address{Max Planck Institute for Mathematics, Vivatsgasse 7, 53111 Bonn, Germany}
\address{Institut de Mathématiqes de Toulouse, UMR 5219, Université de Toulouse, Université Paul Sabatier, 118 route de
Narbonne, 31062 Toulouse Cedex 9, France}
\email{shizhuozhang@mpim-bonn.mpg.de,shizhuo.zhang@math.univ-toulouse.fr}

\begin{abstract}

Let $Y_d$ be a del Pezzo threefold of Picard rank one and degree $d\geq 2$. In this paper, we apply two different viewpoints to study $Y_d$ via a particular admissible subcategory of its bounded derived category, called the Kuznetsov component:  

(i) Brill--Noether reconstruction. We show that $Y_d$ can be uniquely recovered as a Brill--Noether locus of Bridgeland stable objects in its Kuznetsov component.

(ii) Exact equivalences. We prove that, up to composing with an explicit auto-equivalence, any Fourier--Mukai type equivalence of Kuznetsov components of two del Pezzo threefolds of degree $2\leq d\leq 4$ can be lifted to an equivalence of their bounded derived categories. As a result, we obtain a complete description of the group of Fourier--Mukai type auto-equivalences of the Kuznetsov component of $Y_d$.   

We also describe the group of Fourier--Mukai type auto-equivalences of Kuznetsov components of index one prime Fano threefolds $X_{2g-2}$ of genus $g=6$ and $8$. As an application, first we identify the group of automorphisms of $X_{14}$ and its associated $Y_3$. Then we give a new disproof of Kuznetsov's Fano threefold conjecture by assuming Gushel--Mukai threefolds are general.

In an appendix, we classify instanton sheaves on quartic double solids, generalizing a result of Druel. 
\end{abstract}



\maketitle


\setcounter{tocdepth}{1}
\tableofcontents







\section{Introduction}
Let $Y$ be a del Pezzo threefold of Picard rank one, which is an index two prime Fano threefold. By \cite{iskovskih:Fano-threefolds-I}, it belongs to one of the five families of threefolds classified by their degree $1\leq d\leq 5$, see Section~\ref{survey_weak_stabilitycon}. By a series of papers of Bondal--Orlov and Kuznetsov, the bounded derived category $\D^b(Y)$ of these Fano threefolds admit a semiorthogonal decomposition 
$$\D^b(Y)=\langle\Ku(Y),\oh_Y,\oh_Y(1)\rangle=\langle\Ku(Y),\cQ_Y,\oh_Y\rangle,$$
where $\cQ_Y\cong\bL_{\oh_Y}\oh_Y(1)[-1]$ is a rank $d+1$ vector bundle for $d\geq 2$.
This paper aims to employ two different viewpoints to 
extract the critical information of $Y$ from its admissible subcategory $\Ku(Y)$, called the Kuznetsov component. Before that, we give the following theorem which is crucial in both directions. 
Recall that the \emph{rotation functor} $\bO$ is an auto-equivalence of $\Ku(Y)$ sending $E\in\Ku(Y)$ to $\bL_{\oh_Y}(E\otimes\oh_Y(H))$. We denote by $i \colon \Ku(Y) \hookrightarrow \D^b(Y)$ the inclusion functor with the right and left adjoints $i^!$ and $i^*$, respectively.

In the following, we consider the object $i^!\cQ_Y\in\Ku(Y)$, which is the \emph{gluing object} of the semiorthogonal decomposition in the sense of \cite{kuz:cat-resolution-irrational}. As explained in \cite[Section 2.2]{kuz:cat-resolution-irrational}, the \emph{gluing object} together with $\Ku(Y)$ encode the information of $\oh_Y^{\perp}$. We first show that this \emph{gluing object} $i^!\cQ_Y$ is preserved by any exact equivalence between Kuznetsov components of $Y$ and $Y'$, up to some natural auto-equivalences.

\begin{theorem}[{Theorem \ref{prop_gluing_data_fixed}}]\label{gluing_objects_fixed}
 Let $Y$ and $Y'$ be del Pezzo threefolds of Picard rank one and degree $2\leq d\leq 4$, and $\Phi \colon \Ku(Y) \xrightarrow{\simeq}\Ku(Y')$ be an exact equivalence. 
\begin{enumerate}
    \item [(i)] If $2\leq d \leq 3$, there exist a unique pair of integers $m_1, m_2 \in \ZZ$ with $0\leq m_1\leq 3$ when $d=2$ and $0\leq m_1\leq 5$ when $d=3$, so that    $$\Phi(i^!\QY)\cong\bO^{m_1}({i'}^!\mathcal{Q}_{Y'})[m_2].$$
    \item [(ii)] If $d=4$, there exists a unique pair of integers $m_1,m_2$ and a unique auto-equivalence $T_{\mathcal{L}_0} \in \mathrm{Aut}^0(\Ku(Y'))$ (see Section \ref{degree 4-subsection} for definition) so that 
 $$\Phi(i^!\cQ_Y)\cong \bO^{m_1}\circ T_{\mathcal{L}_0}({i'}^!\mathcal{Q}_{Y'})[m_2]. $$
\end{enumerate}
 Here $i' \colon \Ku(Y') \hookrightarrow \D^b(Y')$ is the inclusion functor.
\end{theorem}

 To prove degree $2\leq d\leq 3$ cases, we identify the object $i^!\cQ_Y$ via the uniqueness property\footnote{The object $i^!\cQ_Y$ can be viewed as a generalization of the classical notion of second Rayanud bundle over a genus two curve, which is unique up to twisting a line bundle over the curve.}
 of it. Up to rotations and shifts, we can assume any exact equivalence $\Phi\colon \Ku(Y) \xrightarrow{\simeq}\Ku(Y')$ acts trivially on the numerical Grothendieck group. 
Take a stable object $E$ in $\Ku(Y)$ of the same class as $i^!\cQ_Y$, then we show $\mathrm{RHom}(i^*\oh_p, E)$ is a two-term complex for all points $p \in Y$ if and only if $E \cong i^!\cQ_Y$. Combining it with analysis of the moduli space of stable objects in $\Ku(Y)$ of class $[i^*\oh_p]$ gives Theorem \ref{gluing_objects_fixed}. For $d=4$ case, we use the property of \emph{second Raynaud bundles}.  

\bigskip

Then we discuss our two perspectives on categorical Torelli theorems. 
\subsection*{I. Brill--Noether reconstruction} \label{sub_section_BN}
In \cite{feyz:desing,rota:moduli-space-index-two}, authors apply stability conditions on $\Ku(Y)$ for degree $2\leq d\leq 3$ to show that one can uniquely recover $Y$ as a subscheme of a moduli space of stable objects in $\Ku(Y)$. 
The following theorem shows that we can describe this subscheme explicitly as a Brill--Noether locus. This generalises the classical picture for degree $d=4$, as discussed in Section~\ref{classical_d=4}.

\vspace{.2 cm}

By \cite{pertusi:some-remarks-fano-threefolds-index-two}, \cite{feyzbakhsh2021serre} and \cite{jacovskis2021categorical}, there is a unique Serre-invariant stability condition on $\Ku(Y)$ up to the action of $\widetilde{\mathrm{GL}}^+_2(\mathbb{R})$ for $d\geq 2$, see Section~\ref{survey_weak_stabilitycon}. Denote by $\mathcal{M}_{\sigma}(\Ku(Y),v)$ the moduli space\footnote{Let $\sigma = (Z, \mathcal{A})$, then up to a shift we may assume $\Im[Z(v)] \geq 0$, then we only consider stable objects in the heart $\cA$ to define the moduli space $\cM_{\sigma}(\Ku(Y), v)$} of stable objects of a numerical class $v\in\mathcal{N}(\Ku(Y))$ with respect to a stability condition $\sigma$ on the Kuznetsov component $\Ku(Y)$. 

\begin{theorem}[{Theorem \ref{Brill--Noether reconstruction}}]
\label{main_theorem_BN_reconstruction_del_Pezzo_3fold}
Let $Y$ be a del Pezzo threefold of Picard rank one and degree $d\geq 2$, and let $\sigma$ be a Serre-invariant stability condition on $\Ku(Y)$. Then $Y$ is isomorphic to the \emph{Brill--Noether locus}\footnote{Note that for any $F \in \mathcal{M}_{\sigma}(\mathcal{K}u(Y),\ d\bv -\bw 
)$, we prove $\RHom(F, i^{!}\QY) = \CC^{\delta}[k+1] \oplus \CC^{d+\delta}[k]$ where $\delta$ is either zero or one. Hence there exists at most one $k \in \ZZ$ so that $\dim_{\mathbb{C}}\Hom(F[k],\ i^{!}\QY) \geq d+1$. } 
$$\mathcal{BN}_Y\coloneqq \{F\in\mathcal{M}_{\sigma}(\mathcal{K}u(Y), \ [i^*\oh_p]
) \, \colon\ \exists k \in \mathbb{Z} \ \text{such that} \ \dim_{\mathbb{C}}\Hom(F[k],\ i^{!}\cQ_Y) \geq d+1
\}.
$$
where $\oh_p$ is the skyscraper sheaf supported at a point $p\in Y$. 
\end{theorem}

\vspace{.2 cm}

By \cite{pertusi:some-remarks-fano-threefolds-index-two}, Serre-invariant stability conditions on $\Ku(Y)$ for degree $d \geq 2$ are $\bO$-invariant as well. Thus combining Theorem \ref{main_theorem_BN_reconstruction_del_Pezzo_3fold} and \ref{gluing_objects_fixed} gives a new proof for \emph{Categorical Torelli Theorem}.

\begin{corollary}[{Corollary \ref{coro_Torelli_from_BN}}]\label{main_theorem_categorical_Torelli_theorem}
Let $Y$ and $Y'$ be del Pezzo threefolds of Picard rank one and degree $2\leq d\leq 4$ such that $\Ku(Y)\simeq\Ku(Y')$, then $Y\cong Y'$.
\end{corollary}

\subsection*{II. Exact equivalences.} The second viewpoint is to combine the categorical techniques developed in \cite{li2021refined} with geometric analysis of stable objects in $\Ku(Y)$ to show that any Fourier--Mukai type exact equivalence of Kuznetsov components of two del Pezzo threefolds of degree $2\leq d\leq 4$ 
can be lifted to an equivalence of their bounded derived categories.

\begin{theorem}[{Theorem \ref{prop_gluing_data_fixed}}] \label{thm-FM}
 Let $Y$ and $Y'$ be del Pezzo threefolds of Picard rank one and degree $2\leq d\leq 4$, and let $\Phi \colon \Ku(Y) \to \Ku(Y')$ be an exact equivalence of Fourier--Mukai type such that  $\Phi(i^{!}\cQ_Y) = i'^{!}\cQ_{Y'}$. Then $\Phi=f_*|_{\Ku(Y)}$ for a unique isomorphism $f \colon Y\to Y'$. 
\end{theorem}

Clearly, combining Theorem \ref{gluing_objects_fixed} with Theorem \ref{thm-FM} provides an alternative proof of \emph{Categorical Torelli theorem} for del Pezzo threefold of degree $2\leq d\leq 4$. Furthermore, we obtain a complete description of the group $\mathrm{Aut}_{\mathrm{FM}}(\Ku(Y))$ of exact auto-equivalences of $\Ku(Y)$ of Fourier--Mukai type. For a group $G$ and a subset $S\subset G$, we denote by $\langle S \rangle$ the subgroup of $G$ generated by $S$.

\begin{corollary}[{Corollary \ref{cor_auto_equi_cubic}}]\label{cor_auto_equi_cubic-introdcution}
If $Y$ is a del Pezzo threefolds of Picard rank one and degree $d$. Then we have\footnote{By \cite[Theorem 1.3]{li2022derived}, any exact equivalence between Kuznetsov components of quartic double solids is of Fourier--Mukai type. The same also holds for del Pezzo threefolds of degree $d=4$ as $\Ku(Y)\simeq \D^b(C)$ for a smooth curve $C$.}

\begin{enumerate}
    \item $\mathrm{Aut}_{\mathrm{FM}}(\Ku(Y))=  \langle \mathrm{Aut}(Y), \bO, [1] \rangle$ when $2\leq d\leq 3$, and
    
    \item $\mathrm{Aut}_{\mathrm{FM}}(\Ku(Y))=  \langle \mathrm{Aut}(Y), \Aut^0(\Ku(Y)),  \bO, [1] \rangle$ when $d=4$.
    
\end{enumerate}
Here the subgroup $\mathrm{Aut}^0(\Ku(Y))$ is  defined in Section \ref{degree 4-subsection}. 
\end{corollary}

We may write elements of $\Aut_{\mathrm{FM}}(\Ku(Y))$ in a more explicit way, see Corollary \ref{cor_auto_equi_cubic}. 

\bigskip

Next, let $X$ be an index one prime Fano threefold of genus $g\geq 6$. Then we have a semi-orthogonal decomposition
\[\D^b(X)=\langle \Ku(X), \cE_X, \oh_X \rangle,\]
where $\cE_X$ is the tautological subbundle obtained by pulling back the one on correspondent Grassmannian. When the genus $g$ of $X$ is either $7,9,10$ or $12$, its Kuznetsov component $\Ku(X)$ is either equivalent to the derived category of a smooth projective curve or the derived category of representations of the 3-Kronecker quiver, whose group of auto-equivalences is known (see \cite{miyachi2001derived} for $g=12$ case). So in the following, we focus on $g=6$ and $8$ cases.

As is shown in the del Pezzo threefold case, the Brill--Noether reconstruction plays an important role in computing the group of auto-equivalences of their Kuznetsov components. For the case of index one prime Fano threefolds of genus $g\geq 6$, the Brill--Noether reconstruction is established in \cite{jacovskis2022brill}. Thus we have a complete description on $\Aut_{\mathrm{FM}}(\Ku(X))$.

\begin{theorem}[{Corollary \ref{cor_aut_genus_8}, Corollary \ref{cor_aut_genus_6}}] \label{thm_1.6}
Let $X$ be an index one prime Fano threefold of genus $g=6$ or $8$. When $g=6$, we further assume that $X$ is general. Then we have
\[\Aut_{\mathrm{FM}}(\Ku(X))= \langle \Aut(X), S_{\Ku(X)}, [1] \rangle.\]
\end{theorem}

Theorem \ref{thm_1.6} has two nice applications. First, we identify groups of automorphisms of two different prime Fano threefolds, generalizing a classical result in \cite[Corollary 4.2.3 (1), Corollary 4.3.5 (2)]{kuznetsov2018hilbert}.

\begin{corollary}[{Corollary \ref{cor_aut_variety}}] \label{intro-cor-1.7}
Let $X$ be an index one prime Fano threefold of genus $g=8$ and $Y$ be the Phaffian cubic threefold associated with $X$. Then 
    \[\Aut(X)\cong\Aut(Y).\]
\end{corollary}

The second application is on Kuznetsov's Fano threefold conjecture \cite[Conjecture 3.7]{kuznetsov:derived-category-fano-threefold}. It was disproved in \cite{bayer2022kuznetsov} and \cite{zhang:kuznetsov-conjecture}. Using Corollary \ref{cor_auto_equi_cubic-introdcution} and Theorem \ref{thm_1.6}, we compare groups of auto-equivalences of Kuznetsov components of Gushel--Mukai threefolds and quartic double solids. As a result, we give a simple disproof of \cite[Conjecture 3.7]{kuznetsov:derived-category-fano-threefold} by assuming Gushel--Mukai threefolds are general.

\begin{corollary}[{Corollary \ref{cor_ku_conj}}] \label{intro-cor-1.8}
    Let $X$ be a general Gushel--Mukai threefold and $Y$ a quartic double solid. Then $\Ku(X)$ is not equivalent to $\Ku(Y)$. 
\end{corollary}

\subsection*{Related work}
Here is the list of relevant results for del Pezzo threefolds $Y_d$ of degree $d$:

\begin{enumerate}
    \item [$d=2$.]  In \cite{bernardara:from-semi-orthogonal-decomposition} and \cite{rota:moduli-space-index-two}, the categorical Torelli theorem (Corollary \ref{main_theorem_categorical_Torelli_theorem}) has been proved for \emph{generic} quartic double solids. It has been proved for non-generic cases in \cite{bayer2022kuznetsov} 
    via Hodge theory for K3 categories. In Theorem \ref{main_theorem_BN_reconstruction_del_Pezzo_3fold},  we give an explicit expression for $Y$ as a Brill--Noether locus of stable objects in $\Ku(Y_2)$, and so provide a new proof for the categorical Torelli theorem.  

    \item [$d=3$.] In \cite{macri:categorical-invarinat-cubic-threefolds} and \cite{pertusi:some-remarks-fano-threefolds-index-two}, the categorical Torelli theorem has been proved for cubic threefolds by reducing it to classical Torelli theorem. In \cite{liu2023autoeq}, the author computes the  group of auto-equivalences of Kuznetsov components of cubic threefolds of Fourier--Mukai type via a completely different method and provides a new proof of categorical Torelli theorem for cubic threefold by constructing a Hodge isometry between cubic threefolds. In \cite{feyz:desing}, the cubic threefold $Y_3$ has been described geometrically as a sub-locus of a moduli space of stable objects in $\Ku(Y_3)$. Theorem \ref{main_theorem_BN_reconstruction_del_Pezzo_3fold} gives a point-wise description of it as a Brill--Noether locus. 
    \item [$d=4$.] We know $Y_4$ is the intersection of two quadrics in $\mathbb{P}^5$, and by \cite{newstead1968stable}, it can be reconstructed as the moduli space $M$ of stable vector bundles of rank two with fixed determinant of an odd degree over the associated genus two curve $C_2$. We have $\Ku(Y_4) \simeq \D^b(C_2)$. As discussed in Section \ref{classical_d=4}, our categorical Brill--Noether locus in Theorem \ref{main_theorem_BN_reconstruction_del_Pezzo_3fold} matches with the classical moduli space $M$. 
\end{enumerate}
Other than del Pezzo threefolds, various versions of categorical Torelli theorems are also obtained, see \cite{pertusi2022categorical} for recent development. In particular, in \cite{jacovskis2022brill} the authors provide a Brill--Noether reconstruction for index one prime Fano threefolds, and as a result, the refined categorical Torelli theorem is proved. 

In \cite{druel:instanton-sheaves-cubic-threefold, qin2021compactification, qin:instanton-sheaves-degree-four, liu-zhang:moduli-space-cubic-x14}, a classification of rank two instanton sheaves and the corresponding moduli space in the Kuznetsov component have been discussed for del Pezzo threefolds of degree $d\geq 3$. In Appendix \ref{section_classification_instanton_sheaves_d=2}, we discuss degree $d= 2$ case. 

\subsection*{Organization of the article}
In Section~\ref{survey_weak_stabilitycon}, we recall the basic definitions and properties of (weak) stability conditions on del Pezzo threefolds of Picard rank one $Y_d$ of degree $d$ and their Kuznetsov components $\Ku(Y_d)$. In particular, we introduce Serre-invariant stability conditions on $\Ku(Y_d)$ and describe $\Ku(Y_d)$ for each $d\geq 2$. In Section~\ref{results_general_Wallcross}, we collect results of general wall-crossing for del Pezzo threefolds which will be used in later sections. In Section~\ref{section_Y2}, we describe the moduli space of $\sigma$-stable objects of the same class as twice of ideal sheaves of lines in the Kuznetsov component of a quartic double solid. In Section~\ref{section_Y3} we classify $\sigma$-stable objects of the same class as three times of ideal sheaves of line in the Kuznetsov component of a cubic threefold. In Section~\ref{Brill_Noether_reconstruction_dP} we prove Theorem~\ref{gluing_objects_fixed}. 
In Section~\ref{section_categorical_Torelli_theorem} we provide a \emph{Brill--Noether reconstruction} for a del Pezzo threefold of Picard rank one $Y_d$ with respect to $\Ku(Y_d)$ and its gluing object $i^!\cQ_{Y_d}$, proving Theorem~\ref{main_theorem_BN_reconstruction_del_Pezzo_3fold}. Then we prove \emph{categorical Torelli theorem}~\ref{main_theorem_categorical_Torelli_theorem}. In Section~\ref{section_autoeq_dP} we prove Corollary~\ref{cor_auto_equi_cubic-introdcution}. In Section \ref{sec-index-one}, we prove Theorem \ref{thm_1.6} and Corollary \ref{intro-cor-1.7} and \ref{intro-cor-1.8}.
In Appendix~\ref{section_classification_instanton_sheaves_d=2} we classify semistable sheaves of rank two, $c_1=0, c_2=2, c_3=0$ on quartic double solids. 


\subsection*{Acknowledgements}
We would like to thank Arend Bayer, Daniele Faenzi, Sasha Kuznetsov, Ziqi Liu, Franco Rota and Richard Thomas for useful conversations. We thank Arend Bayer, Sasha Kuznetsov, Pieter Moree, Richard Thomas and Xiaolei Zhao for their comments on the first draft of the paper. The first author acknowledges the support of EPSRC postdoctoral fellowship EP/T018658/1. The third author is supported by ERC Consolidator Grant WallCrossAG, no. 819864 , ANR project FanoHK, grant ANR-20-CE40-0023 and partially supported by GSSCU2021092. The second author would like to thank Institute for Advanced Study in Mathematics at Zhejiang University for financial support and wonderful research environment. Part of the work was finished when the third author is visiting Max-Planck institute for mathematics, IASM--Zhejiang University, Sichuan University and MCM--China Academy of Science. He is grateful for excellent working condition and hospitality. 

\section{Background: (weak) Bridgeland stability conditions}\label{survey_weak_stabilitycon}
In this section, we briefly review the notion of (weak) stability condition on $\D^b(Y)$ and $\Ku(Y)$ when $Y:=Y_d$ is a del Pezzo threefold of Picard rank one and degree $d$. By \cite{iskovskih:Fano-threefolds-I}, every del Pezzo threefold of Picard rank one belongs to the  following five families, indexed by their degree $d:=H^3\in\{1,2,3,4,5\}\colon$
\begin{itemize}
    \item $Y_5=\mathbb{P}^6\cap\mathrm{Gr}(2,5)$ is a codimension 3 linear section of Grassmannian $\mathrm{Gr}(2,5)$.
    \item $Y_4=Q\cap Q'$ is intersection of two quadric hypersurfaces in $\mathbb{P}^5$.
    \item $Y_3\subset\mathbb{P}^4$ is cubic threefold.
    \item $Y_2$ is a quartic double solid, i.e.~a double cover of $\mathbb{P}^3$ with smooth branch divisor $R\in |\oh_{\mathbb{P}^3}(4)|$.
    \item $Y_1$ is a degree 6 hypersurface of weighted projective space $\mathbb{P}(1,1,1,2,3)$.
\end{itemize}




\subsection{Weak stability conditions on $\D^b(Y)$} For any $b \in \mathbb{R}$, consider the full subcatgeory of complexes
\begin{equation}\label{Abdef}
\Coh^b(Y)\ =\ \big\{E^{-1} \xrightarrow{\,d\,} E^0 \ \colon\ \mu_H^{+}(\ker d) \leq b \,,\  \mu_H^{-}(\cok d) > b \big\} \subset\D^b(Y)
\end{equation}
Then $\Coh^b(Y)$ is the heart of a bounded t-structure on $\D^b(Y)$ by \cite[Lemma 6.1]{bridgeland:K3-surfaces}. For any pair $(b,w) \in \R^2$, we define a group homomorphism $Z_{b,w} \colon K(Y) \to \C$ by
\begin{equation} \label{eq:Zab}
Z_{b,w}(E) := -\ch_2(E)H + w \ch_0(E)H^3 + b(H^2\ch_1(E) -bH^3\ch_0(E))  + \mathfrak{i} \bigg(H^2\ch_1(E) -b H^3\ch_0(E) \bigg).
\end{equation}
In \cite{li:fano-picard-number-one}, the author defined an open region $\widetilde{U} \subset \mathbb{R}^2$ as the set of points $(b,w) \in \mathbb{R}^2$ above the curve $w = \frac{1}{2}b^2 - \frac{3}{8d}$ and above tangent lines of the curve $w = \frac{1}{2}b^2$ at $(k, \frac{k^2}{2})$ for all $k \in \ZZ$. 
 
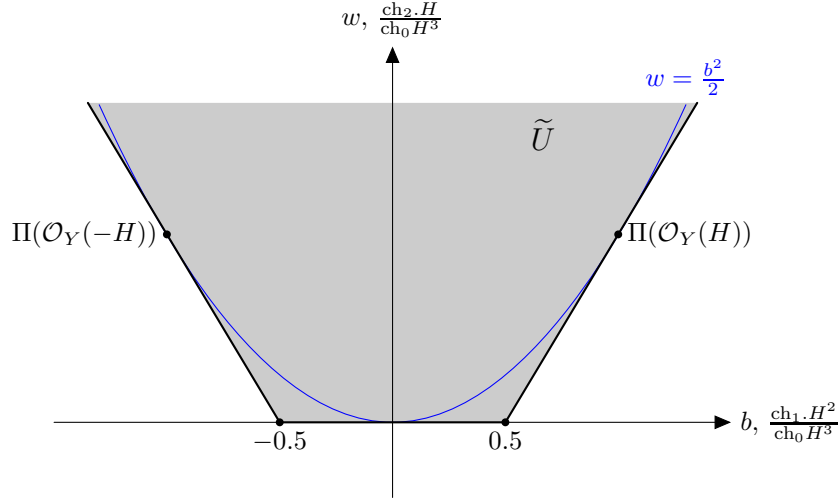
\begin{figure}[h]
		\definecolor{zzttqq}{rgb}{0.27,0.27,0.27}
		\definecolor{qqqqff}{rgb}{0.33,0.33,0.33}
		\definecolor{uququq}{rgb}{0.25,0.25,0.25}
		\definecolor{xdxdff}{rgb}{0.66,0.66,0.66}
		
		\begin{tikzpicture}[line cap=round,line join=round,>=triangle 45,x=1.0cm,y=1.0cm]
		
		\draw[->,color=black] (-4.5,0) -- (4.5,0);
		\draw  (4.5, 0) node [right ] {$b,\,\frac{\ch_1.H^2}{\ch_0H^3}$};

		\fill [fill=gray!40!white] (4.05, 4.25) -- (1.5,0) --(-1.5,0)--(-4.05, 4.25)--(4.05, 4.25);
		
		\draw[blue]  (0,0) parabola (3.9,4.225); 
		\draw[blue]  (0,0) parabola (-3.9,4.225); 
		\draw  (3.9 , 4.223) node [above, color=blue] {$w= \frac{b^2}{2}$};
		
		\draw[->,color=black] (0,-1) -- (0,5);
		\draw  (0, 5) node [above ] {$w,\,\frac{\ch_2.H}{\ch_0H^3}$};
		
		\draw [color=black, thick] (4.05, 4.25) -- (1.5,0);
		\draw [color=black, thick] (-4.05, 4.25) -- (-1.5,0);
		\draw [color=black, thick] (1.5, 0) -- (-1.5,0);
		
		\draw  (1.5, 0) node [below] {$0.5$};
		\draw  (-1.5, 0) node [below] {$-0.5$};
		\draw  (2, 3.5) node [above] {\Large{$\widetilde{U}$}};
		\draw  (3, 2.5) node [right] {$\Pi(\oh_Y(H))$};
		\draw  (-3, 2.5) node [left] {$\Pi(\oh_Y(-H))$};
		\begin{scriptsize}
		\fill (3, 2.5) circle (1.5pt);
		\fill (-3,2.5) circle (1.5pt);
		\fill (1.5,0) circle (1.5pt);
		\fill (-1.5,0) circle (1.5pt);
		\end{scriptsize}
		
		\end{tikzpicture}
		\caption{The space $\widetilde{U}$ when $d \leq 3$}
		\label{fig.widetilde{U}}
\end{figure}

\begin{figure}[h]
		\definecolor{zzttqq}{rgb}{0.27,0.27,0.27}
		\definecolor{qqqqff}{rgb}{0.33,0.33,0.33}
		\definecolor{uququq}{rgb}{0.25,0.25,0.25}
		\definecolor{xdxdff}{rgb}{0.66,0.66,0.66}
		
		\begin{tikzpicture}[line cap=round,line join=round,>=triangle 45,x=1.0cm,y=1.0cm]
        
        \fill [fill=gray!40!white] (0,-.4) parabola (1.75,.44) parabola [bend at end] (-1.75,.44) parabola [bend at end] (0,-.4);	
         \fill [white] (0,-.4) parabola (1.24,.01) parabola [bend at end] (-1.24,.01) parabola [bend at end] (0,-.4);
	
		\fill [fill=gray!40!white] (3.75, 3.62) -- (1.75,.43) --(-1.75,.43)--(-3.75, 3.62)--(3.75, 3.62);

		\draw[->,color=black] (-4.5,0) -- (4.5,0);
		\draw  (4.5, 0) node [right ] {$b,\,\frac{\ch_1.H^2}{\ch_0H^3}$};
		
		\draw[red]  (0,-.4) parabola (3.9,3.625); 
		\draw[red]  (0,-.4) parabola (-3.9,3.625); 
	\draw  (4.4 , 3.63) node [above, color=red] {$w= \frac{b^2}{2} - \frac{3}{8d}$};
		
		\draw[->,color=black] (0,-1) -- (0,4);
		\draw  (0, 4) node [above ] {$w,\,\frac{\ch_2.H}{\ch_0H^3}$};
		
		\draw[->,color=black] (1.24,0) -- (2.5,-1);
		\draw  (2.5, -1) node [right] {$\sqrt{\frac{3}{4d}}$};
		
		\draw[->,color=black] (-1.75,.43) -- (-2,-.43);
		\draw   (-2,-.53) node [left] {$1-\sqrt{\frac{3}{4d}}$};
		
		\draw [color=black, thick] (3.75, 3.62) -- (1.75,.43);
		\draw [color=black, thick] (-3.75, 3.62) -- (-1.75,.43);
		\draw [color=black, thick] (1.5, 0) -- (-1.5,0);
		
		\draw  (2, 2) node [above] {\Large{$\widetilde{U}$}};
		\draw  (3.05, 2.5) node [right] {$\Pi(\oh_Y(H))$};
		\draw  (-3.05, 2.5) node [left] {$\Pi(\oh_Y(-H))$};
		\begin{scriptsize}
		\fill (3.05, 2.5) circle (1.5pt);
		\fill (-3.05,2.5) circle (1.5pt);
		\fill (1.75,.43) circle (1.5pt);
		\fill (-1.75,.43) circle (1.5pt);
		
		\fill (1.24,0) circle (1.5pt);
		\fill (-1.24,0) circle (1.5pt);
		\end{scriptsize}
		
		\end{tikzpicture}
		\caption{The space $\widetilde{U}$ when $d =4, 5$}
		\label{fig..widetilde{U}}
\end{figure}
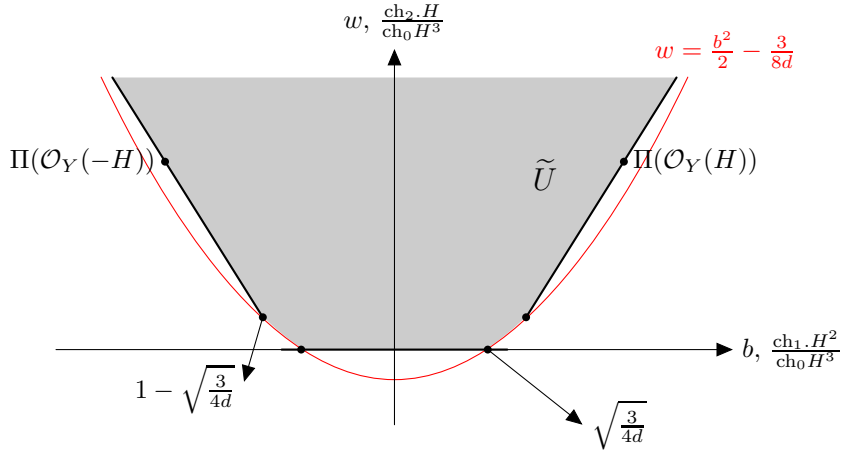

In Figures, we plot the $(b,w)$-plane simultaneously with the image of the projection map
\begin{eqnarray*}
	\Pi\colon\ K(Y) \setminus \big\{E \colon \ch_0(E) = 0\big\}\! &\longrightarrow& \R^2, \\
	E &\ensuremath{\shortmid\joinrel\relbar\joinrel\rightarrow}& \!\!\bigg(\frac{\ch_1(E).H^2}{\ch_0(E)H^3}\,,\, \frac{\ch_2(E).H}{\ch_0(E)H^3}\bigg).
\end{eqnarray*}

\begin{proposition}[{\cite[Proposition B.2]{bayer:the-space-of-stability-conditions-on-abelian-threefolds}}]
\label{first-tilting-wsc}
There is a continuous family of weak stability conditions on $\D^b(Y)$  parametrized by $\widetilde{U} \subset \R^2$, given by\footnote{We replaced the pair $(\alpha, \beta)$ with $(w = \frac{1}{2}\alpha^2 + \frac{1}{2}\beta^2, b= \beta)$.}
$$(b,w) \in \widetilde{U}\ \mapsto\ (\emph{Coh}^{b}(Y),\ Z_{b,w}). $$
\end{proposition}
We now expand upon the above statements. The function $-\frac{\Re[Z_{b,w}(E)]}{\Im[Z_{b,w}(E)]}$ for objects $E \in \Coh^b(Y)$ gives the same ordering as 
\begin{equation}\label{noo}
\nu_{b,w}(E)\ =\ \left\{\!\!\begin{array}{cc} \frac{\ch_2(E).H - w\ch_0(E)H^3}{\ch_1^{bH}(E).H^2}
 & \text{if }\ch_1^{bH}(E).H^2\ne0, \\
+\infty & \text{if }\ch_1^{bH}(E).H^2=0, \end{array}\right.
\end{equation}
where $\ch^{bH}(E):=\exp(-bH).\ch(E)$.
\begin{definition}
Fix a pair $(b,w) \in \widetilde{U}$. We say $E\in\D^b(Y)$ is $\nu_{b,w}$-(semi)stable if and only if
\begin{itemize}
\item $E[k]\in \Coh^b(Y)$ for some $k\in\mathbb{Z}$, and
\item $\nu_{b,w}(F)\,(\le)\,\nu_{b,w}\big(E[k]/F\big)$ for all non-trivial subobjects $F\hookrightarrow E[k]$ in $\Coh^b(Y)$.
\end{itemize}
Here $(\le)$ denotes $<$ for stability and $\le$ for semistability.
\end{definition}

\noindent The image $\Pi(E)$ of $\nu_{b,w}$-semistable objects $E$ with $\ch_0(E)\ne0$ is \emph{outside} $\widetilde{U}$ by \cite[Proposition 3.2]{li:fano-picard-number-one}, so in particular,  
\begin{equation}\label{discr}
	\Delta_H(E)\ =\  \big(\ch_1(E).H^2\big)^2 -2 (\ch_0(E)H^3)(\ch_2(E).H)\ \ge\ 0.
\end{equation}
\begin{proposition}[\textbf{Wall and chamber structure}]\label{locally finite set of walls}
	Fix $v\in K(Y)$ with $\Delta_H(v)\ge0$ and $\ch_{\le2}(v)\ne0$. There exists a set of lines $\{\ell_i\}_{i \in I}$ in $\mathbb{R}^2$ such that the segments $\ell_i\cap \widetilde{U}$ (called ``\emph{walls of instability}") are locally finite and satisfy 
	\begin{itemize}
	    \item[\emph{(}a\emph{)}] If $\ch_0(v)\ne0$ then all lines $\ell_i$ pass through $\Pi(v)$.
	    \item[\emph{(}b\emph{)}] If $\ch_0(v)=0$ then all lines $\ell_i$ are parallel of slope $\frac{\ch_2(v).H}{\ch_1(v).H^2}$.
	   		\item[\emph{(}c\emph{)}] The $\nu_{b,w}$-(semi)stability of any $E\in\D^b(Y)$ of class $v$ is unchanged as $(b,w)$ varies within any connected component (called a ``\emph{chamber}") of  $\widetilde{U} \setminus \bigcup_{i \in I}\ell_i$.
		\item[\emph{(}d\emph{)}] For any wall $\ell_i\cap \widetilde{U}$, there is an integer $k_i$ and a map $f\colon F\to E[k_i]$ in $\D^b(Y)$ such that
\begin{itemize}
\item for any $(b,w) \in \ell_i \cap \widetilde{U}$, the objects $E[k_i],\,F$ lie in the heart $\Coh^{b}(X)$,
\item $E$ is $\nu_{b,w}$-semistable of class $v$ with $\nu_{b,w}(E)=\nu_{b,w}(F)=\,\mathrm{slope}\,(\ell_i)$ constant on the wall $\ell_i \cap \widetilde{U}$, and
\item $f$ is an injection $F\hookrightarrow E[k_i] $ in $\Coh^{b}(Y)$ which strictly destabilises $E[k_i]$ for $(b,w)$ in one of the two chambers adjacent to the wall $\ell_i$.
\hfill$\square$
\end{itemize} 
	\end{itemize}
\end{proposition}


\subsection{Kuznetsov component}
The Kuznetsov component $\Ku(Y)$ is the right orthogonal complement of the exceptional collection $\oh_Y, \oh_Y(1)$ in $\D^b(Y)$ sitting in the semiorthogonal decomposition
$$\D^b(Y)= \langle \Ku(Y), \oh_Y, \oh_Y(H) \rangle=\langle\Ku(Y),\cQ_Y,\oh_Y\rangle,$$
where $\cQ_Y:=\bL_{\oh_Y}\oh_Y(1)[-1]$ is a rank $d+1$ vector bundle for $d\geq 2$
(see Section~\ref{def_gluing_object} for more details).  
We can identify the numerical Grothendieck group $\mathcal{N}(\Ku(Y))$ of $\Ku(Y)$ with the image of Chern character map $$\ch \colon K(\Ku(Y)) \rightarrow H^*(X, \mathbb{Q}).$$
It is a rank 2 lattice spanned by the classes 
\begin{equation*}
    \bv = \left(1,\ 0,\ -\frac{1}{d}H^2,\ 0\right) \qquad \text{and} \qquad \bw = \left(0,\ H,\ -\frac{1}{2}H^2 ,\ \left(\frac{1}{6} -\frac{1}{d}\right)H^3\right). 
\end{equation*}
With respect to this basis, the Euler form on $\mathcal{N}(\Ku(Y))$ is represented by the matrix
\begin{equation}
\label{eq_eulerform}
\begin{pmatrix}
 -1 & -1\\1-d & -d   
\end{pmatrix}.
\end{equation}

Consider any admissible subcategory $i \colon \mathcal{C} \hookrightarrow \D^b(Y)$. It has left and right adjoints $i^*$ and $i^{!}$. Similarly, the embedding $l \colon \mathcal{C}^{\perp} \hookrightarrow \D^b(Y)$ and $r \colon \prescript{\perp}{}{\mathcal{C}} \hookrightarrow \D^b(Y)$ has left and right adjoints. We know that any object $E \in \D^b(Y)$ lies in the exact triangles 
    \begin{equation*}
    r \circ r^{!}(E) \rightarrow    E \rightarrow i \circ i^{*}(E) \quad , \quad
  i \circ i^{!}(E) \rightarrow    E \rightarrow   l \circ l^{*}(E) . 
    \end{equation*}
    We define the right mutation along $\mathcal{C}$ to be the functor 
    \begin{equation*}
        \bR_{\mathcal{C}} \coloneqq r \circ r^{!} \colon \D^b(Y) \rightarrow r(\prescript{\perp}{}{\mathcal{C}})
    \end{equation*}
    and the left mutation along $\mathcal{C}$ to be 
    \begin{equation*}
        \bL_{\mathcal{C}} \coloneqq l \circ l^{*} \colon \D^b(Y) \rightarrow l({\mathcal{C}}^{\perp}).
    \end{equation*}
By \cite[Propostion 3.8]{kuznetsov-derived-category-cubic-3folds-V14}, we know
$\bL_{\cC}|_{r(\prescript{\perp}{}{\cC})}$ and $\bR_{\cC}|_{l({\cC}^{\perp})}$ are mutually inverse equivalence between the two orthogonal $\prescript{\perp}{}{\cC} \rightarrow \cC^{\perp}$ and $\cC^{\perp} \rightarrow \prescript{\perp}{}{\cC}$. Moreover,
\begin{equation*}
    (\bL_{\cC})|_{r(\prescript{\perp}{}{\cC})} = S_{\D^b(Y)} \circ r \circ S^{-1}_{\prescript{\perp}{}{\cC}} \circ r^* \quad , 
    \quad
    (\bR_{\cC})|_{l(\cC^{\perp})} = S_{\D^b(Y)}^{-1} \circ l \circ S_{\cC^{\perp}} \circ l^*. 
\end{equation*} 
Here $S_{\cT}$ denotes the Serre functor of a triangulated category $\cT$ (if it exists). 

Let $E \in \D^b(Y)$ be an exceptional object. Then the triangulated subcategory $\langle E \rangle$ generated by $E$ is an admissible subcategory. The embedding functor $i \colon \langle E \rangle \rightarrow \mathcal{T}$ has the left and right adjoints
    \begin{align*}
        i^{*} = E \otimes \RHom(F, E)^*, \quad  \quad i^{!}(F) = E \otimes \RHom(E, F). 
    \end{align*}
We will abuse notations and write $ \bR_{E}$ and $\bL_{E}$ for the corresponding right and left  mutations, respectively. 

We finish this section by defining the rotation functor. \cite[Lemma 4.1, Lemma 4.2]{kuznetsov-derived-category-cubic-3folds-V14} implies that the functor
\begin{equation}\label{rotation}
    \bO \colon \D^b(Y) \to \D^b(Y), \quad \bO(-)=\bL_{\oh_Y}(- \otimes \oh_Y(H))
\end{equation}
is an auto-equivalence of $\Ku(Y)$, called rotation functor. 
By \cite[Lemma 4.1]{kuznetsov-derived-category-cubic-3folds-V14}, we have
\[S^{-1}_{\Ku(Y)}=\bO^2[-3].\]
The rotation functor $\bO$ induces an auto-isometry of the numerical Grothendieck group $\mathcal{N}(\Ku(Y_d))$ for each $d$. In particular for $d=3$, we have 
\[\begin{tikzcd}
	\bv & {-2\bv+\bw} & {\bv-\bw} & \bv.
	\arrow["\bO", from=1-1, to=1-2]
	\arrow["\bO", from=1-2, to=1-3]
	\arrow["\bO", from=1-3, to=1-4]
\end{tikzcd}\]
And for $d=2$, we have 
\[\begin{tikzcd}
	\bv & {-\bv+\bw} & {-\bv.}
	\arrow["\bO", from=1-1, to=1-2]
	\arrow["\bO", from=1-2, to=1-3]
\end{tikzcd}\]

\subsection{Bridgeland stability conditions on $\Ku(Y)$}
For any pair $(b,w) \in \widetilde{U}$, consider the tilted heart $\Coh^0_{b,w}(Y) = \langle \mathcal{F}_{b,w}[1] , \cT_{b,w} \rangle$ where $\mathcal{F}_{b,w}$ ($\cT_{b,w}$) is the subcategory of objects in $\Coh^{b}(X)$ with $\nu^+_{b,w} \leq b$ ( $\nu^-_{b,w} > b$). By \cite[Proposition 2.14]{bayer:stability-conditions-kuznetsov-component}, the pair $\sigma^0_{b,w} \coloneqq \left(\Coh^0_{b,w}(X) , Z^0_{b,w} \right)$ is a weak stability condition on $\D^b(Y)$, where $Z^0_{b,w}:=-\mathfrak{i}Z_{b,w}$. We denote the corresponding slope function by 
\[\mu^0_{b,w}(-):=-\frac{\Re[Z^0_{b,w}(-)]}{\Im[Z^0_{b,w}(-)]}.\]

\begin{lemma}[{\cite[Proposition 4.1]{feyzbakhsh2021serre}}] \label{compare_stability}
    Any $\sigma^0_{b,w}$-(semi)stable object $E \in \Coh^0_{b,w}(Y)$ is $\nu_{b,w}$-(semi)stable if it does not lie in an exact triangle of the form
\begin{equation*}
F[1] \rightarrow E \rightarrow T    
\end{equation*}
where $F \in \mathcal{F}_{b,w}$ is $\nu_{b,w}$-(semi)stable and $T \in \Coh_0(X)$. Conversely, take a $\nu_{b,w}$-(semi)stable object $E$ such that either
\begin{enumerate}
    \item $E \in \cT_{b,w}$ and $\Hom(\Coh_0(X), E) = 0$, or
    \item $E \in \mathcal{F}_{b,w}$ and $\Hom(\Coh_0(X), E[1]) = 0$. 
\end{enumerate}
Then $E$ is $\sigma^0_{b,w}$-(semi)stable.  
\end{lemma}

By restricting weak stability conditions $\sigma^0_{b,w}$ to the Kuznetsov component $\Ku(Y)$, we obtain stability conditions on it.   
\begin{theorem}[{\cite[Theorem 6.8]{bayer:stability-conditions-kuznetsov-component}}]
\label{thm_stabcondinduced}
For every pair $(b,w)$ in the subset 
\begin{equation*}
    V \coloneqq \left\{(b,w) \in \widetilde{U} \colon  -\frac{1}{2} \leq b < 0, \ w < b^2 \ \ \text{or} \ \  -1 < b < -\frac{1}{2}, \ w \leq  b^{2}+b+\frac{1}{2}   \right\} \ \subset \  \widetilde{U}, 
\end{equation*}
the pair $\sigma(b,w)=(\cA(b,w), Z(b,w))$ is a Bridgeland stability condition on $\Ku(Y_d)$ where
\begin{equation*}
    \cA(b,w) \coloneqq \Coh_{b,w}^{0}(Y_d)\cap \Ku(Y_d) \qquad \text{and} \qquad
    Z(b,w) \coloneqq Z^0_{b,w}|_{\Ku(Y_d)}.
\end{equation*}
\end{theorem}
\begin{proof}
    Applying the same argument as in the proof of {\cite[Theorem 6.8]{bayer:stability-conditions-kuznetsov-component}} shows that $\sigma(b,w)$ is a Bridgeland stability condition on $\Ku(Y_d)$ if $-1 < b< 0$ and 
    \begin{equation*}
        \nu_{b,w}(\oh_{Y_d}(-2H)[1]) \leq \nu_{b,w}(\oh_{Y_d}(-H)[1])\leq  b < \nu_{b,w}(\oh_{Y_d}) \leq \nu_{b,w}(\oh_{Y_d}(H)). 
    \end{equation*}
\end{proof}
On the stability manifold which we denote by $\Stab(\Ku(Y))$ we have: 
\begin{enumerate} 
\item  \emph{a right action of the universal covering space $\widetilde{\mathrm{GL}}^+_2(\mathbb{R})$ of $\mathrm{GL}^+_2(\mathbb{R})$}: for a stability condition $\sigma=(\cP,Z) \in \Stab(\Ku(Y))$ and $\tilde{g}=(g,M) \in \widetilde{\mathrm{GL}}^+_2(\mathbb{R})$, where $g: \mathbb{R} \to \mathbb{R}$ is an increasing function such that $g(\phi+1)=g(\phi)+1$ and $M \in \mathrm{GL}^+_2(\mathbb{R})$, we define $\sigma \cdot \tilde{g}$  to be the stability condition $\sigma'=(\cP',Z')$ with $Z'=M^{-1} \circ Z$ and $\cP'(\phi)=\cP(g(\phi))$ (see \cite[Lemma 8.2]{bridgeland:space-of-stability-conditions}).
\item \emph{a left action of the group of exact auto-equivalences $\Aut(\Ku(Y))$ of $\Ku(Y)$}: for $\Phi \in \Aut(\Ku(Y))$ and $\sigma \in \Stab(\Ku(Y))$, we define $\Phi \cdot \sigma=(\Phi(\cP), Z \circ \Phi_*^{-1})$, where $\Phi_*$ is the automorphism of $K(\Ku(Y))$ induced by $\Phi$.
\end{enumerate}

\begin{remark}\label{remark-the same orbit}
    Note that all stability conditions $\sigma(b,w)$ for $(b,w) \in V$ lie in the same orbit with respect to the action of $\widetilde{\mathrm{GL}}^+_2(\mathbb{R})$ by \cite[Proposition 3.5]{pertusi:some-remarks-fano-threefolds-index-two} \footnote{This is proved for $V \cap U$, but the same proof is valid for $V$.}. Hence if $E \in \Ku(Y_d)$ is $\sigma(b,w)$-(semi)stable with respect to some $(b,w) \in V$, then it is $\sigma(b,w)$-(semi)stable with respect to any $(b,w) \in V$.  
\end{remark}


We now give a case-by-case investigation of the category $\Ku(Y_d)$ when $d \geq 2$:

\begin{enumerate}
    \item [$d=5$.] $Y_5$ is a linear section of codimension 3 of $\Gr(2, 5)$. Let $\mathcal{U}$ be the restriction of the tautological rank $2$ subbundle from $\Gr(2, 5)$ to $Y_5$, and let $\mathcal{U}^{\perp} = \ker(\oh_Y \otimes \Hom(\oh_Y, \mathcal{U}^*) \to \mathcal{U}^*)$, then \cite[Lemma 4.1]{Kuznetsov:instanton-bundle-Fano-threefold} gives
    $$\Ku(Y_5) = \langle \mathcal{U}, \ \mathcal{U}^{\perp}  \rangle . $$
    \item [$d=4$.] $Y_4$ is an intersection of $2$ quadrics in $\PP^5$. By \cite[Theorem 5.1]{Kuznetsov:instanton-bundle-Fano-threefold}, there exists a curve $C$ of genus $2$ such that we have an equivalence $\Ku(Y_4) \cong \D^b(C_2)$. Hence, there is a unique Bridgeland stability condition on $\Ku(Y_4)$ up to the action of $\widetilde{\mathrm{GL}}^+_2(\mathbb{R})$ by \cite{macri:stability-conditions-on-curves}. 
    \item [$d=3$.] $Y_3$ is a cubic 3-fold, and $\Ku(Y_3)$ is a fractional Calabi--Yau category of dimension $\frac{5}{3}$, i.e. $S_{\Ku(Y_3)}^3 =[5]$. Note that by \cite[Lemma 4.1, Lemma 4.2]{kuznetsov-derived-category-cubic-3folds-V14}, we have $S^{-1}_{\Ku(Y_3)} = \bO^2[-3]$. In this case, we only consider {Serre-invariant stability conditions} on $\Ku(Y_3)$, i.e. those $\sigma \in \Stab(\Ku(Y_3))$ so that $S_{\Ku(Y_3)}. \sigma = \sigma . \tilde{g}$ for some $\tilde{g} \in \widetilde{\text{GL}}_2^+(\R)$. By \cite{pertusi:some-remarks-fano-threefolds-index-two}, all stability conditions constructed in Theorem \ref{thm_stabcondinduced} are Serre-invariant. And it is proved in \cite[Sections 4 \& 5]{feyzbakhsh2021serre} and \cite[Theorem 4.25]{jacovskis2021categorical} that all Serre-invariant stability conditions on $\Ku(Y_3)$ lie in the same orbit with respect to the action of $\widetilde{\mathrm{GL}}^+_2(\mathbb{R})$. 
    \item [$d=2$.] $Y_2$ is a double cover of $\PP^3$ ramified in a quartic surface, called a quartic double solid. By \cite[Corollary 4.6]{Kuz:calabi}, the Serre functor of $\Ku(Y_2)$ is $S_{\Ku(Y_3)} = \tau [2]$ where $\tau$ is the auto-equivalence of $\Ku(Y_2)$ induced by the involution $\tau$ of the double covering. As the involution $\tau$ preserves $\Coh(X)$ and Chern characters, the stability conditions $\sigma(b,w)$ constructed in Theorem \ref{thm_stabcondinduced} are Serre-invariant, see \cite[Lemma 6.1]{pertusi:some-remarks-fano-threefolds-index-two}. Moreover, \cite[Theorem 3.2 \& Remark 3.8]{feyzbakhsh2021serre} and \cite[Theorem 4.25]{jacovskis2021categorical} implies that all Serre-invariant stability conditions on $\Ku(Y_2)$ lie in the same orbit with respect to action of $\widetilde{\mathrm{GL}}^+_2(\mathbb{R})$.   
\end{enumerate}

\section{Del Pezzo threefolds of Picard rank one}\label{results_general_Wallcross}
In this section, we gather all results which are valid for del Pezzo threefold $Y$ of Picard rank one and degree $d$. By \cite{kuznetsov:derived-category-fano-threefold}, for any $E \in \D^b(Y)$, we know 
\begin{equation*}
    \chi(\oh_Y, E) = \ch_0(E)+H^2\ch_1(E) \frac{d+3}{3d} +H\ch_2(E) +\ch_3(E).
\end{equation*}

\subsection{Instanton bundles and their acyclic extensions}
\label{subsection.instanton}
An instanton of charge $n$ on $Y$ is a Gieseker-stable vector bundle $E$ with $\ch_{\leq 2}(E) = (2, 0, -n\frac{H^2}{d})$ satisfying instanton condition $H^1(Y, E(-1)) = 0$. By \cite[Lemma 3.5]{Kuznetsov:instanton-bundle-Fano-threefold}, for each instanton bundle $E$, we have $h^1(E) = n-2$, thus there exists a unique short exact sequence 
$$0 \to E \to \tilde{E} \to \oh_Y^{n-2} \to 0 $$
such that $\tilde{E}$ is acyclic, i.e. $H^{i}(Y, \tilde{E}) = 0$ for any $i$. Note that $\tilde{E} = \bL_{\oh_Y} E$ and is of Chern character 
\begin{equation*}
    n \bv =  \left(n,\ 0,\ -n\frac{H^2}{d},\ 0\right). 
\end{equation*}
Moreover, it is $\nu_{b,w}$-semistable for $b <0$ and $w \gg 0$. 

\bigskip

Let $\ell_d$ be the line passing through $\Pi(n\bv) = (0, -\frac{1}{d})$ and $\Pi(\oh_Y(-H)) = (-1, \frac{1}{2})$, so it is of equation $w = -\frac{d+2}{2d} b -\frac{1}{d}$. If $d=2$, then $\ell_d$ coincides with the boundary of $\widetilde{U}$, and if $d \geq 3$, then it intersects $\partial \widetilde{U}$ at two points with $b$-values $b^d_1<b^d_2$ so that  
\begin{equation}\label{b-i-d}
    b^d_1 \leq -1 \qquad \text{and} \qquad -\frac{2}{d+2} = b^d_2.
\end{equation}

\begin{lemma} \label{no_wall_lem}
Take a class $\alpha \in K(X)$ with $\ch_{\leq 2}(\alpha) = n\left(1, 0, -\frac{H^2}{d}\right)$ such that $n \leq d+1$. Then there is no wall for class $\alpha$ above $\ell_d$. In particular, an object $E \in \Coh^b(Y)$ of Chern character $\alpha$ which is $\nu_{b,w}$-semistable for $b<0$ and $w\gg 0$ satisfies $\RHom(\oh_Y, E)=\Hom(\oh_Y, E[1])[-1]$ and hence $\ch_3(E)\leq 0$. 
\end{lemma}
\begin{proof}
Suppose for a contradiction that there is such a wall $\ell$ for class $\alpha$ above $\ell_d$ with the destabilising sequence $E_1 \rightarrow E \rightarrow E_2$. Let $b_1 < b_2$ be the intersection points of $\ell$ with the boundary $\partial \widetilde{U}$. Then for $i=1, 2$,
    	\begin{equation*}
    	\mu_H^+(\cH^{-1}(E_i)) \leq \ b_1  \qquad \text{and} \qquad  b_2 \leq \ \mu_H^-(\cH^0(E_i)). 
    	\end{equation*}
    	 Let $(r, cH) = \ch_{\leq 1}(\cH^{-1}(E_1)) + \ch_{\leq 1}(\cH^{-1}(E_2))$, then $(r+n, cH) = \ch_{\leq 1}(\cH^{0}(E_1)) + \ch_{\leq 1}(\cH^{0}(E_2))$, so  
    	 \begin{equation}\label{ineq.slope}
    	 b_2(r+n) \leq  c \ \leq  \ b_1 r.  
    	 \end{equation}
    	 Note that if $\rk(\cH^{-1}(E_i)) = 0$, then $\cH^{-1}(E_i) = 0$. If $d = 2$, then $\ell_d$ lies on the boundary $\partial\widetilde{U}$, so we have $b_1 < -\frac{3}{2}$ and $-\frac{1}{2} < b_2$, so \eqref{ineq.slope} gives $-\frac{1}{2}(r+n) < c < -\frac{3}{2}r$ which has no solution for $n \leq 3$. If $d \geq 3$, then combining \eqref{b-i-d} and \eqref{ineq.slope} gives $ -\frac{2}{d+2}(r+n) <   c  <  -r$ which is not possible for $k \leq d+1$.   

      For the second claim, we know $E$ is semistable at the large volume limit, so $\Hom(\oh_Y, E)=0$. Also the first part implies that $E$ is $\nu_{b,w}$-semistable for all $(b,w) \in \widetilde{U}$ over $\ell_d$. Since the line segment connecting $\Pi(E)$ and $\Pi(\oh_Y(-2))$ is above $\ell_d$, we have $\Hom(E, \oh_Y(-2H)[1])=\Hom(\oh_Y,E[2])=0$. And we know that $\Hom(\oh_Y,E[i])=\Hom(E,\oh_Y(-2)[3-i])=0$ for $i\neq 1$. Thus  $\chi(E)=-\hom(\oh_Y,E[1])=\ch_3(E)\leq 0$, which gives $\ch_3(E)\leq 0$.
\end{proof}

As a result of the above lemma, we may identity Gieseker stable sheaves with the large volume limit stable ones. 

\begin{lemma} \label{Gstability_lvls_coincide}
    Let $E$ be an object of class $\ch(E)=n\bv$ where $1\leq n\leq d+2$. Then $E$ is $\nu_{b,w}$-(semi)stable for $b<0$ and  $w\gg 0$ (or equivalently, 2-Gieseker-(semi)stable) if and only if $E$ is a Gieseker-(semi)stable sheaf.
\end{lemma}

\begin{proof}
By \cite[Proposition 4.8]{feyz:desing}, the 2-Gieseker-(semi)stability for $E$ coincides with $\nu_{b,w}$-(semi)stability for $b<0$ and $w\gg 0$. Then in the following we will show 2-Gieseker-(semi)stability for $E$ coincides with Gieseker-(semi)stability

It is clear that if $E$ is 2-Gieseker-stable, then $E$ is Gieseker-stable. Conversely, if $E$ is Gieseker-stable but strictly 2-Gieseker-semistable, then we can find an exact sequence $0\to E_1\to E \to E_2\to 0$ such that $E_i$ are 2-Gieseker-semistable of classes $\ch(E_i)=(k_i,0,\frac{k_i}{d}H^2,m_i)$, where $1\leq k_i\leq n-1\leq d+1$ and $m_i\in \ZZ_{\leq 0}$. By the stability of $E_i$, we have $m_i\leq 0$ for any $i$ from Lemma \ref{no_wall_lem}. Since $m_1+m_2=0$, we have  $\ch(E_i)=k_i\bv$ and contradicts the Gieseker-stability of $E$.

And it is clear that if $E$ is Gieseker-semistable, then $E$ is 2-Gieseker-semistable. Now assume that $E$ is 2-Gieseker-semistable but not Gieseker-semistable. Then the maximal destabilizing subsheaf $E_1$ of $E$ with respect to Gieseker-semistability has class $\ch(E_1)=(k_1,0,-\frac{k_1}{d}H^2, m_1)$ where $1\leq k_1<n$ and $m_i\in \ZZ_{> 0}$. But this contradicts Lemma \ref{no_wall_lem} as well.
\end{proof}


\subsection{The bundle $\cQ_Y$ and its projection}\label{def_gluing_object} For any smooth Fano threefold $Y$ of index $2$ and degree $d \geq 2$, we define the sheaf $\cQ_Y$ to be the kernel of the following evaluation map
  	\begin{equation}\label{eq. Q}
  	0 \rightarrow \cQ_Y \rightarrow \oh_Y \otimes \Hom(\oh_Y,\oh_Y(1))\xrightarrow{ev} \oh_Y(1) \rightarrow 0. 
  	\end{equation}
   We have 
   \begin{equation}\label{chern-Q-Y}
  	\ch(\cQ_Y) = \left(d+1,\ -H,\ -\frac{1}{2}H^2 ,\ -\frac{1}{6}H^3\right). 
  	\end{equation}
   \begin{lemma}\label{lem-Q-Y}
       The sheaf $\cQ_Y$ is a $\mu_H$-stable locally-free sheaf.  
   \end{lemma}
\begin{proof}
When the  degree $d$ of $Y$ satisfies $d\geq 2$, $\oh_Y(1)$ has no base-point by \cite[Theorem 2.4.5.(i)]{iskovskikh1999algebraic}, hence $\cQ_Y$ is a bundle of rank $d+1$. If it is not $\mu_H$-stable, there is a stable reflexive sheaf $Q' \subset \cQ_Y$ of bigger or equal slope, thus $\mu_H(Q') \geq 0$. Since it is also a subsheaf of $\oh_Y^{\oplus h^0(\oh_Y(1))}$ and all stable factors of the latter are the direct sum of $\oh_Y$, we get $Q'$ is a direct sum of $\oh_Y$ which is not possible as $h^0(\cQ_Y) = 0$ by the definition.         
\end{proof}

Consider the semiorthogonal decomposition $\D^b(Y)= \langle \Ku(Y), \oh_Y, \oh_Y(H) \rangle$. We know $\cQ_Y\cong\bL_{\oh_Y}\oh_Y(1)[-1]$. Consider the embedding $i \colon \Ku(Y) \hookrightarrow \D^b(Y)$. We know $Q_Y \in \langle \oh_Y(-H), \Ku(Y) \rangle$, thus it lies in the exact triangle 
\begin{equation*}
i^{!}\cQ_Y = \bR_{\oh_Y(-H)}(\cQ_Y) \to \cQ_Y \to \oh_Y(-H)\otimes \RHom(\cQ_Y, \oh_Y(-H))^{\vee}.
\end{equation*}
The $\mu_H$-stability of $\cQ_Y$ implies that $\Hom(\cQ_Y, \oh_Y(-H)[i]) = 0$. Taking $\Hom(\oh_Y(H), -)$ from the exact sequence \eqref{eq. Q} implies that $\hom(\cQ_Y, \oh_Y(-H)[1]) = \hom(\oh_Y(H), \cQ_Y[2]) = 0$. Thus 
\begin{equation}\label{ext-1}
\hom(\cQ_Y, \oh_Y(-H)[2]) = \chi(\cQ_Y, \oh_Y(-H)) =1.    
\end{equation}
Hence $i^{!}\cQ_Y$ is a two-term complex lying in the exact triangle 
\begin{equation}\label{exact. i!QY}
\oh_Y(-H)[1] \to i^{!}\cQ_Y  \to \cQ_Y 
\end{equation}
which is of Chern character $\ch(i^{!}\cQ_Y) = d\bv$. In Sections~\ref{section_Y2} and \ref{section_Y3} we show that if $d=2$ and $d=3$, the object $i^{!}\cQ_Y$ is  Bridgeland-stable in $\Ku(Y)$ and it is the only such object which is not Gieseker-stable.

\section{Moduli spaces on quartic double solids }\label{section_Y2}
In this section, we always fix $Y$ to be a del Pezzo threefold of degree two, i.e.~a quartic double solid. We aim to classify Bridgeland semistable objects of class $2\bv$ in $\Ku(Y)$ as described in the following. 


\begin{proposition} \label{prop_bridgeland_Y2}
    Let $\sigma$ be a Serre-invariant stability condition on $\Ku(Y)$ and $E\in \Ku(Y)$ be a $\sigma$-(semi)stable object of class $2\bv$. Then up to a shift, $E$ is either a Gieseker-(semi)stable sheaf or $i^!\cQ_Y$.
\end{proposition}

\begin{proof}

By the uniqueness of Serre-invariant stability condition, we can assume that $E\in \cA(b, w)$ is a $\sigma(b, w)$-(semi)stable object of class\footnote{We put the shifted class $-2\bv$ to get sure $\Im[Z(b,w)] \geq 0$ for $(b,w) \in V$.} $-2\bv$. We divide the proof into several cases.

\textbf{Step 1.} First we assume that $E$ is $\sigma^0_{b_0, w_0}$-semistable for some $(b_0, w_0)\in V$. Then by Lemma \ref{compare_stability}, we have an exact sequence in $\Coh^0_{b_0, w_0}(Y)$
\[F[1]\to E\to T,\]
where $F\in \Coh^{b_0}(Y)$ with $\nu^+_{b_0, w_0}(F)\leq b$ and $T=0$ or supported on points. Now by the $\sigma^0_{b_0, w_0}$-semistability of $E$, we know that $F$ is $\nu_{b_0,w_0}$-semistable. By Lemma \ref{no_wall_lem}, $F$ is $\nu_{b_0, w}$-semistable for $w\gg 0$ and $\ch_3(F)\leq 0$, which implies $T=0$ and $F[1]=E$. Thus $E[-1]$ is $\nu_{b, w}$-semistable for $w\gg 0$, which implies that $E[-1]$ is a Gieseker-semistable sheaf by Lemma \ref{Gstability_lvls_coincide}.

\textbf{Step 2.} Now we assume that $E$ is not $\sigma^0_{b, w}$-semistable for any $(b, w)\in V$. By \cite[Proposition 2.2.2]{bayer:bridgeland-stability-conditions-on-threefolds}, we can assume that there is an open ball $U' \subset \R^2$ containing the point $(b, w)=(-1,\frac{1}{2})$ such that for any $(b, w)\in U_{-1,\frac{1}{2}}:=U'\cap V$, we have $E\in \cA(b, w)$ and the Harder--Narasimhan filtration of $E$ with respect to $\sigma^0_{b, w}$ is constant.

Let $B$ be the destabilizing quotient object of $E$ with minimum slope and $A\to E\to B$ be the destabilizing sequence of $E$ with respect to $\sigma^0_{b,w}$ for $(b,w)\in U_{-1,\frac{1}{2}}$. Hence $A,B\in \Coh^0_{b,w}(Y)$, which gives 
\begin{equation} \label{imaginary_part}
    \Im(Z^0_{b,w}(E))\geq \Im(Z^0_{b,w}(B))>0,\quad \Im(Z^0_{b,w}(E))>\Im(Z^0_{b,w}(A))\geq 0
\end{equation}
for all $(b,w)\in U_{-1,\frac{1}{2}}$. Since $\Im(Z^0_{-1,\frac{1}{2}}(E))=0$, by the continuity, we have $\Im(Z^0_{-1,\frac{1}{2}}(A))=\Im(Z^0_{-1,\frac{1}{2}}(B))=0$. Therefore, if we assume that $\ch_{\leq 2}(B)=(x,yH,\frac{z}{2}H^2)$ for $x,y,z\in \ZZ$, from $\Im(Z^0_{-1,\frac{1}{2}}(B))=0$ we get $z=-x-2y$. Thus we have

\begin{equation}
    \ch_{\leq 2}(B)=\left(x,\ yH,\ \frac{-x-2y}{2}H^2\right),\quad  \ch_{\leq 2}(A)=\left(-2-x,\ -yH,\ \frac{x+2y+2}{2}H^2\right) 
\end{equation}
and by \eqref{imaginary_part} we get

\begin{equation} \label{im_part}
    1-2b^2+2w = \Im(Z^0_{b,w}(E))\ \geq\ \Im(Z^0_{b,w}(B))=  (2b^2-2w-1)\frac{x}{2}-(b+1)y>0
\end{equation}
for all $(b,w)\in U_{-1,\frac{1}{2}}$. Moreover, by definition we have $\mu^{0}_{b,w}(E)>\mu^{0}_{b,w}(B)$ for any $(b,w)\in U_{-1,\frac{1}{2}}$ where $\mu^{0}_{b,w}(-) = -\frac{\Re[Z^0_{b,w}(-)]}{\Im[Z^0_{b,w}(-)]}$, thus 
\begin{equation} \label{slop_ineq}
    \frac{-2b}{1-2b^2+2w} = \mu^{0}_{b,w}(E)\ >\ \mu^{0}_{b,w}(B) = \frac{(b x-y)}{(2b^2-2w-1)\frac{x}{2}-(b+1)y}\,.
\end{equation}
Now by \eqref{im_part}, $b<0$ and \eqref{slop_ineq}, we have
\begin{equation} \label{eq1}
    -2b>b x-y.
\end{equation}

On the other hand, from \cite[Remark 5.12]{bayer:stability-conditions-kuznetsov-component}, we have 
$$\left(\mu^{0}_{b,w}\right)^-(E) \coloneqq \mu^0_{b,w}(B)\geq \min\{\mu^0_{b,w}(E),\ \mu^0_{b,w}(\oh_Y),\ \mu^0_{b,w}(\oh_Y(1))\}$$ 
for any $(b,w) \in V$.
Note that $\mu^0_{-1,\frac{1}{2}}(\oh_Y)=-2$, $\mu^0_{-1,\frac{1}{2}}(\oh_Y(1))=-1$ and $\mu^0_{b,w}(E)>0$ when $(b,w)\in U_{-1,\frac{1}{2}}$ as $\Re[Z_{b,w}^0(E)] = 2b <0$, thus $\mu^0_{b,w}(B) \geq -2$. By taking the limit $b\to -1$ and $w \to \frac{1}{2}$ and combining with \eqref{eq1}, we get 
$$2\geq -x-y\geq 0.$$ 


\textbf{Case 1.} $-x-y=0$. Then 
\eqref{slop_ineq} for $-y=x$ gives
\[\frac{-2b}{1-2b^2+2w}>\frac{(b+1)}{(2b^2-2w-1)\frac{1}{2}+(b+1)},\]
which has no solution for $(b,w)\in V$.

\textbf{Case 2.} $-x-y=1$. Then $\ch_{\leq 2}(B)=(x,(-x-1)H,(\frac{x}{2}+1)H^2)$. Since $B$ is $\sigma^0_{b,w}$-semistable, Lemma \ref{compare_stability} implies that $\ch_{\leq 2}(B)$ is a possible class for $\ch_{\leq 2}$ of a $\nu_{b,w}$-semistable object $B'[1]$ where $B' \in \Coh^b(Y)$. By \cite[Proposition 3.2]{li:fano-picard-number-one}, the only possible cases are $x=\pm 1$ and $\pm 2$. Using \eqref{slop_ineq}, we get $x=-2$ and other cases are ruled out. Then we see $\ch_{\leq 2}(B')=(-2,H, 0)$. But then $\nu_{b,w}$-semistability of $B'$ for $(b,w) \in U_{-1,\frac{1}{2}}$ and wall and chamber structure described in Proposition \ref{locally finite set of walls} implies that $B'$ is $\nu_{b=-1, w}$-semistable when $\frac{1}{2}< w < \frac{1}{2} + \epsilon$. Since there is no wall for $B'$ crossing the vertical line $b=-1$, we get $B'$ is $\nu_{b =-1,w}$-semistable for $w \gg 0$. Thus $B'$ is a $\mu_H$-stable sheaf which is not possible by the following Lemma \ref{rule_out}.


\textbf{Case 3.} $-x-y=2$. Then we have $\ch_{\leq 2}(B)=(x,(-x-2)H,(\frac{x}{2}+2)H^2)$. By \cite[Proposition 3.2]{li:fano-picard-number-one}, we have $|x|\leq 3$. Using \eqref{slop_ineq}, we get $x=-3$ and other cases are ruled out. Then $\ch_{\leq 2}(B)=(-3,H, \frac{1}{2}H^2)$. We claim that $\RHom(\oh_Y, B)=0$, which implies $\ch(B)=(-3,H, \frac{1}{2}H^2,\frac{1}{6}H^3)$. Indeed, since $\oh_Y, \oh_Y(-2)[2]\in \Coh^0_{b,w}(X)$, by Serre duality we have $\Hom(\oh_Y,B[i])=\Hom(B,\oh_Y(-2)[3-i])=0$ for $i\neq 0,1$. We know $\lim_{(b,w) \to (-1, \frac{1}{2})}\mu^0_{b,w}(B) = +\infty$, so by shrinking the open ball $U'$, we may assume  
\begin{equation}\label{equ}
    (\mu^0_{b,w})^-(A) > \mu^0_{b,w}(B) > \mu^0_{b,w}(\oh_Y(-2)[2])
\end{equation}
Then $\sigma^0_{b,w}$-semistability of $B$ and $\oh_Y(-2)[2]$ implies that $\Hom(\oh_Y, B[1])=\Hom(B,\oh_Y(-2)[2]) = 0$
Moreover, using $E\in \Ku(Y)$, we have $\Hom(\oh_Y, B)=\Hom(\oh_Y, A[1])$. Then \eqref{equ} gives $\Hom(\oh_Y, A[1])=\Hom(A,\oh_Y(-2)[2]) = 0$, so the claim follows. Then Lemma \ref{rk3_deg2} implies that $B=\cQ_Y[1]=\bL_{\oh_Y}\oh_Y(1)$. 

\vspace{.2 cm }
We know $\ch(A) = \ch(\oh_Y(-1)[2])$, so $\lim_{(b,w) \to (-1, \frac{1}{2})}Z^0_{b,w}(A) = 0$, thus if $A$ is not $\sigma^0_{b,w}$-semistable for any $(b,w) \in U'$, then the destabilising factors $A_i$ all satisfy $\lim_{(b,w) \to (-1, \frac{1}{2})}\Im[Z^0_{b,w}(A_i)] = 0$. Since by \eqref{equ}, we know $\mu^0_{b,w}(A_i)\geq 0$, we have $\Re[Z^0_{b,w}(A_i)] \leq 0$ for all $i$. This implies that $\lim_{(b,w) \to (-1, \frac{1}{2})}\Re[Z^0_{b,w}](A_i) = 0$, and so $\ch_{\leq 2}(A_i)$ is a multiple of $\ch_{\leq 2}(\oh_Y(-1))$ which is not possible. Thus $A$ is $\sigma^0_{b,w}$-semistable with 
\[\Hom(A,\oh_Y(-1)[2])=\Hom(\oh_Y(1),A[1])=\Hom(\oh_Y(1),B)\neq 0.\]
This shows that $A = \oh_Y(-1)[2]$ and so $E=i^!\cQ_Y[1]$ as $\Hom(\cQ_Y[1], \oh_Y(-1)[3]) = 1$ by \eqref{ext-1}. Finally Lemma \ref{lem-serre-stability-d-2} completes the proof. 

\end{proof}

\begin{lemma}\label{rule_out}
Let $F$ be a slope stable sheaf with $\ch_{\leq 2}(F)=(2,-H,sH^2,tH^3)$. Then $s\leq -\frac{1}{2}$. And if $s=-\frac{1}{2}$, then $t\leq \frac{1}{3}$. Moreover, when $s=-\frac{1}{2}$ and $t=\frac{1}{3}$, $F$ is locally free.
   
\end{lemma}

\begin{proof}
Assume that $s>-\frac{1}{2}$. By \cite[Proposition 3.2]{li:fano-picard-number-one}, we have $s=0$. Thus $\ch_{\leq 2}(F)=\ch_{\leq 2}(F^{\vee \vee})$ and we can assume that $F$ is reflexive. Since $\ch^{-1}_1(F)=1$, there is no wall for $F$ intersects with $b=-1$. Since the line segment connecting $\Pi(F)$ and $\Pi(\oh_Y(-2))$ intersects with $b=-1$ inside $\widetilde{U}$, we have $\Hom(F,\oh_Y(-2)[1])=H^2(F)=0$. And by the $\mu_H$-stability we have $H^0(F)=0$, which implies $\chi(F)=\frac{c_3(F)+1}{2}<0$. However, since $F$ is reflexive and has rank two, we get $c_3(F)\geq 0$ by \cite[Proposition 2.6]{hartshorne:stable-reflexive-sheaves} \footnote{Although \cite[Proposition 2.6]{hartshorne:stable-reflexive-sheaves} only states for $\mathbb{P}^3$, it is well-known that it also works for any smooth projective threefold of Picard rank one.}, which makes a contradiction.

Now we assume that $s=-\frac{1}{2}$. Since there is no wall for $F$ intersects with $b=-1$ and the line segment connecting $\Pi(F)$ and $\Pi(\oh_Y(-2))$ intersects with $b=-1$ inside $\widetilde{U}$, we have $\Hom(F,\oh_Y(-2)[1])=H^2(F)=0$. Hence by $H^0(F)=0$, we see $\chi(F)=2t-\frac{2}{3}\leq 0$, which implies $t\leq \frac{1}{3}$.

Finally, when $s=-\frac{1}{2}$ and $t=\frac{1}{3}$, we know $F$ is reflexive. By $c_3(F)=0$, $F$ is locally free.
\end{proof}

\begin{lemma}\label{rk3_deg2}
    Let $F$ be a $\mu_H$-stable sheaf of class $\ch_{\leq 2}(F)=(3,-H,sH^2)$, then $s\leq -\frac{1}{2}$. When $s=-\frac{1}{2}$, we have $\ch_3(F)\leq -\frac{1}{6}H^3$. Moreover, $s=-\frac{1}{2}$ and  $\ch_3(F)=-\frac{1}{6}H^3$ if and only if $F=\cQ_Y=\bL_{\oh_Y}\oh_Y(1)[-1]$.
\end{lemma}

\begin{proof}
    We know $s\leq -\frac{1}{2}$ from Lemma \cite[Proposition 3.2]{li:fano-picard-number-one}. 
    When $s=-\frac{1}{2}$, since $\ch_1^{-\frac{1}{2}}(F)=\frac{1}{2}$, and the line segment connecting $\Pi(F)$ and $\Pi(\oh_Y(-2))$ intersects $b=-\frac{1}{2}$ inside $\widetilde{U}$, we know that $\Hom(F,\oh_Y(-2)[1]))=H^2(F)=0$. Since $H^0(F) = 0$ by the $\mu_H$-stability of $F$, we see $\chi(F)\leq 0$, which implies $\ch_3(F)\leq -\frac{1}{6}H^3$.

    Now assume that $s=-\frac{1}{2}$ and  $\ch_3(F)=-\frac{1}{6}H^3$. Then $F$ is reflexive by the previous results. Thus $F[1]$ is $\nu_{0,w}$-semistable for any $w>0$. Since the line segment connecting $\Pi(F)$ and $\Pi(\oh_Y(2))$ intersects with $b=0$ inside $\widetilde{U}$, we see $\Hom(\oh_Y(2),F[1])=\Hom(F,\oh_Y[2])=0$. Thus from $\chi(F,\oh_Y)=4$, we see $\hom(F,\oh_Y)\geq 4$. Pick four sections and consider the corresponding extension
    \begin{equation*}
        \oh_Y^{\oplus 4} \to G \to F[1]
    \end{equation*}
    Let $\ell$ be the line connecting $\Pi(F)$ and $\Pi(\oh_Y)$. We know $G$ is $\nu_{b,w}$-semistable for $(b,w) \in \ell \cap \widetilde{U}$ as $F[1]$ and $\oh_Y$ are $\nu_{b,w}$-stable of the same slope. Moreover, $\Hom(\oh_Y, F[1])=0$. Since $\ch(G) = \ch(\oh_Y(1))$, \cite[Proposition 4.20]{feyz:desing} implies that $G \cong \oh_Y(1)$. Thus $F \cong \cQ_Y$ as $h^0(G) = 4$ and $\Hom(\oh_Y, F[1])=0$. Note that the $\mu_H$-stability of $\cQ_Y$ follows from Lemma \ref{lem-Q-Y}.       
    
    
\end{proof}

\begin{lemma}\label{lem-serre-stability-d-2}
    Let $\sigma$ be a Serre-invariant stability condition on $\Ku(Y)$. Then $i^!\cQ_Y$ is $\sigma$-stable.
\end{lemma}
\begin{proof}
We can assume that $\sigma=\sigma(-\frac{1}{2},w)$ for some $\frac{1}{4}>w>0$. As $\ch^{-1}(\cQ_Y[1]) = \ch^{-1}(\oh_Y(-1)[1]) = \frac{1}{2}$ is minimal, both $\cQ_Y$ and $\oh_Y(-1)[1]$ are $\nu_{b=-\frac{1}{2}, w}$-stable for any $w > 0$. Then Lemma \ref{compare_stability} implies that $\cQ_Y[1],\oh_Y(-1)[2] \in \Coh^0_{b = -\frac{1}{2}, w}$ and both are $\sigma^0_{b,w}$-stable. Thus by the exact sequence \eqref{exact. i!QY}, $i^!\cQ_Y[1]\in \cA(-\frac{1}{2},w)$. Suppose for a contradiction that $i^!\cQ_Y[1]$ is not $\sigma(-\frac{1}{2},w)$-semistable, and let $F$ be the destabilizing quotient object of minimum slope. We can write the class $[F]=x\bv+y\bw$ for $x,y\in \ZZ$. Then by taking $w=\frac{5}{32}$, one can check the only integers $x,y$ satisfying 
\begin{equation*}
    \Im(Z^0_{-\frac{1}{2},w}(i^!\cQ_Y[1]))\geq \Im(Z^0_{-\frac{1}{2},w}(F))> 0
\end{equation*}
and 
\begin{equation}\label{slope orders}
    \mu^0_{-\frac{1}{2},w}(\cQ_Y[1])\leq \mu^0_{-\frac{1}{2},w}(F)<\mu^0_{-\frac{1}{2},w}(i^!\cQ_Y[1])
\end{equation}
are $(x,y)=(-1,1)$. The left-hand inequality in \eqref{slope orders} comes from the short exact sequence \eqref{exact. i!QY} and the fact that $\mu_{b=-\frac{1}{2}, w}^0(\cQ_Y[1]) < \mu_{b=-\frac{1}{2}, w}^0(\oh_Y(-1)[2])$ for any $w>0$. By \cite[Theorem 1.1]{pertusi:some-remarks-fano-threefolds-index-two}, we know that $F$ fits into a triangle $\oh_Y(-1)[1]\to F\to \oh_l(-1)$ for a line $l\subset Y$. However $\Hom(i^!\cQ_Y[1],F)=\Hom(i^!\cQ_Y[1],\oh_l(-1))=0$, which makes a contradiction.
\end{proof}

\begin{remark}
\label{gluing_object_not_stable_double_tilted_heart}
Note that $i^!\cQ_Y[1]$ is not stable in double tilted heart $\mathrm{Coh}^0_{b=-\frac{1}{2},w}$. In fact, it is destabilized by $\oh_Y(-1)[2]$. There is no wall in the $(b,w)$-plane which would make $i^!\cQ_Y[1]$ stable. The objects $E$ fitting in a triangle $\cQ_Y[1]\rightarrow E[1]\rightarrow\oh_Y(-1)[2]$ are obtained from triangle~\eqref{exact. i!QY} as all possible extensions in the other direction. This corresponds to a blow up at the point $[i^!\cQ_Y]$ in the Bridgeland moduli space $\mathcal{M}_{\sigma}(\Ku(Y),2\bv)$ of $\sigma$-stable objects of class $2\bv$ in $\Ku(Y)$ with the exceptional locus parametrizing those semistable sheaves of rank two, $c_1=0, c_2=2$ and $c_3=0$ not in $\Ku(Y)$. For more details, see Section~\ref{section_classification_instanton_sheaves_d=2}.
\end{remark}

\section{Moduli spaces on cubic threefolds}\label{section_Y3}

In this section, we always fix $Y$ to be a del Pezzo threefold of degree three, i.e. a cubic threefold. The goal of this section is to prove Proposition \ref{prop_Y3_classify} which classifies Bridgeland semistable objects of class $3\bv$ in $\Ku(Y)$ .

\bigskip

Consider the line $\ell_{d=3}$ as defined in section \ref{subsection.instanton} which passes through $\Pi(\oh_Y(-H))$ and $\Pi(\bv)$. It is of the equation 
    	\begin{equation*}
    	w = - \frac{5}{6} b - \frac{1}{3}.
    \end{equation*}  
   and intersects $\partial \widetilde{U}$ at two points with $b$-values $b_1 = -1$ and $b_2 =-\frac{2}{5}$.  
We know by Lemma \ref{no_wall_lem} that there is no wall for an object $E$ of class $\ch_{\leq 2}(E) = (3, 0, -H^2)$ between the large volume limit ($b<0$ and $w \gg 0$) and the line $\ell_3$. The following Proposition describes the objects which gets destabilised along the wall $\ell_3$.   
\begin{proposition}\label{porp.wall-v3}
		Take a point $(b, w) \in \ell_3 \cap U$ and let $E$ be a strictly $\nu_{b, w}$-semistable object of class $\ch_{\leq 2}(E)= (3, 0, -H^2)$ which is unstable in one side of the wall $\ell_3$. Then 
	the destabilising sequence is $E_1 \rightarrow E \rightarrow E_2$ where one of the factors $E_i$ is $\oh_Y(-H)[1]$ and the other one $E_j$ is a $\mu_H$-stable sheaf of class $\ch_{\leq 2}(E_j)=(4,-H,-\frac{1}{2}H^2)$. In particular, we have $\ch_3(E)\leq 0$.	
	\end{proposition}
	
\begin{proof}
Let $E_1 \rightarrow E \rightarrow E_2$ be a destabilising sequence along the wall. If the destabilising factors $E_1$ and $E_2$ are both sheaves, then $-\frac{2}{5} =b_2 \leq \mu_H(E_i)$ for $i=1, 2$. Moreover, the location of the wall implies that $\mu_H(E_i) \neq 0$. Thus $\ch_{\leq 1}(E_1) = (3, -H)$ up to relabeling the factors. Moreover $\ch_2(E_1) = -\frac{1}{6}H^2$ because $\Pi(E_1)$ lies on $\ell_3$. We know the wall $\ell_3$ passes through the vertical line $b =-\frac{1}{2}$ at a point inside $\widetilde{U}$, thus $E_1$ is $\nu_{b=-\frac{1}{2}, w}$-semistable for some $w >0$. This implies $E_1$ is $\nu_{b=-\frac{1}{2}, w}$-stable for any $w >0$ by \cite[Lemma 3.5]{feyz:slope-stability-of-restriction}, and so $E_1$ is a $\mu_H$-stable sheaf which is not possible by Lemma \ref{lem.no semistable-rank 3}. Thus $E_1$ or $E_2$ are not both sheaves.

Let $(r, cH) = \ch_{\leq 1}(\cH^{-1}(E_1)) + \ch_{\leq 1}(\cH^{-1}(E_2))$, then \eqref{ineq.slope} gives
	\begin{equation*}
	-\frac{2}{5}(r+3) \leq c \leq -r. 
	\end{equation*}
    Thus either $(r, c)$ is equal to $(2, -2)$ or $(1, -1)$.
    
    \textbf{Case I.} First assume $(r, c)$ is equal to $(2, -2)$. We know $\cH^{-1}(E_i)$ are torsion-free sheaves. They are even reflexive, otherwise there is a torsion sheaf $T$ supported in co-dimension at least 2 with embedding $T \hookrightarrow \cH^{-1}(E_i)[1] \hookrightarrow E_i$ in $\Coh^b(Y)$. This is not possible as $\nu_{b,w}$-slope of semistable factors $E_i$ 's are equal to $E$ which is not $+\infty$. Thus one of the following cases can happen: 
    \begin{itemize}
        \item[(a)] $\ch_{\leq 1}(\cH^{-1}(E_i)) = (1, -H)$ for $i=1, 2$, or
        \item[(b)] $\cH^{-1}(E_1) = 0$ and $\ch_{\leq 1}(\cH^{-1}(E_2)) = (2, -2H)$.  
    \end{itemize}
    On the other hand, we have
    \begin{equation*}
    \ch_{\leq 1}(\cH^0(E_1)) + \ch_{\leq 1}(\cH^0(E_2)) = (5, -2H). 
    \end{equation*}
    Since for $i=1, 2$,
    \begin{equation}\label{eq.slope min}
        \mu_H(\cH^0(E_i)) \geq \mu_H^-(\cH^0(E_i)) \geq -\frac{2}{5},
    \end{equation}
    the sheaf $\cH^0(E_i)$ is torsion supported in dimension at most 1 for either $i=1$ or $i=2$.
    
    In case (a), we have $\cH^{-1}(E_i) = \oh_Y(-H)$ for $i=1, 2$. 
    By relabelling the factors, we may assume $\ch^0(E_2)$ is a torsion sheaf. We know $\Pi(E_2)$ lies on the line $\ell_d$ and 
    \begin{align*}
        \ch_{\leq 2}(E_2) =\  & \ch_{\leq 2}(\cH^0(E_2)) - \ch_{\leq 2}(\cH^{-1}(E_2))\\
        = \ & \big(0, 0, \ch_2(\cH^0(E_2)) \big) - \left(1, -H, \frac{H^2}{2}\right). 
    \end{align*}
    This implies that $\ch_2(\cH^0(E_2)) = 0$, and so 
    \begin{equation*}
        \ch_2(\cH^0(E_1)) = \ch_2(\cH^{-1}(E_1)) + \ch_2(\cH^{-1}(E_2)) + \ch_2(E) =0
    \end{equation*}
    which implies $\ch_{\leq 2}(\cH^0(E_1)) = (5, -2H, 0)$. Thus $\Pi(\cH^0(E_1))$ lies on the boundary of $\widetilde{U}$ which is not possible by \cite[Proposition 3.2]{li:fano-picard-number-one} as \eqref{eq.slope min} implies that $\cH^0(E_1)$ is a $\mu_H$-stable sheaf.

    In case (b), we have $E_1 \cong \cH^0(E_1)$. Thus $\cH^0(E_1)$ cannot be supported in dimension 1, and so $\ch_{\leq 1}(E_1) = \ch_{\leq 1}(\cH^0(E_1)) = (5, -2H)$. Since $\Pi(E_1)$ lies on $\ell_d$, we have $\ch_2(E_1) = 0$ which is not again possible by the same argument as in case (a).


    \bigskip
    
    \textbf{Case II.} Now suppose $(r, c) = (1, -1)$, so by relabelling the factors, we may assume $\cH^{-1}(E_1) = 0$ and $\cH^{-1}(E_2) = \oh_Y(-H)$. Moreover,
    \begin{equation}\label{eq.sum ch}
    \ch_{\leq 2}(\cH^0(E_1)) + \ch_{\leq 2}(\cH^0(E_2)) = \left(4, -H , -\frac{1}{2}H^2\right). 
    \end{equation}
    Let $\ch_{\leq 2}(E_1) = (r_1, c_1H, s_1H^2)$. Since $\mu_H(\cH^0(E_i)) \geq -\frac{2}{5}$, we gain
    \begin{equation*}
        -\frac{2}{5} r_1 \leq\  c_1\ \leq -\frac{2}{5} r_1 + \frac{3}{5}.  
    \end{equation*}
    Thus $(r_1, c_1)$ is equal to $(0, 0)$, $(1, 0)$, $(3, -1)$, or $(4, -1)$. The first case cannot happen as torsion sheaves supported in dimension $\leq 1$ cannot make a wall. If $(r_1, c_1) = (1, 0)$, then since $\Pi(E_1)$ lies on $\ell_d$, we have $s_1 = -\frac{1}{3}$, thus $E_1$ has the same $\nu_{b,w}$-slope as $E$ with respect to any $(b, w)$, thus it cannot make a wall. If $(r_1, c_1) = (3, -1)$, then $s_1 =- \frac{1}{6}$. We know the wall $\ell_3$ passes through the vertical line $b =-\frac{1}{2}$ at a point inside $\widetilde{U}$, thus \cite[Lemma 3.5]{feyz:slope-stability-of-restriction} implies that $E_1$ is a $\mu_H$-stable sheaf which is not possible by Lemma \ref{lem.no semistable-rank 3}. Thus we have  
	 \begin{equation}\label{eq-sum ch}
    \ch_{\leq 2}(E_1)  = \left(4, -H , -\frac{1}{2}H^2\right),  
    \end{equation} 
    and $\cH^0(E_2)$ is a skyscraper sheaf. Then \cite[Proposition 4.20]{feyz:desing} implies that $E_2 \cong \oh_Y(-H)[1]$. Since $E_1$ is $\nu_{b,w}$-semistable on $\ell_3$, it is $\nu_{b=-\frac{1}{2}, w = \frac{1}{12}}$-semistable. Thus by Lemma \ref{lem.no semistable-rank 4}, $E_1$ is a $\mu_H$-stable reflexive sheaf. Finally, the last statement follows from Lemma \ref{lem.uniquness of K} that $\ch_3(E_1)\leq -\frac{1}{6}H^3$.
\end{proof}

\begin{lemma}\label{lem.no semistable-rank 3}
    There is no $\mu_H$-stable sheaf $E$ of class $\ch_{\leq 2}(E) = (3, H, sH^2)$ for $s \geq -\frac{1}{6}$. 
\end{lemma}
\begin{proof}
    Assume there is such a stable sheaf $E$. By replacing $E$ with its double dual, we may assume $E$ is a reflexive sheaf. Consider the line $\ell$ passing through $\Pi(E)$ and $\Pi(E(-2H))$ which is of equation \begin{equation*}
        w= -\frac{2}{3} b + \frac{s}{3} + \frac{2}{9}. 
    \end{equation*} 
    Since $s \geq -\frac{1}{6}$, it crosses the vertical lines $b= 0$ and $b=-\frac{3}{2}$ at points inside $\widetilde{U}$. Thus \cite[Lemma 3.5]{feyz:slope-stability-of-restriction} implies that both $E$ and $E(-2H)[1]$ are $\nu_{b,w}$-stable of the same slope for $(b,w) \in \ell \cap \widetilde{U}$. This implies $\hom(E, E(-2H)[1]) = \hom(E, E[2]) = 0$ which is a contradiction as $\hom(E, E) =1$ and $\chi(E, E) = 18s +6\geq 3$.
\end{proof}
\begin{lemma}\label{lem.no semistable-rank 4}
   Let $b_0 = -\frac{1}{2}$ and pick $w \geq \frac{1}{12}$ (note that the point $(b_0, \frac{1}{12}) \in \widetilde{U} \cap \ell_3$). There is no $\nu_{b_0, w}$-semistable object $E$ of class $\ch_{\leq 2}(E) = (4, -H, sH^2)$ for $s > -\frac{1}{2}$. 
  Moreover, if $s = -\frac{1}{2}$, then $\nu_{b_0, w}$-semistablility of $E$ at some $w \geq \frac{1}{12}$ implies that it is $\nu_{b_0, w}$-stable for any $w \geq \frac{1}{12}$. In particular, in this case, $E$ is a $\mu_H$-stable reflexive sheaf.
  \end{lemma}
  \begin{proof}
 Let $E$ be a $\nu_{b_0, w}$-semistable object of class $\ch_{\leq 1}(E) = (4, -H)$ such that $\ch_2(E)H \geq -\frac{H^3}{2}$. 
  We first claim $E$ is $\nu_{b_0, w}$-stable for any $w \geq \frac{1}{12}$. If not, there is a wall $\ell$ for $E$ passing through $\nu_{b_0, w}$ for some $w \geq \frac{1}{12}$. Let $E_1$ be a destabilising factor of class $(r_1, c_1H, s_1)$ such that $r_1 > 0$. We have
  \begin{equation*}
      0\ <\ \Im[Z_{b = -\frac{1}{2}, w_0}(E_1)] = c_1 + \frac{1}{2}r_1 \ <\ \Im[Z_{b = -\frac{1}{2}, w_0}(E_1)] =1. 
  \end{equation*}
  Thus $c_1 + \frac{1}{2}r_1 = \frac{1}{2}$. If $\frac{c_1}{r_1} < -\frac{2}{5}$, then position of the wall implies that $\Pi(E_1)$ lies in $\widetilde{U}$ which is not possible. Thus 
  \begin{equation*}
     -\frac{2}{5} \leq  \frac{c_1}{r_1} = -\frac{1}{2} + \frac{1}{2r_1}
  \end{equation*}
  which implies $(r_1, c_1)$ is equal to $(3, -1)$, or $(5, -2)$. We know $\Pi(E_1)$ lies above or on the line $\ell_3$. Thus the first cannot happen by Lemma \ref{lem.no semistable-rank 3}. In the latter, $s_1= 0$ and $\Pi(E_1)$ lies on the boundary $\partial \widetilde{U}$ which is not again possible by \cite[Proposition 3.2]{li:fano-picard-number-one}. Therefore, $E$ is $\nu_{b_0, w}$-stable for $w \geq \frac{1}{12}$ and so a $\mu_H$-stable sheaf.

  To complete the proof, we only need to show that we cannot have $s > -\frac{1}{2}$. Assume otherwise, then we may assume $E$ is a reflexive sheaf, so $E(-2H)[1]$ is $\nu_{b, w}$-stable for $b > -\frac{9}{4}$ and $w \gg 0$. Since $s \in \frac{1}{6}\mathbb{Z}$, we have $s \geq -\frac{1}{3}$. We know there is no wall for $E(-2H)[1]$ crossing the vertical line $b = -2$ for $w > 2$. Thus one can check that $E$ and $E(-2H)[1]$ are $\nu_{b, w}$-stable of the same phase for $(b, w) \in \ell \cap U$ where $\ell$ is the line passing through $\Pi(E)$ and $\Pi(E(-2H))$. Hence, $\hom(E, E[2]) =0$ but $\chi(E, E) \geq 5$, a contradiction. 
  \end{proof}
  
\begin{lemma}\label{lem.uniquness of K}
  	Let $E$ be a $\mu_H$-stable sheaf on $Y$ of class 
  	\begin{equation*}
  	\ch(E) = \left(4, -H, -\frac{1}{2}H^2 , sH^3\right).
  	\end{equation*}
  	Then $s\leq -\frac{1}{6}$. Moreover $s=-\frac{1}{6}$ if and only if $E\cong \cQ_Y=\bL_{\oh_Y}\oh_Y(1)[-1]$.
  \end{lemma}
  \begin{proof}
  By $\mu_H$-stability of $E$, we have $\Hom(\oh_Y, E)= 0 = \Hom(\oh_Y, E[3])=\Hom(E,\oh_Y(-2))$. And since the line segment connecting $\Pi(E)$ and $\Pi(\oh_Y(-2))$ intersects $b=-\frac{1}{2}$ at a point with $w > \frac{1}{12}$, by Lemma \ref{lem.no semistable-rank 4} we have $ 0 = \Hom(E,\oh_Y(-2)[1])=\Hom(\oh_Y, E[2])$, which gives $\chi(E)=-\hom^1(\oh_Y, E)\leq 0$ and $s\leq -\frac{1}{6}$.


  Now assume that $s=-\frac{1}{6}$. Then $E$ is reflexive by Lemma \ref{lem.no semistable-rank 4} and the previous result. Thus its shift $E[1]$ is $\nu_{b, w}$-stable for $b > -\frac{1}{4}$ and $w \gg 0$. We know there is no wall for $E[1]$ passing through the vertical line $b = 0$. Therefore $\hom(E, \oh_Y[2]) = \hom(\oh_Y(2H), E[1]) = 0$ and so
  \begin{equation*}
      \hom(E, \oh_Y) \geq \chi(E, \oh_Y) = 5
  \end{equation*}
  Hence the first wall $\ell$ for $E[1]$ will be made by $\oh_Y[1]$. Pick five linearly independent elements from $\Hom(E, \oh_Y)$, and let $G$ be the kernel of the evaluation map in the abelian category of $\nu_{b, w}$-semistable objects of the same slope as $E[1]$ and $\oh_Y[1]$ for $(b, w) \in \ell \cap U$:
  \begin{equation*}
      G \hookrightarrow E[1] \twoheadrightarrow \oh_Y^{\oplus 5}[1]. 
  \end{equation*}
  We know $\ch(G) = \ch(\oh_Y(1))$, so $G \cong \oh_Y(1)$ by \cite[Proposition 4.20]{feyz:desing} and the claim follows.
  \end{proof}

Finally, we can describe Bridgeland stable objects with class $3\bv$ in $\Ku(Y)$.
\begin{proposition} \label{prop_Y3_classify}
Let $\sigma$ be a Serre-invariant stability condition on $\Ku(Y)$ and $E\in \Ku(Y)$ be a $\sigma$-(semi)stable object of class $3\bv$. Then up to a shift, $E$ is either a Gieseker-(semi)stable sheaf or $i^!\cQ_Y$.
\end{proposition}

\begin{proof}
By the uniqueness of Serre-invariant stability conditions on $\Ku(Y)$, we can take $\sigma=\sigma(b_0, w_0)$, where $(b_0, w_0)=(-\frac{5}{6}, \frac{13}{36})$. And we can assume $E\in \cA(b_0,w_0)$ of class $-3\bv$. We have chosen the point $(b_0, w_0) \in V$ so that $\mu^0_{b_0, w_0}(-3\bv)=+\infty$. Thus $E$ is $\sigma^0_{b_0, w_0}$-semistable, then Lemma \ref{compare_stability} implies that $E$ lies in the exact triangle 
\begin{equation*}
    F[1] \to E \to T
\end{equation*}
where $F \in \Coh^{b_0}(Y)$ is $\nu_{b_0, w_0}$-semistable and $T \in \Coh_0(X)$. So we have $\ch(F) = 3\bv + \ch(T)$. As the point $(b_0, w_0)$ lies on $\ell_3$, either (i) $F$ is strictly $\nu_{b_0, w_0}$-semistable and unstable above the wall $\ell_3$, or (ii) it is semistable above the line $\ell_3$ and so it's a large volume limit semistable sheaf by Lemma \ref{no_wall_lem}. 

In case (i), Proposition \ref{porp.wall-v3} implies that $\ch_3(F) \leq 0$ and so $T=0$. Also combining it with Lemma \ref{lem.uniquness of K} implies that $E [-1] =F$ lies in the non-trivial exact sequence
$$\oh_Y(-1)[1] \to E[-1] \to \cQ_Y.$$
Since $\Hom(\cQ_Y, \oh_Y(-1)[2]) = 1$ by \eqref{ext-1}, we get $E = i^{!}\cQ_Y[1]$. 

In case (ii), Lemma \ref{no_wall_lem} shows that $F$ is large volume limit semistable and $\ch_3(F) \leq 0$, so $T=0$. Hence $E[-1] = F$ is a Gieseker-semistable sheaf by Lemma \ref{Gstability_lvls_coincide}. 



\end{proof}

\section{Brill--Noether reconstruction 
}\label{Brill_Noether_reconstruction_dP}
Let $Y \coloneqq Y_d$ be a del Pezzo threefold of Picard rank one of degree $d \geq 2$. 
In this section, we prove Theorem~\ref{main_theorem_BN_reconstruction_del_Pezzo_3fold} in the introduction in Theorem \ref{Brill--Noether reconstruction}.

Let $\oh_p$ be the skyscraper sheaf at any point $p \in Y$. We know $\bL_{\oh_Y(1)}\oh_p\cong \cI_p(1)[1]$, and so  
\begin{equation}\label{o-p}
    i^*\oh_p \cong\bL_{\oh_Y}(\cI_p(1))[1]. 
\end{equation}
We have $\ch(i^*\oh_p) = \left(d, -H, -\frac{1}{2}H^2 , (\frac{1}{d} - \frac{1}{6})H^3 \right) = d\bv-\bw$. The following proposition characterises stable objects in $\Ku(Y)$ of class $d\bv-\bw$. 
\begin{proposition}[{\cite{rota:moduli-space-index-two}}] \label{prop_moduli_w}
    Let $F\in \Ku(Y)$ be a $\sigma$-stable object of class $d\bv -\bw$ for a Serre-invariant stability condition $\sigma$. Then up to a shift, $F$ is either isomorphic to 
    $i^*\oh_p$ for a point $p \in Y$, or it is of the form $\bO\left(j_* T\right)$ where $T$ is a  Gieseker-stable reflexive sheaf supported on a hyperplane section $j \colon S \hookrightarrow Y$.
    This induces a well-defined map 
    \begin{align}\label{psi}
    \Psi \colon Y &\hookrightarrow\mathcal{M}_{\sigma}(\mathcal{K}u(Y), d\bv-\bw)\\
    p & \mapsto i^*\oh_p \nonumber
\end{align}
    which gives an embedding of $Y$ into the moduli space $\cM_{\sigma}(\mathcal{K}u(Y), d\bv-\bw)$ as a smooth subvariety.
\end{proposition}

\begin{proof}
Since all stability conditions $\sigma(b,w)$ for $(b,w) \in V$ lie in the same orbit with respect to the action of $\widetilde{\mathrm{GL}}^+_2(\mathbb{R})$ and they are $\bO$-invariant, we can consider $\sigma\left(-\frac{1}{2}, w_0\right)$ where $\left(b=-\frac{1}{2}, w_0\right)\in V$, and characterise $\sigma\left(-\frac{1}{2}, w_0\right)$-stable objects of class $\bO^{-1}(d\bv-\bw) = \bw$.

Take a $\sigma\left(-\frac{1}{2}, w_0\right)$-stable object $E\in \cA\left(-\frac{1}{2},w\right)$ of class $-\bw$. Since $\mu^0_{-\frac{1}{2},w}(E)=+\infty$, we know $E$ is $\sigma^0_{-\frac{1}{2},w}$-semistable. Then by \cite[Lemma 4.15]{rota:moduli-space-index-two}, $E[-1]$ is $\nu_{-\frac{1}{2},w_0}$-semistable. By the proof of \cite[Proposition 4.7]{rota:moduli-space-index-two}, the only wall for $E[-1]$ intersecting $b=-\frac{1}{2}$ is the line $\ell$ passing through $\Pi(\oh_Y(-1))$ of slope $-\frac{1}{2}$. Thus when we move up from the point $\left(-\frac{1}{2}, w_0\right)$ along the line $b=-\frac{1}{2}$, either 
\begin{enumerate}
    \item [(i)] $E[-1]$ is $\nu_{b=-\frac{1}{2}, w}$-semistable for all $w \gg 0$, i.e. it is a Gieseker-stable sheaf, or 
    \item [(ii)] $E[-1]$ gets destabilised along the wall $\ell$.
\end{enumerate}
In case (ii), the destabilizing sequence is of form $A\to E[-1]\to B$, where $\ch_{\leq 2}(B)=\ch_{\leq 2}(\oh_Y)$ as in the proof of \cite[Proposition 4.7]{rota:moduli-space-index-two}. Hence $\ch_{\leq 2}(A)=\ch_{\leq 2}(\oh_Y(-1)[1])$. Since $\Delta_H(A)=\Delta_H(B)=0$, $A$ and $B$ are $\nu_{-\frac{1}{2},w}$-semistable for any $w$. This proves $A=\oh_Y(-1)[1]$ and $B=\cI_p$ for a point $p\in Y$. Thus $E[-1] = E_p$ where $E_p$ is the unique extension
\begin{equation}\label{E-p}
    \oh_Y(-1)[1]\to E_p\to \cI_p\,.
\end{equation}
Thus $\bO(E[-1]) = \bO(E_p)\cong i^*\oh_p$ as claimed. Hence $\Psi$ is a well-defined map which is the composition of the embedding $Y\hookrightarrow \cM_{\sigma}(\Ku(Y),-\bw)$ given in \cite[Lemma 4.8]{rota:moduli-space-index-two} (which sends $p \in Y$ to $E_p$), and the isomorphism $\cM_{\sigma}(\Ku(Y),-\bw)\to \cM_{\sigma}(\Ku(Y),d\bv-\bw)$ given by $\bO$. In particular, $\Psi$ is an embedding. 
\end{proof}
Note that although in \cite{rota:moduli-space-index-two}, $Y$ is assumed to be general, the above results hold for any smooth Fano threefold $Y$ of index 2 and degree $d$. Their aim for the generality assumption is to get an explicit description for Gieseker-stable sheaves of class $\bw$ using roots on del Pezzo surfaces, which we do not need in this paper.

\begin{theorem}[Brill--Noether reconstruction for del Pezzo threefolds]
\label{Brill--Noether reconstruction}
Let $\sigma$ be a Serre-invariant stability condition on $\Ku(Y)$. Then the map $\Psi$ defined in \eqref{psi} induces an isomorphism between $Y$ and the \emph{Brill--Noether locus} 
$$\mathcal{BN}_Y\coloneqq \{F\in\mathcal{M}_{\sigma}(\mathcal{K}u(Y), [i^*\oh_p]) \, \colon \ \dim_{\mathbb{C}}\Hom(F,\ i^{!}\QY) \geq d+1 
\}
$$
where $\oh_p$ is the skyscraper sheaf supported at a point $p\in Y$. 
\end{theorem}
\begin{proof}
    Recall that $\cQ_Y \coloneqq \bL_{\oh_Y}\oh_Y(1)[-1]$ as defined in \eqref{eq. Q} which is a vector bundle when $d\geq 2$. By adjunction of $i^*$ and $i^!$, we have $\RHom(F,\ i^{!}\QY)=\RHom(F,\QY)$. Up to a shift, by Proposition \ref{prop_moduli_w}, we can assume $F$ is either (i) isomorphic to $i^*\oh_p$ for a point $p\in Y$, or (ii) of the form 
    $\bO(j_*T)$ where $T$ is a Gieseker-stable sheaf supported on a hyperplane section $j \colon S \hookrightarrow Y$. 

    In case (i), since $\RHom(\oh_Y, \cQ_Y)=0$, by \eqref{o-p}, we only need to compute $\RHom(\cI_p(1), \cQ_Y)$. Since $\cQ_Y$ is a bundle of rank $d+1$, we get $\RHom(\oh_p,\cQ_Y)=\mathbb{C}^{d+1}[-3]$. Now applying $\Hom(-,\cQ_Y)$ to the exact sequence $0\to \cI_p(1)\to \oh_Y(1)\to \oh_p\to 0$, since $\RHom(\oh_Y(1),\cQ_Y)=\mathbb{C}[-1]$, we see $\RHom(\cI_p(1),\cQ_Y)=\mathbb{C}[-1]\oplus \mathbb{C}^{d+1}[-2]$. Hence there exists $k \in \ZZ$, so that $\Psi(p)[k]\in \mathcal{BN}_Y$ for any point $p \in Y$.  
    \vspace{.2 cm}

    In case (ii), by definition of the rotation functor $\bO$ in \eqref{rotation}, we only need to compute $\RHom(j_*T(1), \cQ_Y)$ as $\RHom(\oh_Y, \cQ_Y)=0$. Clearly $\Hom(j_*T(1),\cQ_Y)=0$ and 
    \begin{equation}\label{equality}
        \hom\left(j_*T(1),\cQ_Y[k]\right)=\hom\left(\cQ_Y,j_*T(-1)[3-k]\right)=\hom_S\left(\cQ_Y|_S,T(-1)[3-k]\right).
    \end{equation}
Now we apply next Lemma \ref{lem_restriction_semistable} to show that the above $\Hom$-spaces vanish for $k=3, 1$, so we get $\RHom(j_*T(1), \cQ_Y)=\mathbb{C}^d[-2]$ as $\chi(j_*T(1), \cQ_Y)=d$. 

\textbf{$k=3$:} Since $S \in |H|$ is irreducible, Lemma \ref{lem_restriction_semistable} implies that both $j_*\oh_S$ and $j_*\cQ_S$ are 2-Gieseker semistable of classes
\begin{equation*}
    \ch(j_*\oh_S) = \left(0,\ H,\ -\frac{H^2}{2}, \ \frac{H^3}{6} \right) \qquad \text{and} \qquad \ch_{\leq 2}(j_*\cQ_S) = \left(0,\ (d+1)H,\ -\frac{d+3}{2}H^2 \right). 
\end{equation*}
Since $\ch_{\leq 2}(j_*T(-1)) = \left(0, H, -\frac{3}{2}H^2\right)$, comparing slopes implies that $$\Hom(j_*\oh_S, j_*T(-1)) = 0 = \Hom(j_*\cQ_S, j_*T(-1)). $$
Thus the short exact  sequence \eqref{short-exact} implies that  
$\Hom(j_*\cQ_Y|_S, j_*T(-1)) = 0$. 

\textbf{$k=1$:} By Serre-duality on $S$, we know $\hom_S(\cQ_Y|_S,T(-1)[2])=\hom_S(T,\cQ_Y|_S)$ which vanishes as 
$$\Hom(j_*T, j_*\oh_S) = 0 = \Hom(j_*T, j_*\cQ_S) $$
by comparing slopes. 

Totally we get $j_*T \notin \mathcal{BN}_Y$ and so $\Psi(Y)=\mathcal{BN}_Y$, then the claim follows from Proposition 
\ref{prop_moduli_w}.

\end{proof}

\begin{remark}
 The proof of Theorem \ref{Brill--Noether reconstruction} also shows that $\mathcal{BN}_Y$ can be  written as
 $$\mathcal{BN}_Y= \{F\in\mathcal{M}_{\sigma}(\mathcal{K}u(Y), [i^*\oh_p]) \, \colon \ \RHom(F,\ i^{!}\QY) \text{ is a two-term complex}
\}.
$$
\end{remark}


\begin{lemma} \label{lem_restriction_semistable}
Let $Y$ be a del Pezzo threefold of Picard rank one of degree $d\geq 2$, and let $S\hookrightarrow Y$ be a hyperplane section. Then $\cQ_Y|_S$ fits into an exact sequence
\begin{equation}\label{short-exact}
    0\to \oh_S\to \cQ_Y|_S\to \cQ_S\to 0,
\end{equation}
where $\cQ_S:=\bL_{\oh_S}\oh_S(1)[-1] \in \Coh(S)$ is a $H|_S$-$\mu_H$-semistable bundle on $S$.
\end{lemma}

\begin{proof}
By the restriction of the exact sequence \eqref{eq. Q}, we get the exact sequence  
$$0\to \cQ_Y|_S\to \oh_S^{\oplus d+2}\to \oh_S(1)\to 0$$ 
on $S$. This gives $\RHom_S(\oh_S,\cQ_Y|_S)=\mathbb{C}$. Take a non-zero section $s\colon \oh_S\to \cQ_Y|_S$, then we get the following commutative diagram with exact rows
\[\begin{tikzcd}
	0 & {\oh_S} & {\oh_S} & 0 \\
	0 & {\cQ_Y|_S} & {\oh_S^{\oplus d+2}} & {\oh_S(1)} & 0.
	\arrow[shift right=1, no head, from=1-2, to=1-3]
	\arrow[from=2-1, to=2-2]
	\arrow[from=2-2, to=2-3]
	\arrow[from=2-3, to=2-4]
	\arrow[from=2-4, to=2-5]
	\arrow[from=1-1, to=1-2]
	\arrow[from=1-3, to=1-4]
	\arrow[from=1-4, to=2-4]
	\arrow[hook, from=1-3, to=2-3]
	\arrow["s", hook, from=1-2, to=2-2]
	\arrow[no head, from=1-2, to=1-3]
\end{tikzcd}\]
By taking the cokernel, we get an exact sequence 
\begin{equation}\label{ex.11}
    0\to \cok(s)\to \oh_S^{\oplus d+1}\to \oh_S(1)\to 0.
\end{equation}
This implies $\cok(s)\cong \cQ_S$ as $\RHom_S(\oh_S, \cok(s))=0$. To complete the proof, we only need to show $\cQ_S$ is $\mu_{H|_S}$-semistable. Assume otherwise, and let $F$ be a destabilising subsheaf. We may assume $F$ is $\mu_{H|_S}$-stable. Then the exact sequence \eqref{ex.11} implies that 
\begin{equation*}
    -\frac{1}{d} = \mu_{H|_S}(\cQ_S) < \mu_{H|_S}(F) \leq 
    \mu_{H|_S}(\oh_S) =0.
\end{equation*}
Since $\rk(F) < d$, we must have $\mu_{H|_s}(F) = 0$. 
We can assume that $F$ is saturated in $\cQ_S$, hence is saturated in $\oh_S^{d+1}$ as well. By the uniqueness of Jordan--H\"older factors, we get $F\cong \oh_S^{\oplus \rk F}$. Thus $\Hom_S(\oh_S, \cQ_S)\neq 0$, which contradicts the construction of $\cQ_S$.
\end{proof}

\subsection{Classical moduli spaces on curves and Brill--Noether reconstruction}\label{classical_d=4}
Let $Y$ be a smooth degree $4$ del Pezzo threefold, which is the intersection of two quadrics in $\mathbb{P}^5$. There is an FM equivalence $\Phi_{\mathcal{S}} \colon D^b(C)\ \xrightarrow{\cong} \Ku(Y)$ for a genus two curve $C$. Denote by $M_C(2,\mathcal{L}_1)$ the moduli space of stable vector bundle of rank two with fixed determinant $\mathcal{L}_1$ such that degree $d(\mathcal{L}_1)=1$. By \cite[Theorem 1]{newstead1968stable} we know 
\begin{equation}\label{one side}
    Y\cong M_C(2,\mathcal{L}_1)
\end{equation} 
Note that $\mathcal{S}$ is the universal spinor bundle on $C \times Y$. On the other hand, by Theorem~\ref{Brill--Noether reconstruction} and action of inverse of the rotation functor $\bO$, we get   
\begin{equation*}
   Y\cong \bO^{-1}(\mathcal{BN}_Y) = \{E\in\mathcal{M}_{\sigma}(\Ku(Y),\bw)\, \colon \ \dim_{\mathbb{C}}\Hom(F,\ i^{!}\oh_Y) \geq 5\}. 
\end{equation*}
 By \cite[Lemma 5.9]{Kuznetsov:instanton-bundle-Fano-threefold}, $\Phi_{\mathcal{S}}^{-1}(i^!\oh_Y) \cong\mathcal{R}[1]$ where $\mathcal{R}$ is a \emph{second Raynaud bundle}, which is a semistable vector bundle of rank $4$ and degree $4$ on $C$. Moreover, it is unique up to a twist by a line bundle of degree $0$, see \cite[Section 5.4]{Kuznetsov:instanton-bundle-Fano-threefold}. By \cite[Section 5.2]{rota:moduli-space-index-two}, the equivalence $\Phi$ sends the Bridgeland moduli space $\mathcal{M}_{\sigma}(\Ku(Y),\bw)$ to $M_C(2,1)$. Thus
 \begin{align}\label{second side}
     Y\cong &\ \{F\in M_C(2,1)\, \colon \ \dim_{\mathbb{C}}\Hom(F,\ \mathcal{R}[1]) \geq 5\}\nonumber\\
     \cong &\ \{F\in M_C(2,1)\, \colon \ \dim_{\mathbb{C}}\Hom(F,\ \mathcal{R}) \geq 1\}
 \end{align}
 as $\chi(F, \mathcal{R}) = -4$. Comparing \eqref{one side} and \eqref{second side} gives the impression that fixing determinant of $F \in M_C(2, 1)$ is equivalent to imposing the \emph{Brill--Noether} condition. 
 \bigskip

Let $J(Y)$ be the intermediate Jacobian of $Y$. As in \cite[Section 4.4 \& Section 5.2]{rota:moduli-space-index-two}, we consider the map 
\begin{align*}
    \Psi \colon \mathcal{M}_{\sigma}(\Ku(Y),\bw) &\ \to \ J(Y)\\
    E & \ \mapsto \  \widetilde{c}_2(E) -H^2
\end{align*}
where $\widetilde{c}_2(E)$ is the second Chern class of $E$ up to rational equivalence. We know $\Psi(E_p) = 0$ and we know $\Psi^{-1}(0)$ is isomorphic to $Y$, thus $\Psi(\bO(T)) \neq 0$ where $T$ is a Gieseker-stable sheaf supported on a hyperplane $S$.

By \cite[Section 5.2]{rota:moduli-space-index-two} $\Psi^{-1}(0)\cong Y\subset\mathcal{M}_{\sigma}(\Ku(Y),\bw)$ such that $Y\cong\{E_p,p\in Y\}$(See \cite[Proposition 4.7]{rota:moduli-space-index-two} for definition of $E_p$). Then $Y\cong\bO^{-1}(\mathcal{BN}_Y)\cong\mathcal{BN}_Y$. 

There is an equivalence $\Phi_1 \colon \Pic^1(C) \to J(Y)$ so that $\Phi_1(\mathcal{L}_1) = 0$ and it induces the commutative diagram \cite[Theorem 4.14(c’)]{Ried:thesis}
\[
\xymatrix{
M_C(2, 1) \ar[r]^{\det} \ar[d]^{\Phi_{\mathcal{S}}} & \Pic^1(C) \ar[d]^{\Phi_1} \\
\mathcal{M}_{\sigma}(\Ku(Y),\bw) \ar[r]^{\Psi} & J(Y).
}
\]
This shows that we have an isomorphism 
\begin{equation*}
    M_C(2, \mathcal{L}_1) \cong\ \text{det}^{-1}(\mathcal{L}_1) \cong \Psi^{-1}(0) \ \cong \mathcal{BN}_Y.
\end{equation*}


\section{Uniqueness of the gluing object}\label{section_categorical_Torelli_theorem}

In this section, we prove the following Theorem. 


\begin{theorem}\label{prop_gluing_data_fixed}
 Let $\Phi \colon \Ku(Y)\simeq\Ku(Y')$ be an exact equivalence of Kuznetsov components of del Pezzo threefolds of the same degree $d$ where $2\leq d\leq 4$.  
\begin{enumerate}
    \item [(i)] If $d =2, 3$, there exist a unique pair of integers $m_1, m_2 \in \ZZ$ with $0\leq m_1\leq 3$ when $d=2$ and $0\leq m_1\leq 5$ when $d=3$, so that    $$\Phi(i^!\QY)\cong\bO^{m_1}({i'}^!\mathcal{Q}_{Y'})[m_2].$$
    \item [(ii)] If $d=4$, there exists a unique pair of integers $m_1,m_2$ and a unique auto-equivalence $T_{\mathcal{L}_0} \in \mathrm{Aut}^0(\Ku(Y'))$ (see Section \ref{degree 4-subsection} for definition) so that 
 $$\Phi(i^!\cQ_Y)\cong \bO^{m_1}\circ T_{\mathcal{L}_0}({i'}^!\mathcal{Q}_{Y'})[m_2]. $$
\end{enumerate}
 Here $i' \colon \Ku(Y') \hookrightarrow \D^b(Y')$ is the inclusion functor.
\end{theorem}

\begin{remark} \label{fix_i!O}
Theorem \ref{prop_gluing_data_fixed} also holds if we replace $i^!\cQ_Y$ and $i'^!\cQ_{Y'}$ by $i^!\oh_Y$ and $i'^!\oh_{Y'}$, respectively. The reason is that $\bO(i^!\oh_Y)\cong i^!\cQ_Y$ and the proof only uses the properties of Bridgeland moduli spaces with respect to Serre-invariant stability conditions and objects in them, which are all preserved by $\bO$. 
\end{remark}

\begin{remark}\label{rmk_trivial}
 The proof of Theorem \ref{prop_gluing_data_fixed} actually shows that if $\Phi$ maps $\bv$ and $\bw$ to $\bv'$ and $\bw'$ respectively, then $\Phi(i^!\cQ_Y)=i'^!\cQ_{Y'}$ up to shift and action of $T_{\mathcal{L}_0}$ (when $d=4$).
\end{remark}

We first discuss the action of equivalences on the numerical Grothendieck groups, and then investigate each degree separately. 

\begin{lemma} \label{numerical_lemma}
    Let $Y$ and $Y'$ be two del Pezzo threefolds of Picard rank ones of degree $d$ and $\Phi:\Ku(Y)\to \Ku(Y')$ an equivalence. Let $\phi:\mathcal{N}(\Ku(Y))\to \mathcal{N}(\Ku(Y'))$ be the induced isometry. Then 
    \begin{enumerate}
        \item[(a)] If $\phi(m\bv)=m\bv'$ for a non-zero integer $m$, then $\phi(\bv)=\bv'$ and $\phi(\bw)=\bw'$.
        \item[(b)] Up to composing with $\bO$ and $[1]$, $\phi$ maps classes $\bv$ and $\bw$ to $\bv'$ and $\bw'$, respectively.
    \end{enumerate}

\end{lemma}

\begin{proof}

Recall that the numerical Grothendieck group $\mathcal{N}(\Ku(Y'))$ has no torsion. In part (a), from $\phi(m\bv)=m\bv'$ we have $\phi(\bv)=\bv'$. Now we assume that $\phi(\bw)=a\bv'+b\bw'$ for $a,b\in \ZZ$. Using $\chi(\bv,\bw)=-1$ and $\chi(\bw,\bv)=1-d$, we get $\chi(\bv',a\bv'+b\bw')=-1$ and $\chi(a\bv'+b\bw',\bv')=1-d$. Thus we obtain $-a-b=-1$ and $-a+(1-d)b=1-d$, which gives $(a,b)=(0,1)$ when $d\neq 2$. When $d=2$, using $\chi(\bw,\bw)=\chi(a\bv'+b\bw',a\bv'+b\bw')=-d$, we obtain $(a,b)=(0,1)$ or $(2,-1)$. We claim the latter cannot happen, otherwise  
\begin{equation*}
 \phi(\bv)= \bv' \qquad \text{and} \qquad \phi(\bv-\bw)= -(\bv'-\bw').    
\end{equation*}
For any line $l\subset Y$, we define $\cJ_l:=\bO^{-1}(\cI_l)[1]\in \Ku(Y)$ as in \cite{pertusi:some-remarks-fano-threefolds-index-two}. Fix two lines $l_1, l_2\subset Y$ such that $l_1\cap l_2\neq \varnothing$. Then by \cite[Remark 4.8]{pertusi:some-remarks-fano-threefolds-index-two}, we have $\Hom(\cI_{l_1}, \cJ_{l_2})\neq 0$. Since $\chi(\cI_{l_1}, \cJ_{l_2})=0$ and $\Hom(\cI_{l_1}, \cJ_{l_2}[n])=0$ when $n\leq -1$ and $n\geq 2$, we get $\Hom(\cI_{l_1}, \cJ_{l_2}[1])\neq 0$.

Let $\sigma$ be a Serre-invariant stability condition on $\Ku(Y)$, then by \cite[Theorem 1.1]{pertusi:some-remarks-fano-threefolds-index-two} any $\sigma$-stable object of class $[\cI_{l}]$ in $\Ku(Y)$ is the shifted ideal sheaf $\cI_{l'}[k]$ for some line $l'$ on $Y$. The same claim also holds for objects of class $[\cJ_l]=-[\bO^{-1}(\cI_l)]$ as $\sigma$ is $\bO$-invariant. Recall that there is a unique Serre-invariant stability condition on $\Ku(Y)$ up to $\widetilde{\mathrm{GL}}^+_2(\mathbb{R})$-action. Since $\Phi$ commutes with the Serre functor, $\Phi. \sigma$ is also a Serre-invariant stability condition on $\Ku(Y')$. Thus up to a shift, we can assume that $\Phi(\cI_{l_1})=\cI_{l_1'}$ and $\Phi(\cJ_{l_2})=\cJ_{l_2'}[k]$ for lines $l_1',l_2'\subset Y'$ and an odd integer $k$. Thus we get $\Hom(\cI_{l_1'}, \cJ_{l_2'}[k])=\Hom(\cI_{l_1'}, \cJ_{l_2'}[1+k])\neq 0$. This implies $k=0$ and makes a contradiction which completes the proof of part (a). 

For part (b), we claim that up to composing with $\bO$ and $[1]$, $\bv$ maps to $\bv'$. Indeed, the image of $\bv$ is still a $(-1)$-class in $\mathcal{N}(\Ku(Y))$ since $\Phi$ is an equivalence. Then the claim for $d\geq 3$ follows from \cite[Corollary 4.2]{liu-zhang:moduli-space-cubic-x14}. And up to sign, a $(-1)$-class is either $\bv'$ or $\bv'-\bw'$ for $d=2$, and $\bv',\bw'$ or $\bv'-\bw'$ for $d=1$. They are permuted by rotation functor $\bO$ and the claim follows. Thus the result follows from part (a) and the claim above.
\end{proof}

\subsection{Degree 2 case}\label{degree 2-subsection} We first consider a del Pezzo threefold $Y$ of degree $2$ which is a quartic double solid. It is a double cover $\pi: Y\to \mathbb{P}^3$ which is ramified over a smooth surface $R\subset \PP^3$ of degree $4$. The branch divisor of $\pi$ maps isomorphic to $R$, which we also denote by $R\subset Y$. The involution on $Y$ given by the double cover is denoted by $\tau$. The Serre functor of $\Ku(Y)$ is  $S_{\Ku(Y)}=\tau[2]$. Moreover we have $\oh_Y(R)=\oh_Y(2)$. The key idea to prove Theorem \ref{prop_gluing_data_fixed} is to investigate the singular locus of a suitable moduli space in $\Ku(Y)$.


\begin{lemma}\label{singular_locus}
Let $\sigma$ be a Serre-invariant stability condition on $\Ku(Y)$. Then the singular locus of the moduli space $\cM_{\sigma}(\Ku(Y), 2\bv-\bw)$ is at least two dimensional, consists of objects of form $i^*\oh_p$ such that $p\in R$, and $\bO(j_*F)$ where $j\colon S\hookrightarrow Y$ is a hyperplane section and $F$ is a reflexive sheaf on $S$ with $\tau(j_*F)\cong j_*F$.
\end{lemma}

\begin{proof}
Since $\sigma$ is $\bO$-invariant, the functor $\bO$ makes an isomorphism $\cM_{\sigma}(\Ku(Y), -\bw) \cong \cM_{\sigma}(\Ku(Y), 2\bv-\bw))$. Thus for any $F\in\cM_{\sigma}(\Ku(Y), 2\bv-\bw)$, there exists $E\in \cM_{\sigma}(\Ku(Y), -\bw)$ so that $F=\bO(E)$. Since $\RHom(F,F)=\RHom(E,E)$, we only need to consider the smoothness of $[E]$ in $\cM_{\sigma}(\Ku(Y), -\bw)$. By Proposition \ref{prop_moduli_w} and its proof, there are two possibilities: 
\newline
Case (i). $E = E_p$ for a point $p \in Y$ as defined in \eqref{E-p}. Since $\tau(E_p)=E_{\tau(p)}$, we know that $[E]$ is a singular point if and only if $\Ext^2(E_p, E_p)=\Hom(E_p,E_{\tau(p)})\neq 0$, which is equivalent to $p=\tau(p)$, i.e.~$p\in R$. 
\newline
Case (ii). $E=j_*F$ is a reflexive Gieseker-stable sheaf supported on a hyperplane section $j\colon S\hookrightarrow Y$. Then by $\sigma$-stability,  $\Ext^2(E,E)=\Hom(E,\tau E)\neq 0$ if and only if $\tau(j_*F)\cong j_*F$.  

\end{proof}

The next Proposition analyses further the second case in Lemma \ref{singular_locus}.

\begin{proposition} \label{prop_bundle}
Let $\sigma$ be a Serre-invariant stability condition and $j_*F\in \Ku(Y)$ be a $\sigma$-stable object of class $\bw$, where $j\colon S\hookrightarrow Y$ is a hyperplane section and $F$ is a reflexive sheaf on $S$. Let $E\in \Ku(Y)$ be a Gieseker-stable sheaf of class $2\bv$. Assume that $\tau(j_*F)\cong j_*F$,  then we have
\[\RHom(\bO(j_*F),E)=\mathbb{C}^2[-2].\]
\end{proposition}

\begin{proof}
By Lemma \ref{Gstability_lvls_coincide}, $E$ is 2-Gieseker-stable. Thus $j^*E$ is a sheaf by the torsion-freeness of $E$. Since $F\in \Ku(Y)$, we see $\RHom(\bO(j_*F),E)=\RHom(j_*F(1),E)$.
It is clear that $\Hom(j_*F(1),E)=0$. 

We claim $\Ext^3(j_*F(1),E)=\Hom(E,j_*F(-1))=0$. If not, there is a nonzero map $\pi\colon E\to j_*F(-1)$ with $\ch_{\leq 1}(\ker(\pi))=(2,-H)$ and $H.\ch_2(\ker(\pi))\geq 1$. Thus by \cite[Proposition 3.2]{li:fano-picard-number-one}, $\ker(\pi)$ cannot be $\mu_H$-semistable. But since it is torsion-free, it has a two-term HN filtration $E_1 \hookrightarrow \ker(\pi) \twoheadrightarrow E_2$. Since $E_1$ is a subsheaf of $E$ as well, we have $\ch_{\leq 2}(E_1) = (1, 0, \frac{a}{2}H^2)$ where $a \leq -2$. Thus $\ch(E_2) = (1, -H)$ and $\ch_2(E_2).H = \ch_2(\ker(\pi)).H -a \geq 3$, which is not possible.    


Therefore we get $-\ext^1(\bO(j_*F),E) + \ext^2(\bO(j_*F),E) = \chi(\bO(j_*F),E)=2$, so we only need to show $\Ext^1(\bO(j_*F),E)=0$. Note that \[\Ext^1(\bO(j_*F),E)=\Ext^1(j_*F(1),E)=\Hom_S(F,j^*E)=\Hom(j_*F,j_*j^*E).\]
Assume there is a non-zero map $s\in \Hom_S(F,j^*E)$. Since $F$ is torsion-free of rank one on $S$, $s$ is injective. Let $G:=\cok(s)$. 
\newline
\textbf{Claim:} 
$G$ is a torsion-free sheaf on $S$. As $G$ has rank one on $S$, this implies $j_*G$ is Gieseker-stable. To this end, we consider a commutative diagram of exact triangles
\[\begin{tikzcd}
	0 & {j_*F} & {j_*F} \\
	E & {j_*j^*E} & {E(-1)[1]}
	\arrow[from=1-1, to=2-1]
	\arrow[from=2-1, to=2-2]
	\arrow[from=2-2, to=2-3]
	\arrow["j_*(s)", from=1-2, to=2-2]
	\arrow[from=1-3, to=2-3]
	\arrow[shift right=1, no head, from=1-2, to=1-3]
	\arrow[from=1-1, to=1-2]
	\arrow[no head, from=1-2, to=1-3]
\end{tikzcd}\]
By taking cones, we get a commutative diagram with rows and columns exact
\[\begin{tikzcd}
	0 & {j_*F} & {j_*F} \\
	E & {j_*j^*E} & {E(-1)[1]} \\
	E & {j_*G} & {K[1]}
	\arrow[from=1-1, to=2-1]
	\arrow[from=2-1, to=2-2]
	\arrow[from=2-2, to=2-3]
	\arrow["j_*(s)", from=1-2, to=2-2]
	\arrow[from=1-3, to=2-3]
	\arrow[shift right=1, no head, from=1-2, to=1-3]
	\arrow[from=1-1, to=1-2]
	\arrow[no head, from=1-2, to=1-3]
	\arrow["a", from=3-1, to=3-2]
	\arrow[from=3-2, to=3-3]
	\arrow[shift left=1, no head, from=2-1, to=3-1]
	\arrow[from=2-2, to=3-2]
	\arrow[from=2-3, to=3-3]
	\arrow[no head, from=2-1, to=3-1]
\end{tikzcd}\]
Here $K$ is a sheaf since it is an extension of $j_*F$ and $E(-1)$ from the construction. Thus $a$ is surjective and $K=\ker(a)$. Note that $\ch(K)=2\bv-\bw$. We consider two cases: 
\begin{itemize}
    \item If $K$ is $\mu_H$-stable, by Lemma \ref{rule_out} $K$ is locally free. Since $E$ is torsion-free,  
    we get torsion-freeness of $G$ on $S$.
    \item If $K$ is not $\mu_H$-semistable, then there is a destabilising sequence $K_1 \to K \to K_2$ where both $K_1$ and $K_2$ are rank one $\mu_H$-stable sheaf. Note that since $K$ is a subsheaf of $E$, it is torsion-free. The composition of injections $K_1 \to K \to E$ and $2$-Gieseker stability of $E$ implies that $\ch_{\leq 2}(K_1) = (1, 0, -\frac{a+2}{2}H^2)$ where $a \geq 0$. Since $K_2$ is torsion-free with class $\ch_{\leq 2}(K_2) = (1, -H, \frac{1 +a}{2}H^2 )$, we get $a = 0$. Thus $K_2 \cong \cI_{p_2}(-H)$ for some points $p_2$ on $Y$. 
    We denote $W:=\cok(K_1\hookrightarrow E)$. Then we have a commutative diagram 
\[\begin{tikzcd}
	& 0 & 0 \\
	0 & {K_1} & {K_1} & 0 \\
	0 & K & E & {j_*G} & 0 \\
	0 & {K_2} & W & {j_*G} & 0 \\
	& 0 & 0 & 0
	\arrow[from=1-2, to=2-2]
	\arrow[from=2-2, to=3-2]
	\arrow[from=3-2, to=4-2]
	\arrow[from=4-2, to=5-2]
	\arrow[from=1-3, to=2-3]
	\arrow[from=2-3, to=3-3]
	\arrow[from=3-3, to=4-3]
	\arrow[from=2-1, to=2-2]
	\arrow[shift right=1, no head, from=2-2, to=2-3]
	\arrow[from=2-3, to=2-4]
	\arrow[from=2-4, to=3-4]
	\arrow[no head, from=3-4, to=4-4]
	\arrow[from=4-4, to=5-4]
	\arrow[from=4-1, to=4-2]
	\arrow[from=4-2, to=4-3]
	\arrow[from=4-3, to=4-4]
	\arrow[from=4-4, to=4-5]
	\arrow[from=3-1, to=3-2]
	\arrow[from=3-2, to=3-3]
	\arrow[from=3-3, to=3-4]
	\arrow[from=3-4, to=3-5]
	\arrow[no head, from=2-2, to=2-3]
	\arrow[shift left=1, no head, from=3-4, to=4-4]
	\arrow[from=4-3, to=5-3]
\end{tikzcd}\]
with rows and columns exact. Since $\RHom(\oh_Y,j_*F)=\RHom(\oh_Y, j_*j^*E)=0$, we get the vanishing $\RHom(\oh_Y,j_*G)=0$. In particular, $G$ has no zero-dimensional torsion. We know $\ch_{\leq 2}(W) = (1, 0, 0)$, from $2$-Gieseker-stability of $E$, we see that the torsion part of $W$ is zero-dimensional, which is not possible as $G$ has no zero-dimensional subsheaf. Thus $W \cong \cI_{p}$ for some points $p$ in $Y$, so the third row in the above diagram gives the short exact sequence $\cI_{p_2}(-H) \hookrightarrow \cI_{p} \twoheadrightarrow j_*G$ which implies $j_*G$ is pure. 
\end{itemize}
Hence $G$ is torsion-free as claimed. Thus $j^*E$ is also torsion-free as $F$ and $G$ are. 

We divide the rest of the proof into two cases.

\textbf{Case 1.} First assume $E$ is not locally free. By Proposition \ref{classification_non_reflexive_sheaf_quartic_double_solid}, we have an exact sequence
\[0\to E\to \oh_Y^{\oplus 2}\to Q\to 0,\]
where $Q$ is supported on a curve. Hence we get a triangle $j^*E\to \oh_S^{\oplus 2}\to j^*Q$ on $S$. Since $Q$ is supported on a curve, $\cH^i_{\Coh(S)}(j^*Q)$ is at most one-dimensional for each $i$ by \cite[Lemma 3.29]{huyb-book-FM}. Using the fact that $j^*E$ is torsion-free, we see $j^*Q\in \Coh(S)$ and hence $j^*E\subset \oh^{\oplus 2}_S$. Thus $F\subset \oh^{\oplus 2}_S$, which implies that $\Hom_S(F,\oh_S)=\Hom(j_*F,j_*\oh_S)\neq 0$. Hence $\Hom(j_*F,\oh_Y(-1)[1])\neq 0$, which contradicts $j_*F\in \Ku(Y)$.

\textbf{Case 2.} Now assume $E$ is locally free, and so $j^*E$ is locally free. Then taking $\mathcal{H}om_S(-,F)$ from the short exact sequence $F \to j^*E \to G$ gives $\mathcal{E}xt_S^1(F,F)=\mathcal{E}xt_S^2(G,F)$. 
By Lemma \ref{lem_ext_sheaf_nonzero}, we get  $\mathcal{E}xt_S^2(G,F)\neq 0$, which implies $\Ext^3(j_*G,j_*F)\neq 0$ from Lemma \ref{lem_smooth_torsion}. However, by Serre duality we get $\Hom(j_*F, j_*G(-2))\neq 0$, which contradicts the Gieseker-stability of $j_*F$ and $j_*G$.

\end{proof}

\begin{lemma}\label{lem_smooth_torsion}
    Let $j\colon S\hookrightarrow Y$ be a hyperplane section and $E,F$ be two coherent sheaves on $S$ with $E$ torsion-free. Let $n\geq 2$ be the maximal integer with $\mathcal{E}xt^n_S(E,F)\neq 0$. Then $\Ext^{n+1}(j_*E,j_*F)\neq 0$.
\end{lemma}

\begin{proof}
We first show that any hyperplane section $S\in |\oh_Y(1)|$ is normal and Gorenstein. Since $Y$ is Gorenstein, $S$ is too. Then by Serre's criterion, to prove the normality of $S$, we only need to prove $S$ has only finitely many singular closed points. Note that $S=\pi^{-1}(P)$ is a double cover ramified over $R\cap P$ for a projective plane $P\subset \PP^3$. By the property of double cover, we only need to show $R\cap P$ has isolated singularties. This follows from applying \cite[Corollary 3.4.19]{lazarsfeld:positivity-in-AG-I} to $R$.

Since $S$ is normal, the non-locally free locus of $E$ has codimension two. Thus $\mathcal{E}xt_S^i(E,F)$ is supported on points for any $i>0$.
Now we compute $\mathcal{E}xt_Y^i(j_*E,j_*F):=\cH^i(R\mathcal{H}om_Y(j_*E,j_*F))$. By adjunction, we have
\[R\mathcal{H}om_Y(j_*E,j_*F)=j_*R\mathcal{H}om_S(j^*j_*E,F).\]
Since $\cH^0(j^*j_*E)\cong E$ and $\cH^{-1}(j^*j_*E)\cong E(-1)$, using \cite[(3.8)]{huyb-book-FM}, we have a spectral sequence convergent to $\mathcal{E}xt_S^{p+q}(j^*j_*E,F)$ with $E_2^{p,0}=\mathcal{E}xt^p_S(E,F)$, $E_2^{p,1}=\mathcal{E}xt^p_S(E,F)(1)$ and $E_2^{p,q}=0$ for $p\neq 0,1$. Therefore, we see that $\mathcal{E}xt_S^{i}(j^*j_*E,F)$ is supported on points for $i\geq 2$. Moreover, the term $E_2^{n,1}$ survives, hence $E_2^{n,1}=E_{\infty}^{n,1}\neq 0$ implies that $\mathcal{E}xt_S^{n+1}(j^*j_*E,F)\neq 0$. Thus $\mathcal{E}xt^i_Y(j_*E,j_*F)$ is supported on $S$, and furthermore supported on points for $i\geq 2$ with $\mathcal{E}xt^{n+1}_Y(j_*E,j_*F)\neq 0$.

Next, using \cite[(3.16)]{huyb-book-FM}, we have a spectral sequence \[E_2^{p,q}=H^p(\mathcal{E}xt_Y^q(j_*E,j_*F))\Rightarrow \Ext^{p+q}(j_*E,j_*F).\]
By the previous argument, we know that $E^{0,n+1}_2=\mathrm{length}(\mathcal{E}xt^{n+1}_Y(j_*E,j_*F))\neq 0$. Moreover, from the dimension of support, we see $E_2^{p,q}=0$ for $p\in \{1,2\}, q\geq 2$ and any $p\geq 3$, $q\in \ZZ$. Since $n\geq 2$, this implies $E_2^{0,n+1}=E_{\infty}^{0,n+1}\neq 0$, which gives $\Ext^{n+1}(j_*E,j_*F)\neq 0$.
\end{proof}

\begin{lemma} \label{lem_ext_sheaf_nonzero}
Let $\sigma$ be a Serre-invariant stability condition on $\Ku(Y)$ and $j_*F\in \Ku(Y)$ be a $\sigma$-stable object where $j\colon S\hookrightarrow Y$ is a hyperplane section and $F$ is a reflexive sheaf on $S$. If $\tau(j_*F)\cong j_*F$, or equivalently $\Ext^2(j_*F,j_*F)\neq 0$,  then $\mathcal{E}xt_S^1(F,F)$ is non-zero and supported on a single point with length one.
\end{lemma}

\begin{proof}
Note that by $\sigma$-stability and $S_{\Ku(Y)}=\tau[2]$, we know that $\Ext^2(j_*F,j_*F)=\Hom(j_*F,\tau(j_*F))\neq 0$ if and only if $\Ext^2(j_*F,j_*F)=\Hom(j_*F,\tau(j_*F))=\mathbb{C}$.

Since $F$ is reflexive and $S$ is normal, we have $\mathcal{H}om_S(F,F)=\oh_S$. Moreover, by Lemma \ref{lem_smooth_torsion} and the vanishing 
$\Ext^i(j_*F,j_*F)=0$ when $i\geq 3$, we get $\mathcal{E}xt_S^i(F,F)=0$ for $i\geq 2$. Therefore, if we compute $\mathcal{E}xt^2_Y(j_*F, j_*F)$ as in Lemma \ref{lem_smooth_torsion}, we get $\mathcal{H}om_Y(j_*F,j_*F)=j_*\oh_S$, $\mathcal{E}xt^1_Y(j_*F, j_*F)$ is an extension of $j_*\oh_S(1)$ with $j_*\mathcal{E}xt_S^1(F,F)$, and  $\mathcal{E}xt^2_Y(j_*F, j_*F)=j_*\mathcal{E}xt_S^1(F,F)(1)$. Thus, if we compute $\Ext^i(j_*F,j_*F)$ as in Lemma \ref{lem_smooth_torsion}, we see $\Ext^2(j_*F,j_*F)=H^0(\mathcal{E}xt^2_Y(j_*F, j_*F))$. This implies that $\mathcal{E}xt^1_S(F, F)$ is non-zero and supported on a single point with length one.
\end{proof}

\begin{proof}[Proof of Theorem \ref{prop_gluing_data_fixed} for degree $d =2$]
Note that $\bO^4\cong [2]$ when $d=2$, so by Lemma \ref{numerical_lemma} we can assume that there is a pair of integers $m_1, \delta$ with $0\leq m_1\leq 3$ and $\delta = 0, 1$ such that $\bO^{-m_1}\circ \Phi[\delta]$ maps classes $\bv$ and $\bw$ on $Y$ to $\bv'$ and $\bw'$ on $Y'$, respectively. Moreover, we know such $m_1$ and $\delta$ is unique by looking at the action of $\bO$ and $[1]$ on $\mathcal{N}(\Ku(Y))$ and using the restricted values of $m_1$ and $\delta$. We may replace $\Phi$ by $\bO^{-m_1}\circ \Phi[\delta]$. 

We know $\Phi(i^!\cQ_Y) \in \cM_{\Phi(\sigma)}(\Ku(Y'), 2\bv')$, so by Proposition \ref{prop_bridgeland_Y2}, up to a shift, it is either ${i'}^!\cQ_{Y'}$ or a Gieseker-stable sheaf $E'$.  Assume for a contradiction that the latter happens. We know $\Phi$ maps the singular locus of $\cM_{\sigma}(\Ku(Y), 2\bv-\bw)$ to the singular locus of $\cM_{\Phi(\sigma)}(\Ku(Y'), 2\bv'-\bw')$.

\begin{itemize}
    \item Assume that $\Phi$ maps $R\subset \cM_{\sigma}(\Ku(Y), 2\bv-\bw)$ to $R'\subset \cM_{\Phi(\sigma)}(\Ku(Y'), 2\bv'-\bw')$. Thus by Proposition \ref{prop_moduli_w}, we get 
$\RHom\left(\Phi(i^*\oh_p), \Phi(i^{!}\cQ_Y)\right)$ and so $\RHom\left({i'}^*\oh_{p'}, E'\right) = \RHom\left(\oh_{p'}, E'\right)$ are a two-term complex for all $p' \in R$. But this makes a contradiction since $E'$ is torsion-free so the non-locally free locus of $E'$ has at most dimension one.

\item Assume that $\Phi$ does not map  $R\subset \cM_{\sigma}(\Ku(Y), 2\bv-\bw)$ to $R'\subset \cM_{\Phi(\sigma)}(\Ku(Y'), 2\bv'-\bw')$. By Lemma \ref{singular_locus}, there is a point $p\in R$ such that $\Phi(i^*\oh_p)=\bO(j_*F)$ up to shift, where $j\colon S\hookrightarrow Y'$ is a hyperplane section and $F$ is a reflexive sheaf on $S$ with $\tau'(j_*F)\cong j_*F$. Moreover, $\RHom(i^*\oh_p,i^!\cQ_Y)=\RHom(\bO(j_*F),E')$ is a two-term complex. But this contradicts Proposition \ref{prop_bundle}.
\end{itemize}


Hence in both cases, we get $\Phi(i^!\cQ_Y)=i'^!\cQ_{Y'}[m_2+\delta]$ for a unique $m_2 \in \ZZ$ and the claim follows. 
\end{proof}

\begin{remark}
\cite[Lemma 4.4]{rota:moduli-space-index-two} claims $\mathrm{Ext}^2(j_*F,j_*F)=0$ for any hyperplane section $j\colon S\hookrightarrow Y$ and a rank one reflexive sheaf $F$ on $S$ such that $j_*F\in \Ku(Y)$. However, the proof is valid only for smooth $S$ via the vanishing of $\mathcal{E}xt^1_S(F,F)$. That is why in this section, we investigated further the singular locus in order to prove Theorem \ref{prop_gluing_data_fixed}.  
\end{remark}

\subsection{Degree three case} Now assume $Y$ is a cubic threefold. 

\begin{proof}[Proof of Theorem \ref{prop_gluing_data_fixed} for degree $d =3$]
In this case $\bO^6\cong[4]$, so by Lemma \ref{numerical_lemma} there is a unique pair of integer $m_1, \delta$ with $0\leq m_1 \leq 5$ and $\delta = 0, 1$ such that $\bO^{-m_1}\circ \Phi[\delta]$ maps classes $\bv$ and $\bw$ on $Y$ to $\bv'$ and $\bw'$ on $Y'$, respectively. We replace $\Phi$ by $\bO^{-m_1}\circ \Phi[\delta]$. Then by Proposition \ref{prop_Y3_classify}, the object $\Phi(i^!\cQ_Y) \in \cM_{\Phi(\sigma)}(\Ku(Y'), 3\bv')$, up to a shift, is either ${i'}^!\cQ_{Y'}$ or a Gieseker-semistable sheaf $E'$. Assume for a contradiction that the latter happens.

By \cite[Lemma 7.5, Theorem 8.7]{feyz:desing}, $\mathcal{BN}_{Y'}$ is the union of all rational curves in $\cM_{\Phi(\sigma)}(\Ku(Y'), 3\bv'-\bw')$. Thus $\phi(\mathcal{BN}_{Y})=\mathcal{BN}_{Y'}$. In other words, for any $p\in Y$ we have $\Phi(i^*\oh_p)\cong i'^*\oh_{p'}$ for a point $p'\in Y'$ up to shift and vice verse. In particular, $\RHom(i'^*\oh_{p'}, E')$ is a two-term complex for all $p'\in Y'$. But this contradicts the torsion-freeness of $E'$. Hence we get $\Phi(i^!\cQ_Y)=i'^!\cQ_{Y'}[m_2+\delta]$ for a unique $m_2 \in \ZZ$ as claimed. 



\end{proof}

\subsection{Degree four case}\label{degree 4-subsection}
Let $Y$ be a del Pezzo threefold of degree $4$, then $\Ku(Y)$ is equivalent to the bounded derived category $\D^b(C)$ of a smooth projective curve $C$ of genus $2$. As in \cite[Section 5]{Kuznetsov:instanton-bundle-Fano-threefold}, we fix the Fourier--Mukai equivalence $\Psi_{\mathcal{S}} \colon \D^b(C) \to \Ku(Y)$ for the universal spinor bundle $\mathcal{S}$ on $C \times Y$, where we see $Y$ as a moduli space of stable rank $2$ bundles on $C$ with fixed determinant $\xi$ of degree $\deg(\xi) =1$.


For any line bundle $\mathcal{L}$ on $C$, we denote the induced auto-equivalence of $\Ku(Y)$ by $T_{\mathcal{L}}:=\Psi_{\mathcal{S}}\circ( -\otimes \mathcal{L}) \circ\Psi_{\mathcal{S}}^{-1} $. We write $\mathrm{Aut}^0(\Ku(Y))$ for the subgroup of $\mathrm{Aut}(\Ku(Y))$ consists of $T_{\mathcal{L}}$ such that $\mathcal{L}\in \Pic^0(C)$. We will apply the following two facts about the action of $\bO$: 
\begin{enumerate}
    \item[(a)] By \cite[Lemma 5.2]{Kuznetsov:instanton-bundle-Fano-threefold}, we know that via the equivalence $\Psi_{\mathcal{S}}$, the action of $\bO$ on $\mathcal{N}(\Ku(Y))$ is the same as twisting by a degree $-1$ line bundle on $C$, up to sign.
    \item [(b)] Since any stability condition $\sigma$ on $\Ku(Y)$ is $\bO$-invariant, (semi)stability of a vector bundle on $C$ will be preserved after the action of $\Psi_{\mathcal{S}}^{-1} \circ \bO \circ \Psi_{\mathcal{S}}$.  
\end{enumerate}

\begin{proof}[Proof of Theorem \ref{prop_gluing_data_fixed} for degree $d=4$]
By Lemma \ref{numerical_lemma}, there exist a pair of integers $m_1, m_2$ such that $\bO^{-m_1}\circ \Phi[-m_2]$ maps classes $\bv$ and $\bw$ to $\bv'$ and $\bw'$. By the above point (a), such $m_1$ is unique. Furthermore, we can take $m_2$ uniquely by imposing the condition that $
\Psi_{\mathcal{S}}^{-1}\circ (\bO^{-m_1}\circ \Phi[-m_2])\circ \Psi_{\mathcal{S}}\colon \D^b(C)\to \D^b(C')$ maps bundles to bundles. We replace $\Phi$ by $\bO^{-m_1}\circ \Phi[-m_2]$. 

By \cite[Lemma 5.9]{Kuznetsov:instanton-bundle-Fano-threefold}, $\Psi_{\mathcal{S}}^{-1}(i^!\oh_Y)$ is a second Raynaud bundle\footnote{It is a semistable vector bundle of rank $4$ and degree $4$ on a genus $2$ curve so that for any line bundle $\mathcal{L}$ of degree zero on $C$, we have $\Hom(\mathcal{L}, \mathcal{R}) \neq 0$. } $\mathcal{R}$ on $C$ up to a shift. We know this bundle is unique on $C$ up to tensoring by a line bundle of degree zero. Thus by the above point (b), $\Psi_{\mathcal{S}}^{-1}(\bO(i^!\oh_Y)) = \Psi_{\mathcal{S}}^{-1}(i^!\cQ_Y)$ is also unique up to tensoring by a line bundle of degree zero (Indeed, let $R$ and $R'$ be two Raynaud bundle, then we can assume $R'=R\otimes L_0$ for a degree $0$ line bundle $L_0$, note that $\bO = f_* \circ (- \otimes L_{-1})$ for a degree $-1$ line bundle $L_{-1}$ up to shift, so that $\bO(R')=\bO(R \otimes L_0) = f_*(R) \otimes L'_{-1}$ a degree $-1$ line bundle $L'_{-1}$. On the other hand, $\bO(R)=f_*(R)\otimes L''_{-1}$ for a degree $-1$ line bundle $L''_{-1}$. Hence $\bO(R)$ and $\bO(R')$ differ by a degree $0$ line bundle.
This proves there is a unique line bundle $\mathcal{L}_0$ on $C'$ such that 
$$
({\Psi'}^{-1}_{\mathcal{S}'}\circ\Phi(i^!\cQ_Y)) \otimes \mathcal{L}_0^{-1} = {\Psi'}_{\mathcal{S}'}^{-1}({i'}^!\cQ_{Y'})
$$ 
and so the claim follows.

\end{proof}

Proof of Theorem \ref{prop_gluing_data_fixed}, in particular, implies the following. 

\begin{corollary}\label{cor-7}\leavevmode
     \begin{itemize}
         \item If $d=2$, $i^{!}\cQ_Y$ is the unique object in the moduli space $\mathcal{M}_{\sigma}(\Ku(Y), 2\bv)$ satisfying the following condition: there is a 2-dimensional sub-locus $\mathcal{M}'$ of the singular locus of the moduli space $\mathcal{M}_{\sigma}(\Ku(Y),2\bv-\bw)$ such that for any object $\cE \in \mathcal{M}'$, $\RHom(\cE, i^{!}\cQ_Y)$ is a two-term complex. 
     
     \item If $d=3$, $i^{!}\cQ_Y$ is the unique object in the moduli space $\mathcal{M}_{\sigma}(\Ku(Y),3\bv)$ such that for any object $\cE \in \mathcal{M}_{\sigma}(\Ku(Y),3\bv-\bw)$ whose corresponding point lies on a ration curve in $\mathcal{M}_{\sigma}(\Ku(Y),3\bv-\bw)$, $\RHom(\cE, i^{!}\cQ_Y)$ is a two-term complex. 
     
     \item If $d=4$, $i^{!}\cQ_Y$ is a unique object in the moduli space $\mathcal{M}_{\sigma}(\Ku(Y),4\bv)$ up to the action of an auto-equivalence $T_{\mathcal{L}} \in \mathrm{Aut}^0(\Ku(Y))$ such that $\RHom(\cE, i^{!}\cQ_Y)$ is a two-term complex for any object $\cE \in \mathcal{M}_{\sigma}(\Ku(Y),-3\bv+\bw)$.  
     
   \end{itemize}
\end{corollary}
\begin{proof}
    The first two cases for degree $d=2, 3$ are a direct result of proof of Theorem \ref{prop_gluing_data_fixed}. For $d=4$, note that $\chi(\cE, i^{!}\cQ_Y)=0$. Combining \cite[Lemma 5.9]{Kuznetsov:instanton-bundle-Fano-threefold} and \cite[Theorem 1.1]{pertusi:some-remarks-fano-threefolds-index-two} implies that $i^{!}\oh_Y$ is a unique object in $\mathcal{M}_{\sigma}(\Ku(Y),4(\bv-\bw))$ up to the action of an auto-equivalence $T_{\mathcal{L}_0} \in \mathrm{Aut}^0(\Ku(Y))$ such that $\RHom(\cE, i^{!}\cQ_Y)$ is a two-term complex for any object $\cE \in \mathcal{M}_{\sigma}(\Ku(Y),\bv)$. Thus taking the rotation functor $\bO$ implies the claim.     
\end{proof}

\subsection{Categorical Torelli theorem}
As a result of Theorem \ref{prop_gluing_data_fixed}, we show a categorical Torelli theorem for any del Pezzo threefold of degree $2\leq d\leq 4$. 
\begin{corollary}
\label{coro_Torelli_from_BN}
Let $Y$ and $Y'$ be del Pezzo threefolds of degree $2\leq d\leq 4$ such that $\Phi:\Ku(Y)\simeq\Ku(Y')$ is an exact equivalence of Kuznetsov components, then $Y\cong Y'$. 
\end{corollary}

\begin{proof}
By Theorem~\ref{prop_gluing_data_fixed}, we can assume that $\Phi(i^!\cQ_Y)\cong i'^!\cQ_{Y'}$.  There is an isometry of numerical Grothendieck group $\phi:\mathcal{N}(\Ku(Y))\cong\mathcal{N}(\Ku(Y'))$ induced by $\Phi:\Ku(Y)\simeq\Ku(Y')$. As $\Phi(i^!\cQ_Y)\cong i'^!\cQ_{Y'}$, we get $\phi(\bv)=\bv'$ and $\phi(\bw)=\bw'$ by Lemma \ref{numerical_lemma}. Then the result follows from the uniqueness of Serre-invariant stability conditions and Theorem \ref{Brill--Noether reconstruction} via the same argument in \cite[Corollary 6.11]{jacovskis2022brill}.
\end{proof}

\section{Auto-equivalences of Kuznetsov components: index two case}\label{section_autoeq_dP}

In this section, we are going to prove Theorem \ref{thm-conj} and Corollary \ref{cor_auto_equi_cubic}, and describe the auto-equivalences of Kuznetsov components of del Pezzo threefolds. We begin with a lemma.

\begin{lemma} \label{unique_extend}
    Let $f,g\colon Y\to Y'$ be two isomorphisms between del Pezzo threefolds of Picard one. If $f_{*}|_{\Ku(Y)}=g_*|_{\Ku(Y)}\colon \Ku(Y)\to \Ku(Y')$, then $f=g$. Thus the homomorphism 
    \[\Aut(Y)\to \Aut(\Ku(Y)), \quad f\mapsto f_*|_{\Ku(Y)}\]
    is injective.
\end{lemma}
\begin{proof}
We know $f_*$ and $g_*$ maps $\oh_Y$ and $\oh_Y(1)$ to $\oh_{Y'}$ and $\oh_{Y'}(1)$ respectively. For any point $p \in Y$, we know $f_{*}(\oh_p) = \oh_{f(p)}$ and the same for $g$. Thus we have
\begin{equation*}
f_*(i^*\oh_p)=i'^*\oh_{f(p)} \qquad \text{and } \qquad g_*(i^*\oh_p)=i'^*\oh_{g(p)}. 
\end{equation*}
Since $f_{*}|_{\Ku(Y)}=g_*|_{\Ku(Y)}$, we get $i'^*\oh_{f(p)} = i'^*\oh_{g(p)}$, i.e. $i'^*\oh_{f(p)}$ and $i'^*\oh_{g(p)}$ correspond to the same point in the moduli space $\cM_{\sigma}(\mathcal{K}u(Y), d\bv-\bw)$ by Proposition \ref{prop_moduli_w}. Thus the embedding $\Psi$ in \eqref{psi} implies that $f(p) = g(p)$ for any point $p \in Y$. Since both $Y$ and $Y'$ are smooth, we get $f=g$.  
\end{proof}

\begin{theorem} \label{thm-conj}
Let $Y$ and $Y'$ be two del Pezzo threefolds of the same degree $d$ where $d= 2, 3$ or $4$, and let $\Phi \colon \Ku(Y) \to \Ku(Y')$ be an exact equivalence of Fourier--Mukai type such that  $\Phi(i^{!}\cQ_Y) = i'^{!}\cQ_{Y'}$. Then $\Phi=f_*|_{\Ku(Y)}$ for a unique isomorphism $f \colon Y\to Y'$. 
\end{theorem}

\begin{proof}
Since $[i^{{!}}\cQ_Y] = d\bv \in \mathcal{N}(\Ku(Y))$, Lemma \ref{numerical_lemma}~(a) implies that $\Phi$ maps $\bv$ and $\bw$ to $\bv'$ and $\bw'$, respectively. Then Theorem \ref{Brill--Noether reconstruction} shows that for any $p\in Y$, there is a point $p'\in Y'$ such that 
\begin{equation}\label{o-p2}
  \Phi(i^*\oh_{p})\cong i'^*\oh_{p'}.  
\end{equation}
Conversely, for any $p'\in Y'$, there is $p \in Y$ such that the above holds. From Remark \ref{fix_i!O}, we also have $\Phi(i^!\oh_Y)=i'^!\oh_{Y'}$. Thus using  \cite[Proposition 2.5 \& Remark 2.2]{li2021refined}, $\Phi$ can be extended to an equivalence $\oh_Y(1)^{\perp}\cong \oh_{Y'}(1)^{\perp}$, denoted again by $\Phi$, so that 
$\Phi(\oh_Y)\cong \oh_{Y'}$. Since $i^*=\bL_{\oh_Y}\bL_{\oh_Y(1)}$, \eqref{o-p2} implies that $\Phi(\bL_{\oh_Y(1)}(\oh_p))\cong \bL_{\oh_{Y'}(1)}(\oh_{p'})$.

Let $j:\oh_Y(1)^{\perp} \hookrightarrow \D^b(Y)$ and $j': \oh_{Y'}(1)^{\perp}\hookrightarrow \D^b(Y')$ be the natural inclusions. We know $$j^! \oh_Y(1) = \bR_{\oh_Y(-1)}(\oh_Y(1)),$$ so it lies in the triangle 
	\begin{equation}\label{eq_j}
	    \oh_Y(-1)[2]\to j^!\oh_Y(1)\to \oh_Y(1).
	\end{equation}
	
	The next step is to compute $i^*(j^!\oh_Y(1)) = \bL_{\oh_Y}(j^!\oh_Y(1))$. Using the triangle above, it is easy to see $\RHom(\oh_{Y}, j^!\oh_Y(1)) = \mathbb{C}^{d+2}$, so we have an triangle
	\begin{equation}\label{eq_equiv}
	    \oh_Y^{\oplus d+2} \rightarrow  j^!\oh_Y(1) \rightarrow \bL_{\oh_Y}(j^!\oh_Y(1)).
	\end{equation}
  Thus by taking cohomology we obtain 
	\begin{equation*}
	\oh_Y(-1)[2] \rightarrow \bL_{\oh_Y}(j^!\oh_Y(1)) \rightarrow \cQ_Y[1]
	\end{equation*} 
	and so $\bL_{\oh_Y}(j^!\oh_Y(1)) = i^{!}\cQ_Y[1]$. Therefore, we know that $\Phi(\bL_{\oh_Y}(j^!\oh_Y(1))) = \bL_{\oh_{Y'}}({j'}^!\oh_{Y'}(1))$. Applying $\Phi$ to \eqref{eq_equiv} gives a triangle
	\begin{equation} \label{eq_equiv_1}
	\oh_{Y'}^{\oplus d+2} \to \Phi(j^!\oh_Y(1)) \to {i'}^!\cQ_{Y'}[1]. 
	\end{equation}
This implies that $\cH^{-2}(\Phi(j^!\oh_Y(1))) = \oh_{Y'}(-1)$ and we have the long exact sequence 
	\begin{equation}\label{eq_long_exact}
	0 \rightarrow \cH^{-1}(\Phi(j^!\oh_Y(1))) \to \cQ_{Y'} \to \oh_{Y'}^{\oplus d+2} \to \cH^0(\Phi(j^!\oh_Y(1))) \to 0. 
	\end{equation}
	Since $j^!\oh_Y(1)\in \oh_Y(1)^{\perp}$, by the adjunction of mutations, we have $\RHom(\bL_{\oh_Y(1)}(\oh_p), j^!\oh_Y(1))=\RHom(\oh_p, j^!\oh_Y(1))$ for any $p\in Y$. Thus we have
\[\RHom(\oh_p, j^!\oh_Y(1))=\RHom(\bL_{\oh_Y(1)}(\oh_p), j^!\oh_Y(1))=\RHom(\Phi(\bL_{\oh_Y(1)}(\oh_p)), \Phi(j^!\oh_Y(1)))\]
\[=\RHom(\bL_{\oh_{Y'}(1)}(\oh_{p'}), \Phi(j^!\oh_{Y}(1)))=\RHom(\oh_{p'}, \Phi(j^!\oh_Y(1))).\]
Using \eqref{eq_j}, we know that $\RHom(\oh_p, j^!\oh_Y(1))=\mathbb{C}[-1]\oplus \mathbb{C}[-3]$. Hence $\RHom(\oh_{p'}, \Phi(j^!\oh_Y(1)))=\mathbb{C}[-1]\oplus \mathbb{C}[-3]$ for any $p'\in Y'$. By Serre-duality, we have
\begin{equation}\label{rhom_phi}
    \RHom(\Phi(j^!\oh_Y(1)),\oh_{p'})=\mathbb{C}\oplus \mathbb{C}[-2].
\end{equation}
Then from \cite[Proposition 5.4]{bridgeland:K3-and-elliptic-fibration}, $\Phi(j^!\oh_Y(1))$ is quasi-isomorphic to a complex
\begin{equation}\label{complex}
A_{-2}\to A_{-1}\xra{\alpha} A_0,    
\end{equation}
where $A_k$ is a bundle of rank $r_k$ sitting in degree $k$ in the complex. Note that \eqref{complex} is a locally-free resolution of $\Phi(j^!\oh_Y(1))$. Therefore, we have $\cH^0(\Phi(j^!\oh_Y(1)))\cong \cok(\alpha)$ and by applying $\Hom(-,\oh_{p'})$ to \eqref{complex}, we have a complex 
\[\Hom(A_0,\oh_{p'})=\mathbb{C}^{r_0}\xra{\overline{\alpha}} \Hom(A_{-1},\oh_{p'})\to \Hom(A_{-2},\oh_{p'}).\]
Since $\Hom(\Phi(j^!\oh_Y(1)),\oh_{p'})=\mathbb{C}$, we get $\ker(\overline{\alpha})=\mathbb{C}$. But note that $\overline{\alpha}$ can be factored as $\Hom(A_0,\oh_{p'})\to \Hom(\im(\alpha),\oh_{p'})\hookrightarrow \Hom(A_{-1},\oh_{p'})$ which implies
\[\hom((\im(\alpha),\oh_{p'}))\geq r_0-1.\]
Since $p'\in Y'$ is an arbitrary closed points, we have $\rk(\im(\alpha))\geq r_0-1$. Thus $\rk(\cH^0(\Phi(j^!\oh_Y(1))))\leq 1$. Since $\cH^0(\Phi(j^!\oh_Y(1)))$ sits in an exact sequence \eqref{eq_long_exact} and $\rk(\cQ_{Y'}) =d+1$, we have $\rk(\cH^0(\Phi(j^!\oh_Y(1))))=1$, which implies \[\rk(\cH^{-1}(\Phi(j^!\oh_Y(1))))=0.\] 
Since $\cQ_{Y'}$ is torsion-free, we have $\cH^{-1}(\Phi(j^!\oh_Y(1)))=0$ and $\cH^{0}(\Phi(j^!\oh_Y(1)))=\oh_{Y'}(1)$ by definition \eqref{eq. Q}. Thus $\Phi(j^!\oh_Y(1))$ lies in the exact triangle
\begin{equation} \label{triangle_1}
    \oh_{Y'}(-1)[2] \to \Phi(j^!\oh_Y(1)) \to \oh_{Y'}(1). 
\end{equation}
Note that $\Hom(\Phi(j^!\oh_Y(1)), \Phi(j^!\oh_Y(1))) = \Hom(j^!\oh_Y(1),j^!\oh_Y(1))=\Hom(j^!\oh_Y(1),\oh_Y(1))=\mathbb{C}$ by \eqref{eq_j}, so the exact triangle \eqref{triangle_1} is non-splitting. Since $\Hom(\oh_{Y'}(1), \oh_{Y'}(-1)[3])=1$, we get $$\Phi(j^!\oh_Y(1))\cong j^!\oh_{Y'}(1).$$
Then applying again \cite[Proposition 2.5]{li2021refined} shows that the equivalence $\Phi\colon \oh_Y(1)^{\perp}\to \oh_{Y'}(1)^{\perp} 
$ can be extended to an equivalence $\Phi\colon \D^b(Y)\xra{\cong} \D^b(Y')$ such that $\Phi(\oh_Y(1))\cong \oh_{Y'}(1)$. Then \cite[Corollary 5.23]{huyb-book-FM} implies that $\Phi$ is the composition of $f_*$ for an isomorphism $f \colon Y \to Y'$ with the twist by a line bundle on $Y$. Since we know $\Phi(\oh_Y) = \oh_{Y'}$, we get $\Phi = f_*$. 
Finally, such isomorphism $f$ is unique by Lemma \ref{unique_extend}.

\end{proof}

\begin{remark}\label{application_categorical_Torelli}
Combing Theorem~\ref{prop_gluing_data_fixed} with Theorem~\ref{thm-conj} provides an alternative proof of \emph{Categorical Torelli theorem} for del Pezzo threefold of degree $2\leq d\leq 4.$
\end{remark}

As an application, we obtain a complete description of the group $\mathrm{Aut}_{\mathrm{FM}}(\Ku(Y))$ of exact auto-equivalences of $\Ku(Y)$ of Fourier--Mukai type.


\begin{corollary}\label{cor_auto_equi_cubic}
Let $Y$ be a del Pezzo threefold of Picard rank one and degree $d$, and $\Phi\in \mathrm{Aut}_{\mathrm{FM}}(\Ku(Y))$ be an auto-equivalence of $\Ku(Y)$ of Fourier--Mukai type.

\begin{enumerate}
    \item [(i)] If $2\leq d \leq 3$, there exist a unique $f\in \Aut(Y)$ and unique pair of integers $m_1, m_2 \in \ZZ$ with $0\leq m_1\leq 3$ when $d=2$ and $0\leq m_1\leq 5$ when $d=3$, so that    $$\Phi=\bO^{m_1}\circ f_*|_{\Ku(Y)}\circ [m_2].$$
    \item [(ii)] If $d=4$, there exists a unique $f\in \Aut(Y)$ and unique pair of integers $m_1,m_2$ and a unique auto-equivalence $T_{\mathcal{L}_0} \in \mathrm{Aut}^0(\Ku(Y))$ (see Section \ref{degree 4-subsection} for definition) so that 
 $$\Phi= \bO^{m_1}\circ T_{\mathcal{L}_0}\circ f_*|_{\Ku(Y)}\circ [m_2]. $$
\end{enumerate}
\end{corollary}

\begin{proof}
The result follows from Theorem \ref{prop_gluing_data_fixed} and Theorem \ref{thm-conj}.
\end{proof}

\begin{remark}
 \label{lemma_strengthen_Sasha_result}
 Assume $Y'=Y$, Then Remark \ref{rmk_trivial} and Theorem \ref{thm-conj} 
show that the homomorphism 
$$\mathrm{Aut}(Y)\rightarrow\mathrm{Aut}_{\mathrm{FM}}(\Ku(Y)), \quad f\mapsto f_*|_{\Ku(Y)}$$ is injective, and its image together with $[2]$ generates the sub-group of auto-equivalences that act trivially on $\mathcal{N}(\Ku(Y))$. This strengthens a result \cite[Lemma B.2.3]{kuznetsov2018hilbert}. 
\end{remark}

\begin{remark}
\label{auto_equivalence_d245}

 For a del Pezzo threefold $Y$ of degree $5$, its Kuznetsov component $\Ku(Y)$ is equivalent to the derived category of representations of $3$-Kronecker quiver. It is known the  group of auto-equivalences of $\Ku(Y)$ is $\mathbb{Z}\times (\mathbb{Z}\rtimes\mathrm{PGL}_3(\mathbb{C}))$ by \cite[Theorem 4.3]{miyachi2001derived}.

\end{remark}


\section{Auto-equivalences of Kuznetsov components: index one case}
\label{sec-index-one}

Our goal in this section is to prove Corollary \ref{cor_aut_genus_8} and Corollary \ref{cor_aut_genus_6}, and describe the auto-equivalences of Kuznetsov components of index one prime Fano threefolds of genus $6$ and $8$.

Let $X$ be an index one prime Fano threefold of even genus $g\geq 6$. Then we have semiorthogonal decompositions
\[\D^b(X)=\langle \Ku(X), \cE_X, \oh_X \rangle\]
and
\[\D^b(X)=\langle \cA_X, \oh_X, \cE^{\vee}_X\rangle,\]
where $\cE_X$ is a unique rank two bundle on $X$ with a certain Chern character, obtained by pulling back the tautological subbundle of $\Gr(2,\frac{g}{2}+2)$ via a closed immersion $X\hookrightarrow \Gr(2,\frac{g}{2}+2)$ (cf.~\cite[Remark 2.6]{kuznetsov:derived-category-fano-threefold}). By \cite[Lemma 3.6]{jacovskis2021categorical}, there is an equivalence $\Xi\colon \Ku(X)\xra{\simeq} \cA_X$.

Let $i\colon \Ku(X)\hookrightarrow \D^b(X)$ be the inclusion. The object $i^!\cE_X$ is also called the gluing object and plays the same role in \cite{jacovskis2022brill} as $i^!\cQ_Y$ in the current paper. In the following, we will prove an analog for Theorem \ref{prop_gluing_data_fixed}. 

We begin with a lemma, which is a generalization of the argument in \cite[Section 4]{liu2023autoeq}. 

\begin{lemma} \label{lem_kernel_trivial}
    Let $X$ be a smooth projective variety and $i\colon \cA\hookrightarrow \D^b(X)$ be an admissible subcategory with a right adjoint $i^*$. Assume furthermore that there exists $n\in \ZZ$ such that $i^*\oh_x[-n]=K^x$ is a Gieseker-stable sheaf for any $x\in X$ and $i^*$ induces a closed immersion
    \[\pi\colon X\hookrightarrow M_X([K^x]),\quad x\mapsto i^*\oh_x[-n]=K^x\]
    where $M_X([K^x])$ is a fine moduli space of Gieseker-semistable sheaves of class $[K^x]$. Assume furthermore that for any non-trivial line bundle $L$ on $X$, there is a point $x\in X$ with
    \begin{equation}\label{assumption}
        K^x\neq i^*(K^x\otimes L).
    \end{equation}
    Then a Fourier--Mukai type functor $\Phi\colon \cA\to \cA$ is isomorphic to $\identity_{\cA}$ if and only if $\Phi(K^x)\cong K^x$ for any point $x\in X$.
\end{lemma}

\begin{proof}
The 'only if' part is obvious. We now prove the 'if' part.

By \cite[Lemma 3.31]{huyb-book-FM}, the kernel of $\Phi$ is isomorphic to a coherent sheaf $\mathcal{K}[n]$ on $X\times X$ flat over the first factor and $i^*_x\cK\cong K^x$ for any $x\in X$ and the closed immersion $i_x\colon \{x\}\times X \hookrightarrow X\times X$. Since $M_X([K^x])$ is fine, we have an induced morphism $\pi'\colon X\to M_X([K^x])$. Let $\cP[n]$ be the kernel of $\i\circ i^*\colon \D^b(X)\to \D^b(X)$. Note that $\identity_{\cA}=i\circ i^*|_{\cA}$. Then $\cP$ also induces a morphism $X\to M_X([K^x])$, which is exactly $\pi$.
    
     It is clear that $\pi$ and $\pi'$ have the same image as continuous maps. Then by taking the schematical image of $\pi'$ and the uniqueness of the reduced scheme structure of a closed subset, we see that $\pi'$ can be factored as $\pi\circ f$, where $f\colon X\to X$ is a bijective morphism. Since $X$ is smooth and connected, we deduce that $f=\identity_X$, hence $\pi'$ is isomorphic to $\pi$. By the definition of the moduli functor, this implies $\cK$ and $\cP$ only differ by a line bundle $L$ on $X$ (cf.~\cite[Section 4.1]{huybrechts:geometry-of-moduli-space-of-sheaves}).

    We claim that $L\cong \oh_X$ and hence $\Phi\cong \identity_{\cA}$. Indeed, if we denote $q_i$ the projection of $X\times X$ to the $i$-th factor, then we have
    \[K^x\cong \Phi(K^x)\cong q_{2*}(\cP\otimes q_1^*(K^x\otimes L))\cong i^*(K^x\otimes L)\in \cA\]
    for any point $x\in X$. Then by the assumption \eqref{assumption}, we see $L\cong \oh_X$ and the result follows.
\end{proof}

\begin{remark}
Note that by \cite[Theorem 7.1]{kuznetsov:base-change}, the functor $i\circ i^*\colon \D^b(X)\to \D^b(X)$ is of Fourier--Mukai type.
\end{remark}

\begin{example} \label{example_index_one}
    Let $X$ be a prime Fano threefold of even genus $g\geq 6$. When $g=6$, we furthermore assume that $X$ is ordinary. Then if we take $\cA=\Ku(X)$, \cite[Lemma 5.3]{jacovskis2022brill} and the same argument in \cite[Theorem 5.14]{jacovskis2022brill} show that $i^*$ induces the closed immersion $\pi$ in Lemma \ref{lem_kernel_trivial}. Moreover, $M_X([K^x])$ is a fine moduli space by the Chern character reason (cf.~\cite[Remark 5.2]{jacovskis2022brill}). Finally, it is straightforward to check that
    \[\ch(K^x)\neq \ch(i^*(K^x\otimes L))\]
    for any $L\neq \oh_X\in \Pic(X)=\ZZ\oh_X(1)$. Hence the assumption \eqref{assumption} is also satisfied in this case.
\end{example}

We also need  following two lemmas.

\begin{lemma} \label{lem_bundle}
Let $Y$ and $Y'$ be two del Pezzo threefolds of Picard rank one and degree $2\leq d\leq 3$. Let $\Psi\colon \Ku(Y)\xra{\simeq} \Ku(Y')$ be an exact equivalence that maps classes $\bv$ and $\bw$ to $\bv'$ and $\bw'$, respectively. If $E\in \Ku(Y)$ is a vector bundle, $\Psi(E)$ is also a bundle up to shift.
\end{lemma}

\begin{proof}
    By Remark \ref{rmk_trivial}, we know that after replacing $\Psi$ with $\Psi\circ [n]$ for some $n$, we can assume that $\Psi(i^!\cQ_Y)\cong i'^!\cQ_{Y'}$. Then using Theorem \ref{Brill--Noether reconstruction}, we know that for any $p'\in Y$, there is a point  $p\in Y$ such that $\Psi(i^*\oh_p)\cong i'^*\oh_{p'}$. Then we have
    \[\RHom(\oh_{p'}, \Psi(E))=\RHom(i'^*\oh_{p'}, \Psi(E))=\RHom(i^*\oh_{p}, E)=\RHom(\oh_p, E),\]
where the first and the last equality follows from the fact that $i^*$ and $i'^*$ are left adjoint to $i$ and $i'$, respectively. Then the locally-freeness of $\Psi(E)$ follows from the locally-freeness of $E$ and \cite[Proposition 5.4]{bridgeland:K3-and-elliptic-fibration}.
\end{proof}

\begin{lemma} \label{unique_extend_index_one}
    Let $f,g\colon X\to X'$ be two isomorphisms between index one prime Fano threefolds of genus $g\geq 6$. If $f_{*}|_{\Ku(X)}=g_*|_{\Ku(X)}\colon \Ku(X)\to \Ku(X')$, then $f=g$. Thus the homomorphism 
    \[\Aut(X)\to \Aut(\Ku(X)), \quad f\mapsto f_*|_{\Ku(X)}\]
    is injective.
\end{lemma}
\begin{proof}
We know $f_*$ and $g_*$ maps $\oh_X$ and $\cE_X$ to $\oh_{X'}$ and $\cE_{X'}$ respectively. For any point $p \in X$, we know $f_{*}(\oh_p) = \oh_{f(p)}$ and the same for $g$. Thus we have
\begin{equation*}
f_*(i^*\oh_p)=i'^*\oh_{f(p)} \qquad \text{and } \qquad g_*(i^*\oh_p)=i'^*\oh_{g(p)}. 
\end{equation*}
Since $f_{*}|_{\Ku(Y)}=g_*|_{\Ku(Y)}$, we get $i'^*\oh_{f(p)} = i'^*\oh_{g(p)}$, i.e. $i'^*\oh_{f(p)}$ and $i'^*\oh_{g(p)}$ correspond to the same point in the moduli space $\cM_{\sigma}(\mathcal{K}u(X), [i^*\oh_p])$ by \cite[Theorem 5.14]{jacovskis2022brill}. Thus the embedding in \cite[Theorem 5.14]{jacovskis2022brill} implies that $f(p) = g(p)$ for any point $p \in X$. Since both $X$ and $X'$ are smooth, we get $f=g$.  
\end{proof}

\subsection{Genus $8$}

We start with genus $g=8$ prime Fano threefold case.

\begin{theorem} \label{thm_genus_8}
    Let $X$ be an index one prime Fano threefold of genus $8$. Then for any exact auto-equivalence $\Phi\colon \Ku(X)\xra{\simeq} \Ku(X)$, after composing with the Serre functor $S_{\Ku(X)}$ and shift functor, we have
    \[\Phi(i^!\cE_X)\cong i^!\cE_X.\]
\end{theorem}

\begin{proof}
  We define an equivalence $\Phi'\colon \cA_X\xra{\simeq} \cA_X$ by $\Phi':=\Xi\circ \Phi\circ \Xi^{-1}$. Let $Y$ be the Phaffian cubic threefold associated with $X$. Then by \cite[Theorem 4.7]{kuznetsov:derived-category-fano-threefold} and \cite[Proposition 8.9]{jacovskis2022brill}, there is an equivalence of Fourier--Mukai type $\Theta\colon \cA_X\xra{\simeq} \Ku(Y)$ which maps $\Xi(i^!\cE_X)$ to an instanton bundle on $Y$ up to shift. We define $\Psi\colon \Ku(Y)\xra{\simeq} \Ku(Y)$ by $\Psi:=\Theta\circ \Phi'\circ \Theta^{-1}$. By Lemma \ref{numerical_lemma}, after composing $\Phi$ with the Serre functor $S_{\Ku(X)}$ and shift functor, we can assume that $\Psi$ acts trivially on $\mathcal{N}(\Ku(Y))$. Hence using Lemma \ref{lem_bundle}, $\Psi(\Xi(i^!\cE_X))$ is also a bundle up to shift. Moreover, $\Psi(\Xi(i^!\cE_X))$ is stable with respect to every Serre-invariant stability condition on $\Ku(Y)$ since $\Xi(i^!\cE_X)$ is (cf.~\cite[Theorem 7.6]{liu-zhang:moduli-space-cubic-x14}). By \cite[Theorem 7.6]{liu-zhang:moduli-space-cubic-x14} and the locally-freeness of $\Psi(\Xi(i^!\cE_X))$, we can assume that $\Psi(\Xi(i^!\cE_X))$ is an instanton bundle on $Y$ after composing $\Phi$ with shift functor. Then from \cite[Theorem 4.7]{kuznetsov:derived-category-fano-threefold} and \cite[Proposition 8.9]{jacovskis2022brill}, there is another index one genus $8$ prime Fano threefold $X'$ with an equivalence $\Theta'\colon \cA_{X'}\xra{\simeq} \Ku(Y)$ such that $\Theta'(\Xi(i'^!\cE_{X'}))\cong \Psi(\Xi(i^!\cE_X))$. 
  
  If $\Phi(i^!\cE_X)$ is not isomorphic to $i^!\cE_X$, then $\Psi(\Xi(i^!\cE_X))$ is not isomorphic to $\Xi(i^!\cE_X)$ and hence $X$ is not isomorphic to $X'$ as well by \cite[Theorem 4.7]{kuznetsov:derived-category-fano-threefold}. But we get an equivalence $\Ku(X)\xra{\simeq} \Ku(X')$ which maps $i^!\cE_X$ to $i'^!\cE_{X'}$, which contradicts with \cite[Theorem 1.3]{jacovskis2022brill}. 
\end{proof}

\begin{corollary} \label{cor_aut_genus_8}
     Let $X$ be an index one prime Fano threefold of genus $8$. Then we have
     \[\Aut_{\mathrm{FM}}(\Ku(X))= \langle \Aut(X), S_{\Ku(X)}, [1] \rangle.\]
\end{corollary}

\begin{proof}
By Lemma \ref{unique_extend_index_one}, we have an injection 
\[\Aut(X)\to \Aut_{\mathrm{FM}}(\Ku(X)), \quad f\mapsto f_*|_{\Ku(X)}.\]
Note that $\Aut(X)$ acts trivially on $\cN(\Ku(X))$, while the only elements in $\langle S_{\Ku(Y)},[1]\rangle$ act trivially on $\cN(\Ku(X))$ is of form $[2m]$. Hence $\Aut(X)\cap \langle S_{\Ku(Y)},[1]\rangle=\identity_{\Ku(X)}$ and we see that the induced homomorphism
\[\eta\colon \Aut(X)\to \frac{\Aut_{\mathrm{FM}}(\Ku(X))}{\langle S_{\Ku(Y)},[1]\rangle}\]
is also injective. On the other hand, we have a homomorphism
\[\eta'\colon \frac{\Aut_{\mathrm{FM}}(\Ku(X))}{\langle S_{\Ku(Y)},[1]\rangle}\to \Aut(X)\]
given by Theorem \ref{thm_genus_8} and \cite[Theorem 1.1]{jacovskis2022brill}. Then any element in the kernel of $\eta'$ can be represented by an auto-equivalence $\Phi$ such that $\Phi(i^*\oh_p)\cong i^*\oh_p$ for any point $p\in X$. Then by Lemma \ref{lem_kernel_trivial} and Example \ref{example_index_one}, we have $\Phi\cong\identity_{\Ku(X)}$ and $\eta'$ is injective as well. It is straightforward to check that $\eta'\circ \eta=\identity_{\Aut(X)}$, hence $\eta$ and $\eta'$ are inverse to each other and the result follows.
\end{proof}

The Corollary~\ref{cor_aut_genus_8} has an immediate application on group of automorphism of Fano threefolds. 

\begin{corollary} \label{cor_aut_variety}
    Let $X$ be an index one prime Fano threefold of genus $8$ and $Y$ be the Phaffian cubic threefold associated with $X$. Then we have
    \[\Aut(X)\cong\Aut(Y).\]
\end{corollary}

\begin{proof}
    By \cite[Theorem 4.7]{kuznetsov:derived-category-fano-threefold}, we have an equivalence of Fourier--Mukai type $\Ku(X)\simeq \Ku(Y)$. Hence we get an isomorphism 
    \[s\colon \Aut_{\mathrm{FM}}(\Ku(X))\cong \Aut_{\mathrm{FM}}(\Ku(Y)).\]
    Since an exact equivalence commutes with the Serre functor and shift functor, we have $s(\langle S_{\Ku(X)}, [1] \rangle)=\langle S_{\Ku(Y)}, [1] \rangle$. Note that $S_{\Ku(Y)}=\bO[1]$. Then taking quotient on both sides, by  Corollary \ref{cor_aut_genus_8} we get an induced isomorphism
    \[\Aut(X)\cong\Aut(Y).\]
\end{proof}

\subsection{Genus $6$}

Now let $X$ be an index one genus $6$ prime Fano threefold, which is also called a Gushel--Mukai threefold.

\begin{theorem} \label{thm_genus_6}
    Let $X$ be a general ordinary Gushel--Mukai threefold. Then for any exact auto-equivalence $\Phi\colon \Ku(X)\xra{\simeq} \Ku(X)$, after composing with the Serre functor $S_{\Ku(X)}$ and shift functor, we have
    \[\Phi(i^!\cE_X)\cong i^!\cE_X.\]
\end{theorem}

\begin{proof}
Let $\Psi:=\Xi\circ \Phi\circ \Xi^{-1}\colon \cA_X\simeq \cA_X$ be the induced equivalence. Since $X$ is general, $X$ is not the period dual of itself since the involution on the period domain defined in  \cite[(1.0.14)]{ogrady:double-EPW-period} is non-trivial. Hence by \cite[Theorem 10.3]{jacovskis2021categorical} and its proof, up to shift, $\Psi$ and $\Phi$ act trivially on $\cN(\cA_X)$ and $\cN(\Ku(X))$, respectively. Moreover, by \cite[Theorem 7.13]{jacovskis2021categorical},  $\Psi$ induces an automorphism $\cC_m(X)\cong \cC_m(X)$ of the minimal model of the Fano surface of conics on $X$, which is either the identity map or an involution (cf.~\cite[Corollary 9.2]{debarre2012period}). Then by \cite[Remark 7.4]{jacovskis2021categorical}, up to composing with $S_{\cA_X}$ and shift functor, $\Psi$ induces an automorphism $\cC_m(X)\cong \cC_m(X)$ such that maps the point $\Xi(i^!\cE_X)$ to $\Xi(i^!\cE_X)$, and the result follows.
\end{proof}

\begin{corollary} \label{cor_aut_genus_6}
     Let $X$ be a general ordinary Gushel--Mukai threefold. Then we have
     \[\Aut_{\mathrm{FM}}(\Ku(X))=   \langle \Aut(X), S_{\Ku(X)},  [1]\rangle.\]
\end{corollary}

\begin{proof}
By Lemma \ref{unique_extend_index_one}, we have an injection 
\[\Aut(X)\to \Aut_{\mathrm{FM}}(\Ku(X)), \quad f\mapsto f_*|_{\Ku(X)}.\]
Since $\Aut(X)\cap \langle S_{\Ku(Y)},[1]\rangle=\identity_{\Ku(X)}$, we see that the induced homomorphism
\[\eta\colon \Aut(X)\to \frac{\Aut_{\mathrm{FM}}(\Ku(X))}{\langle S_{\Ku(Y)},[1]\rangle}\]
is injective as well. On the other hand, we have a homomorphism
\[\eta'\colon \frac{\Aut_{\mathrm{FM}}(\Ku(X))}{\langle S_{\Ku(Y)},[1]\rangle}\to \Aut(X)\]
given by Theorem \ref{thm_genus_6} and \cite[Theorem 1.1]{jacovskis2022brill}. Then any element in the kernel of $\eta'$ can be represented by an auto-equivalence $\Phi$ such that $\Phi(i^*\oh_p)\cong i^*\oh_p$ for any point $p\in X$. Then by Lemma \ref{lem_kernel_trivial} and Example \ref{example_index_one}, we have $\Phi\cong\identity_{\Ku(X)}$ and $\eta'$ is also injective. It is straightforward to check that $\eta'\circ \eta=\identity_{\Aut(X)}$, hence $\eta$ and $\eta'$ are inverse to each other and the result follows.
\end{proof}

\begin{remark}
    One can not drop the generality assumption in Corollary \ref{cor_aut_genus_6}. Indeed, in the setting of Corollary \ref{cor_aut_genus_6}, any element in $\Aut_{\mathrm{FM}}(\Ku(X))$ acts on $\mathcal{N}(\Ku(X))$ by $\identity_{\mathcal{N}(\Ku(X))}$ up to sign. However, when the period point of $X$ is a fixed point of the involution on the period domain, then $X$ is the period dual of itself. Then by \cite[Theorem 1.3]{kuznetsov:categorical-cone} there is an equivalence $\cA_X\to \cA_X$, and one can check it maps a numerical class of rank one to a class of rank two.
\end{remark}

Corollary~\ref{cor_aut_genus_6} has a nice application on Kuznetsov's Fano threefold conjecture \cite[Conjecture 3.7]{kuznetsov:derived-category-fano-threefold}. It was disproved in \cite{bayer2022kuznetsov} and \cite{zhang:kuznetsov-conjecture} independently in its most general form. By assuming the Gushel-Mukai threefold is \emph{general}, we present a simple disproof.

\begin{corollary} \label{cor_ku_conj}
    Let $X$ be a general Gushel--Mukai threefold and $Y$ a quartic double solid. Then $\Ku(X)$ is not equivalent to $\Ku(Y)$. 
\end{corollary}

\begin{proof}
    Assume that there is an exact equivalence $\Ku(X)\simeq \Ku(Y)$. By \cite[Theorem 1.3]{li2022derived}, such an equivalence is of Fourier--Mukai type. Hence it induces an isomorphism of the numerical Grothendieck groups and 
    \[\Aut_{\mathrm{FM}}(\Ku(X))\cong \Aut_{\mathrm{FM}}(\Ku(Y)).\]
However, Corollary \ref{cor_aut_genus_6} shows that any element in $\Aut_{\mathrm{FM}}(\Ku(X))$ acts on $\cN(\Ku(X))$ by $\identity_{\mathcal{N}(\Ku(X))}$ up to sign, while the action of the rotation functor $\bO\colon \Ku(Y)\to \Ku(Y)$ on $\mathcal{N}(\Ku(Y))$ is not. Thus we get a contradiction.
\end{proof}

\begin{appendix}

\section{Moduli space of instanton sheaves on quartic double solids}\label{section_classification_instanton_sheaves_d=2}
In this section, we fix $Y$ to be a quartic double solid and study the moduli space $M_Y(2,0,2)$ of semistable sheaves of rank two, $c_1=0, c_2=2,c_3=0$ and the Bridgeland moduli space $\mathcal{M}_{\sigma}(\Ku(Y),2\bv)$ of semistable objects of class $2\bv$ in the Kuznetsov component $\Ku(Y)$. 

\subsection{Classifications}

As is shown in Proposition~\ref{prop_bridgeland_Y2} that up to shift, the $\sigma$-stable objects of class $2\bv$ in the Kuznetsov component $\Ku(Y)$ of a quartic double solid $Y$ is either a two-term complex $i^!\cQ_Y$ or a Gieseker semistable sheaf of rank two, $c_1=0, c_2=2$ and $c_3=0$. Denote by $E$ such a sheaf. It is clear that $H^1(Y,E(-1))=0$ since $E\in\Ku(Y)$. Then it is an instanton sheaf in the sense of \cite[Definition 6.2]{liu-zhang:moduli-space-cubic-x14}. To study the geometric structure and properties of the Bridgeland moduli space $\mathcal{M}_{\sigma}(\Ku(Y),2\bv)$, first we classify sheaves in the moduli space $M_Y^{inst}(2,0,2)$ of instanton sheaves on $Y$.

\begin{proposition}
\label{classification_instanton_sheaves}
Let $E\in M_Y(2,0,2)$. Then $E\notin \Ku(Y)$ if and only if it is a locally free sheaf fitting into an exact sequence
\begin{equation} \label{non_bundle_seq}
    0\to \oh_Y(-1)\to \cQ_Y\to E\to 0.
\end{equation}

If $E\in\Ku(Y)$, then $E$ is \begin{enumerate}
    \item either a strictly Gieseker-semistable sheaf, which is an extension of two ideal sheaves of lines,
    
    \item or a non-locally free sheaf fitting into a short exact sequence
    $$0\rightarrow E\rightarrow\oh_Y^{\oplus 2}\rightarrow Q\rightarrow 0,$$ where $Q=\theta_C(1)$ is the theta characteristic of a smooth conic $C$, or $Q$ is a sheaf on  a codimension two linear section $C$ of $Y$ given by
    \[0\to \oh_C\to Q\to R\to 0,\]
    where $R$ is a zero-dimensional sheaf on $C$ of length two,
    
    \item or a $\mu_H$-stable vector bundle that $E(1)$ is globally generated and fits into the short exact sequence
    $$0\rightarrow\oh_Y(-H)\rightarrow E\rightarrow I_D(H)\rightarrow 0,$$
    where $D$ is the zero locus of a generic section of $H^0(E(1))$, which is a degree $4$ smooth elliptic curve.
\end{enumerate}
\end{proposition}

\begin{proof}
If $E$ is strictly Gieseker-semistable, then the result follows from applying Lemma \ref{no_wall_lem} to Jordan--H\"older factors. If $E$ is Gieseker-stable, the result follows from Proposition \ref{classification_non_reflexive_sheaf_quartic_double_solid}, Lemma \ref{instanton_bundles_quartic_double_solid} and Proposition \ref{classfication_semistable_sheaves} below.
\end{proof}

In the following, we are going to prove the results used in Proposition \ref{classification_instanton_sheaves}. We only need to consider Gieseker-stable one.

\begin{lemma}\label{lem_slope_stable_instanton}
Let $E$ be a $\mu_H$-semistable reflexive sheaf of rank two, $c_1(E)=0$ and $H^0(E)$=0. Then $E$ is $\mu_H$-stable. 
\end{lemma}
\begin{proof}
If not, its Jordan--H\"older filtration with respect to $\mu_H$-stability has two terms $E_1 \hookrightarrow E \twoheadrightarrow E_2$ where $E_1$ and $E_2$ are $\mu_H$-stable sheaves with $\ch_{\leq 1}(E_i) = (1, 0)$. Then $E_1^{\vee \vee}=\oh_Y$ since $\Pic(Y)=\mathbb{Z}H$. Then taking the double dual, we get a non-zero map $\oh_Y\to E^{\vee \vee}=E$, which contradicts $H^0(E)=0$.
\end{proof}

\begin{lemma}\label{lem-y-2-no stable reflexive sheaf}
There is no $\mu_H$-semistable reflexive sheaf  $E$ of classes

\begin{enumerate}
    \item $\ch(E)=(2, 0, -\frac{1}{2}H^2, \alpha_1H^3)$,
    
    \item $\ch(E)=(2, 0, -H^2, \alpha_2H^3)$ where $\alpha_2 \neq 0$, and
    
    \item $\ch(E)=(2, 0, 0, \alpha_3H^3)$ where $\alpha_3 \neq 0$.
\end{enumerate}

Moreover, if $\ch(E)=2\ch(\oh_Y)$, then $E\cong \oh_Y^{\oplus 2}$.   
\end{lemma}

\begin{proof}
Note that being rank two and reflexive implies $c_3(E)\geq 0$ by \cite[Proposition 2.6]{hartshorne:stable-reflexive-sheaves}, hence $\alpha_i\geq 0$. Then the case (2) follows from Lemma \ref{lem_slope_stable_instanton} and Lemma \ref{no_wall_lem}. And case (3) follows from the same argument as in \cite[Proposition 4.20]{feyz:desing}. 

So we only need to prove (1). Assume for a contradiction that $E$ is a $\mu_H$-semistable reflexive sheaf of classes $\ch(E)=(2, 0, -\frac{1}{2}H^2, \alpha_1H^3)$ with $\alpha_1\geq 0$. We know that there is no wall for $E$ crossing the vertical line $b= - \frac{1}{2}$, so $\Hom(E,\oh_Y(-2)[1])=H^2(E)=0$. And by $\mu_H$-semistability, we get $H^0(E)=H^3(E)=0$, which implies 
\begin{equation*}
    2\alpha_1 +1 = \chi(\oh_Y, E) = -\hom(\oh_Y, E[1]) \leq 0 
\end{equation*}
which makes a contradiction.  If $\ch(E)=2\ch(\oh_Y)$, then 
\begin{equation*}
    \hom(\oh_Y, E) -\hom(\oh_Y, E[1]) = 2. 
\end{equation*}
Thus Jordan--H\"older factors of $E$ with respect to the $\mu_H$-stability are all $\oh_Y$, and the result follows.
\end{proof}

\begin{proposition}\label{classification_non_reflexive_sheaf_quartic_double_solid}
Let $E\in \Ku(Y)$ be a non-reflexive Gieseker-stable sheaf of character $2\bv$, then $E$ fits into a short exact sequence
    $$0\rightarrow E\rightarrow\oh_Y^{\oplus 2}\rightarrow Q\rightarrow 0,$$ where $Q=\theta_C(1)$ is the theta characteristic of a smooth conic $C$, or $Q$ is a sheaf on  a codimension two linear section $C$ of $Y$ given by
    \[0\to \oh_C\to Q\to R\to 0,\]
    where $R$ is a zero-dimensional sheaf on $C$ of length two.

\end{proposition}

\begin{proof}
Taking the reflexive hull of $E$ gives the exact sequence 
\begin{equation}\label{reflexive hull}
    E \rightarrow E^{\vee \vee} \rightarrow Q
\end{equation}
where $E^{\vee \vee}$ is a reflexive $\mu_H$-semistable sheaf and $Q$ is a torsion sheaf supported in dimension at most one. Applying Lemma \ref{lem-y-2-no stable reflexive sheaf} to the exact sequence \eqref{reflexive hull} shows that $E^{\vee \vee} = \oh_Y^{\oplus 2}$ and $Q$ is a torsion sheaf of class $\ch(Q) = (0, 0, H^2, 0)$. Since $E\in \Ku(Y)$ and $\RHom(\oh_Y(1), E^{\vee\vee})=0$, we know that $H^0(Q(-1))=0$. Then the result follows from Lemma \ref{lem_torsion_sheaf}.
\end{proof}

\begin{lemma} \label{lem_line}
Let $Z\subset Y$ be a one-dimensional closed subscheme with $H.Z=1$. If $Z$ is pure, then $Z$ is a line.
\end{lemma}

\begin{proof}
Since $H.Z=1$, we see $Z$ is irreducible since it is pure. Then $H.Z_{red}=1$, which implies that $\ker(\oh_Z\to \oh_{Z_{red}})$ is zero-dimensional. But this is impossible since $\oh_Z$ is pure. Hence $Z$ is integral, and $\pi(Z)\subset \PP^3$ is also an integral subscheme of degree one, which is a line. Since $Z\subset \pi^{-1}(\pi(Z))$ is an irreducible component, $\pi^{-1}(\pi(Z))$ is reducible. Hence $\pi^{-1}(\pi(Z))$ is union of two lines on $Y$, which implies that $Z$ is a line. 
\end{proof}

\begin{lemma} \label{lem_elliptic}
Let $C\subset Y$ be a pure one-dimensional closed subscheme with $H.C=2$ and $\chi(\oh_C)=0$. Then $C$ is irreducible and is the intersection of two hyperplane sections of $Y$. Moreover,  $C=\pi^{-1}(\pi(C))$ and $\pi(C)\subset \PP^3$ is a line.
\end{lemma}

\begin{proof}
If $C$ is reducible, then from $H.C=2$, each component is pure-dimensional with degree one, which is a line by Lemma \ref{lem_line}. Then these two components are either disjoint which implies $\chi(\oh_C)=2$, or intersect at a single point, which gives $\chi(\oh_C)=1$. Hence $C$ is irreducible.

If $H.C_{red}=2$, then $C$ is reduced since $\oh_C$ is pure. Then $\pi(C)$ is also integral. If the degree of $\pi(C)$ is two, then $C\cong \pi(C)$ which contradicts \cite[Corollary 1.38]{sanna:instanton-on-Y5} since $\chi(\oh_C)=0$. Thus $\pi(C)$ is a line, and $C\subset \pi^{-1}(\pi(C))$. Since $\pi^{-1}(\pi(C))$ is also a degree two curve of genus one, we have $C=\pi^{-1}(\pi(C))$.

If $H.C_{red}=1$, then $C_{red}=l$ is a line, and we have an exact sequence $0\to \oh_l(-2)\to \oh_C\to \oh_l\to 0$. Thus $h^0(\oh_C(1))=2$. Therefore, we have $h^0(\cI_C(1))\geq 2$ and $C$ is contained in two different hyperplane sections $S,S'$ of $Y$. This implies that $C\subset S\cap S'$. Since $S\cap S'$ is also a degree two curve of genus one, we have $C=S\cap S'=\pi^{-1}(l)$. 
\end{proof}

\begin{lemma} \label{lem_torsion_sheaf}
    Let $Q$ be a coherent sheaf on $Y$ of class $\ch(Q)=(0,0,H^2,0)$ with $H^0(Q(-1))=0$ on $Y$. Then $Q$ is either 
    
    \begin{enumerate}
        \item an extension of structure sheaves of lines on $Y$,
        
        \item $Q=\theta_C(1)$, where $\theta_C$ is the theta characteristic of a smooth conic $C$ on $Y$, or
        
        \item $Q$ is a sheaf on  a codimension two linear section $C$ of $Y$ given by
    \[0\to \oh_C\to Q\to R\to 0,\]
    where $R$ is a zero-dimensional sheaf on $C$ of length two.
    \end{enumerate}
    
\end{lemma}

\begin{proof}
Since $\chi(Q)=2$, we have $H^0(Q)\neq 0$. Let $s\colon \oh_Y\to Q$ be a non-zero map. Then $\im(s)=\oh_Z$, where $Z\subset Y$ is a subscheme. Since $H^0(Q(-1))=0$, we see $H^0(\oh_Z(-1))=0$ and hence $Z$ is pure-dimensional. Note that if  $H.Z_{red}=H.Z$, then the kernel of $\oh_Z\to \oh_{Z_{red}}$ is zero-dimensional, which implies $Z=Z_{red}$ by $H^0(\oh_Z(-1))=0$. Let $R:=\cok(s)$.

\begin{itemize}
    \item Assume that $H.Z=1$. Then by Lemma \ref{lem_line}, $Z$ is a line and hence $H^1(\oh_Z(-1))=0$. Thus  $\ch(R)=(0,0,\frac{H^2}{2},0)$ and $H^0(R(-1))=0$. We claim that $R$ is also the structure sheaf of a line. Indeed, by  $\chi(R)=1$, we have a non-zero map $s'\colon \oh_Y\to R$. By the same argument above, we see $H^0(\im(s')(-1))=0$ and hence $\im(s')$ is the structure sheaf of line by Lemma \ref{lem_line}. By the reason of Chern characters, we see $\im(s')=R$ and the result follows.
    
    \item Assume that $H.Z=2$. First, we assume that $R=0$, hence $\oh_Z=Q$. If $H.Z_{red}=1$, then $Z_{red}$ is a line by Lemma \ref{lem_line}. Thus $\ker(\oh_Z\to \oh_{Z_{red}})$ satisfies properties of $R$ in the first case. The same argument shows that $\ker(\oh_Z\to \oh_{Z_{red}})$ is also the structure sheaf of a line. If $H.Z_{red}=2$, then the kernel of $\oh_Z\to \oh_{Z_{red}}$ is zero-dimensional, which implies $Z=Z_{red}$ by $H^0(\oh_Z(-1))=0$. Note that $Z$ is reducible, otherwise we have $h^0(\oh_Z)=1$, which contradicts $\chi(\oh_Z)=\chi(Q)=2$. Hence by Lemma \ref{lem_line}, each of the  irreducible components of $Z$ is a line. Since $\ch(\oh_Z)=(0,0,H^2,0)$, we see $Z$ is an extension of structure sheaves of two lines.
    
    Now we assume that $R\neq 0$. The same argument as in \cite[Lemma 3.3]{druel:instanton-sheaves-cubic-threefold} shows that $Q$ is a $\oh_Z$-module.

    \begin{itemize}
        \item If $Z$ is reducible, then each component of $Z$ has degree one. Hence $H.Z_{red}=H.Z=2$. This implies $Z=Z_{red}$ as above since $H^0(\oh_Z(-1))=0$. By Lemma \ref{lem_line}, $Z$ is a union of two lines. And from $R\neq 0$, we see these two lines intersect with each other. In other words, $Z$ is a reducible conic. Now since $Z$ is a conic, the same argument as in \cite[Lemma 3.3]{druel:instanton-sheaves-cubic-threefold} shows that $Z$ is a smooth conic and $Q=\theta_Z(1)$.
        
        \item If $Z$ is irreducible and $H.Z_{red}=2$, then we also have $Z=Z_{red}$, which implies that $h^0(\oh_Z)=1$ and $\chi(\oh_Z)\leq 1$. From \cite[Lemma 4.3]{liu-ruan:cast-bound}, we see $0\leq \chi(\oh_Z)\leq 1$. When $\chi(\oh_Z)=1$, $Z$ is also a conic, hence the same argument as in \cite[Lemma 3.3]{druel:instanton-sheaves-cubic-threefold} shows that $Z$ is a smooth conic and $Q=\theta_Z(1)$. When $\chi(\oh_Z)=0$, $Z$ is the intersection of two hyperplane sections by Lemma \ref{lem_elliptic} and hence $\mathrm{length}(R)=2$.  
        
        \item If $Z$ is irreducible and $H.Z_{red}=1$, then $Z_{red}$ is a line by Lemma \ref{lem_line}. Therefore, we have an exact sequence $0\to \oh_l(-n)\to \oh_Z\to \oh_l\to 0$, where $n\in \ZZ_{>0}$. In particular, we have $h^0(\oh_Z)=1$ which implies $\chi(\oh_Z)\leq 1$. From \cite[Lemma 4.3]{liu-ruan:cast-bound}, we see $0\leq \chi(\oh_Z)\leq 1$. When $\chi(\oh_Z)=1$, $Z$ is a conic. By \cite[Lemma 3.3]{druel:instanton-sheaves-cubic-threefold}, $Z$ is smooth and contradicts $H.Z_{red}=1$. When $\chi(\oh_Z)=0$, we have $\mathrm{length}(R)=2$ and the result follows.
    \end{itemize}

\end{itemize}

\end{proof}

Now assume $E$ is a Gieseker-semistable reflexive sheaf of class $2\bv$. It follows from \cite[Proposition 2.6]{hartshorne:stable-reflexive-sheaves} that $E$ is a locally free sheaf and it is a slope stable locally free sheaf by Lemma~\ref{lem_slope_stable_instanton}.

\begin{lemma}
\label{instanton_bundles_quartic_double_solid}
Let $E\in M_Y(2,0,2)$ be a bundle with $E\in \Ku(Y)$, then $E(1)$ is globally generated and it fits into the short exact sequence 
$$0\rightarrow\oh_Y(-H)\rightarrow E\rightarrow I_D(H)\rightarrow 0,$$
where $D$ is a degree 4 smooth elliptic curve as the zero locus of a general section of $E(1)$. 
\end{lemma}

\begin{proof}
Note that $H^3(E(-2))=H^0(E^{\vee})=H^0(E)=0$ since $E^{\vee}\cong E$. Then from $E\in \Ku(Y)$, we see $H^i(E(1-i))=0$ for $i>0$, thus $E(1)$ is globally generated by Castelnuovo--Mumford regularity. Then the zero locus of a generic section of $E(1)$ is smooth. The remaining statement follows from the Serre correspondence.

\end{proof}

On the other hand, the next proposition characterizes a semistable sheaf of rank two, $c_1=0,c_2=2,c_3=0$, which is not in the Kuznetsov component $\Ku(Y)$. 

\begin{proposition}
\label{classfication_semistable_sheaves}
Let $E\in M_Y(2,0,2)$, then $E\not\in\Ku(Y)$ if and only if $E$ is locally free and fits into an exact sequence of the form \eqref{non_bundle_seq}. 
\end{proposition}

\begin{proof}
By Lemma \ref{no_wall_lem}, we have $\RHom(\oh_Y, E)=0$. Note that $H^0(E(-1))=H^3(E(-1))=0$ by Serre duality and stability. Thus from $\chi(E(-1))=0$, we see $E\not\in\Ku(Y)$ if and only if $H^1(E(-1))=H^2(E(-1))\neq 0$.

First we assume that $E$ fits into an exact sequence as above. Since $\cQ_Y$ is a $\mu_H$-stable vector bundle by Lemma \ref{lem-Q-Y}, it is clear that there is a non-zero morphism $E\rightarrow\oh_Y(-1)[1]$, then $\mathrm{Hom}(\oh_Y,E(-1)[2])=\mathrm{Hom}(E,\oh_Y(-1)[1])\neq 0$ by Serre duality. 

Now we assume that $H^1(E(-1))\neq 0$. Applying $\mathrm{Hom}(-,E)$ to \eqref{eq. Q} and using $\RHom(\oh_Y,E)=0$, we have  $\mathrm{Hom}(\cQ_Y,E)=H^1(E(-1))\neq 0$. Let $\pi\neq 0\in\mathrm{Hom}(\cQ_Y,E)$. We claim that $\pi$ is surjective and $\ker(\pi)\cong\oh_Y(-H)$. Indeed, if $\rk(\im(\pi))=2$, then $\ker(\pi)$ is a reflexive torsion-free sheaf of rank one since $\cQ_Y$ is locally free and $E$ is torsion-free. From the smoothness of $Y$, we know that $\ker(\pi)$ is a line bundle. By the $\mu_H$-semistability of $\cQ_Y$ and $E$, we know that $c_1(\im(\pi))=0$, i.e.~$c_1(\ker(\pi))=-H$ and $\ker(\pi)=\oh_Y(-H)$. Therefore, we only need to show that $\rk(\im(\pi))\neq 1$. 

To this end, we assume that $\rk(\im(\pi))=1$. Then by the $\mu_H$-semistability, we have $c_1(\im(\pi))=0$. Thus $\ch_{\leq 2}(\im(\pi))=(1,0,-\frac{a}{2}H^2)$ for $a\geq 1$. But we also know that Gieseker-stable implies 2-Gieseker-stable for $E$ by Lemma \ref{Gstability_lvls_coincide}. Thus the only possible case is $a\geq 2$. Then $\ch_{\leq 2}(\ker(\pi))=(2,-H,\frac{a-1}{2}H^2)$ with $a-1\geq 1$. But from the stability of $\cQ_Y$, we know that $\ker(\pi)$ is also $\mu_H$-stable. This contradicts \cite[Proposition 3.2]{li:fano-picard-number-one}.
Then the claim is proved. 

The only part we remain to show is the local freeness of $E$. Assume that $E$ fits into \eqref{non_bundle_seq}. If $E$ is not reflexive, then as in Proposition \ref{classification_non_reflexive_sheaf_quartic_double_solid}, we get $E^{\vee \vee}=\oh^{\oplus 2}_Y$. However, using \eqref{non_bundle_seq} we can compute that $\Hom(E, \oh_Y)=0$, which makes a contradiction. Thus $E$ is reflexive, and by $\rk(E)=2$ and $c_3(E)=0$, we see $E$ is locally free.
\end{proof}


\subsection{Singularities of moduli spaces}
In this section, we study singularities of stable moduli spaces $M^s_Y(2,0,2)$ and $\mathcal{M}^s_{\sigma}(\Ku(Y),2\bv)$. 


\begin{lemma} \label{lem_QY_coho}
    We have
    
    \begin{enumerate}
        \item $\RHom(\oh_Y(1), \cQ_Y)=\mathbb{C}[-1]$,
        
        \item $\RHom(\cQ_Y,\cQ_Y)=\mathbb{C}$, and 
        
        \item $\RHom(\oh_Y(-1), \cQ_Y)=\mathbb{C}^6\oplus \mathbb{C}[-1]$.
    \end{enumerate}

\end{lemma}

\begin{proof}
(1) follows from applying $\Hom(\oh_Y(1),-)$ to \eqref{eq. Q}. Note that $\RHom(\oh_Y,\cQ_Y)=0$, then (2) follows from (1) and applying $\Hom(-,\cQ_Y)$ to \eqref{eq. Q}.

For (3), recall that $\pi_*\oh_Y=\oh_{\PP^3}\oplus \oh_{\PP^3}(-2)$. Since $\cQ_Y=\pi^*\Omega_{\PP^3}(1)$, we have $H^0(\cQ_Y(1))=H^0(\Omega_{\PP^3}(2)\oplus \Omega_{\PP^3})$. Thus $h^0(\cQ_Y(1))=6$ by the standard result on $\PP^3$. And by \eqref{eq. Q}, we get $H^i(\cQ_Y(1))=0$ for $i>1$. Then the result follows from $\chi(\cQ_Y(1))=5$.
\end{proof}

\begin{lemma} \label{lem_ext_gluing}
    We have $\RHom(i^!\cQ_Y,i^!\cQ_Y)=\mathbb{C}\oplus \mathbb{C}^6[-1]\oplus \mathbb{C}[-2]$.
\end{lemma}

\begin{proof}
By the adjunction of $i$ and $i^!$, we have $\RHom(i^!\cQ_Y,\cQ_Y)=\RHom(i^!\cQ_Y,i^!\cQ_Y)$. Then the result follows from applying $\Hom(-,\cQ_Y)$ to \eqref{exact. i!QY} and using Lemma \ref{lem_QY_coho}.
\end{proof}

\begin{lemma}
\label{smoothness_semistablesheaves_notin_Ku}
Let $E\in M_Y(2,0,2)$ and $E\not\in\Ku(Y)$, then $\RHom(E,E)=\mathbb{C}\oplus \mathbb{C}^6[-1]\oplus \mathbb{C}[-2]$.
\end{lemma}

\begin{proof}
Since $E$ is stable, we have $\Hom(E,E)=\mathbb{C}$. And by stability we get $\Ext^3(E,E)=\Hom(E,E(-2))=0$. To prove the statement, we only need to show $\ext^2(E,E)=1$.

We compute $\mathrm{Ext}^2(E,E)$ via the standard spectral sequence (see e.g.~\cite[Lemma 2.27]{pirozhkov2020admissible}) and \eqref{non_bundle_seq}. We have a spectral sequence with the first page $$E_1^{p,q}=
\begin{cases}
	\mathrm{Ext}^q(\cQ_Y,\oh_Y(-1)), & p=-1 \\
	\mathrm{Ext}^q(\oh_Y(-1),\oh_Y(-1))\bigoplus\mathrm{Ext}^q(\cQ_Y,\cQ_Y), & p=0\\
	\mathrm{Ext}^q(\oh_Y(-1),\cQ_Y) , &  p=1 \\
    0 , &  p\leq -2, p\geq 2 
	\end{cases}$$
and convergent to $\Ext^{p+q}(E,E)$. Then using Lemma \ref{lem_QY_coho}, we obtain $\ext^2(E,E)=1$ and the result follows.	

\begin{remark}\label{locus_ni_Kuz}
Denote by $M^{ni}$ the locus of Gieseker-semistable sheaves $E\in M_Y(2,0,2)$ but $E\not\in\Ku(Y)$. By Lemma~\ref{smoothness_semistablesheaves_notin_Ku} the locus $M^{ni}$ is everywhere singular. But according to Lemma~\ref{lem_QY_coho} and \eqref{non_bundle_seq}, the reduction $M^{ni}_{red}$ of such locus is isomorphic to $\mathbb{P}\mathrm{Hom}(\oh_Y(-1),\cQ_Y)\cong\mathbb{P}^5$. In the following section~\ref{Bridgeland_mod_space_contraction}, we show it is contracted to a singular point in the Bridgeland moduli space $\mathcal{M}_{\sigma}(\Ku(Y),2\bv)$ via projection functor $i^*$. 
\end{remark}

\end{proof}




\subsection{Bridgeland moduli space}\label{Bridgeland_mod_space_contraction}

Finally, we study the relation between $M_Y(2,0,2)$ and $\mathcal{M}_{\sigma}(\Ku(Y),2\bv)$.

\begin{lemma} \label{lem_proj_equal_gluing}
    Let $E\in M_Y(2,0,2)$ such that  $E\not\in\Ku(Y)$. Then $i^*E\cong i^!\cQ_Y$.
\end{lemma}

\begin{proof}
Note that $i^*\oh_Y(-1)[1]\cong i^!\cQ_Y$. Then applying $i^*$ to \eqref{non_bundle_seq}, we only need to show $i^*\cQ_Y\cong 0$. By definition, we get an exact triangle
\[\oh_Y(1)[-1]\xra{s} \cQ_Y\to \bL_{\oh_Y(1)}\cQ_Y,\]
where $s$ is the unique non-zero map in $\Hom(\oh_Y(1)[-1],\cQ_Y)$ up to scalar. We claim that the induced map
\[\bL_{\oh_Y}(s)\colon \bL_{\oh_Y} \oh_Y(1)[-1]\to \bL_{\oh_Y}\cQ_Y\]
is an isomorphism, which implies $i^*\cQ_Y\cong 0$. Indeed, we have an exact triangle 
\[\oh_Y(1)[-1]\xra{s} \cQ_Y\to \oh_Y^{\oplus 4}\]
which comes from \eqref{eq. Q}. Since $\bL_{\oh_Y}\oh_Y\cong 0$, the claim follows.
\end{proof}

\begin{proposition}
Let $Y$ be a quartic double solid and $\sigma$ be a Serre-invariant stability condition on $\Ku(Y)$. Then the projection functor $i^*$ induces a morphism
\[p\colon M_Y(2,0,2)\twoheadrightarrow \cM_{\sigma}(\Ku(Y), 2\bv)\]
such that contracts $M^{ni}$ to a singular point represented by $i^!\cQ_Y$, and is an isomorphism outside $M^{ni}$.
\end{proposition}

\begin{proof}
Note that up to shift, all strictly $\sigma$-semistable objects are extensions of two ideal sheaves of lines, which are exactly all strictly Gieseker-semistable of class $2\bv$ by Theorem \ref{classification_instanton_sheaves}. Thus $i^*$ affects nothing on the strictly Gieseker-semistable locus. From Lemma \ref{lem-serre-stability-d-2}, we also know that $i^!\cQ_Y$ is $\sigma$-stable. Then the result follows from Theorem \ref{classification_instanton_sheaves},  Lemma \ref{lem_proj_equal_gluing} and Lemma \ref{lem_ext_gluing}.
\end{proof}

\begin{remark}
 It looks plausible that for generic quartic double solids $Y$, the only singular point in $\mathcal{M}_{\sigma}(\Ku(Y),2\bv)$ would be the point $[i^!\cQ_Y]$. As a result, up to composing with $\bO$ and $[1]$, any exact equivalence $\Phi:\Ku(Y)\simeq\Ku(Y')$ would send $i^!\cQ_Y$ to $i'^!\cQ_{Y'}$, then by Theorem~\ref{Brill--Noether reconstruction}, we can get an alternative proof of categorical Torelli theorem for \emph{generic} quartic double solids. 
\end{remark}

\end{appendix}

\bibliography{mybib}

\bibliographystyle{alpha}

\end{document}